\documentclass[sn-mathphys-num]{sn-jnl}% Math and Physical Sciences Numbered Reference Style 
\usepackage{color}
\usepackage{latexsym}
\usepackage{graphicx}
\usepackage{caption}
\usepackage{subcaption}
\usepackage{enumitem}
\usepackage{bbm}
\usepackage{mathrsfs}

\usepackage{amsmath}
\numberwithin{equation}{section}
\usepackage{amsthm}
\usepackage{amssymb}

\definecolor{darkblue}{rgb}{0.0,0,0.5}
\usepackage{algorithm,algorithmic}
\usepackage{caption}
\newcommand{\R}{\mathbb{R}}
\newcommand{\E}{\mathbb{E}}

\renewcommand{\H}{\mathcal{H}}
\newcommand{\B}{\mathbb{B}}
\newcommand{\rad}{\varepsilon}

\newcommand{\norm}[1]{\left\|{#1}\right\|} % A norm with 1 argument

\newcommand{\defeq}{:=}

\newcommand{\normal}{\mathsf{N}}  % Normal distribution
\newcommand{\N}{\mathbb{N}}
\renewcommand{\P}{\mathbb{P}}

\newtheorem{lemma}{Lemma}
\newtheorem{corollary}{Corollary}
\newtheorem{assumption}{Assumption}[section]
\newtheorem{innercustomgeneric}{\customgenericname}
\providecommand{\customgenericname}{}

\newcommand{\grad}{\nabla}

\newcommand{\dist}{{\rm dist}}
\renewcommand{\theassumption}{\Alph{assumption}}

\newcommand{\D}{\mathbb{D}}
\renewcommand{\H}{\mathbb{H}}
\newcommand{\one}{\mathbf{1}}
\newcommand{\eps}{\epsilon}

\newcommand{\sm}{^{sm}}
\newcommand{\ns}{^{ns}}

\newcommand{\loss}{h}
\newcommand{\hh}{\psi}
\newcommand{\openset}{\mathcal{O}}
\theoremstyle{thmstyleone}%
\newtheorem{theorem}{Theorem}%  meant for continuous numbers
\newtheorem{proposition}[theorem]{Proposition}% 

\theoremstyle{thmstyletwo}%
\newtheorem{example}{Example}%
\newtheorem{remark}{Remark}%

\theoremstyle{thmstylethree}%
\newtheorem{definition}{Definition}%

\begin{document}

\title[On the Uniform Convergence of Subdifferentials in Stochastic Optimization and Learning]{On the Uniform Convergence of Subdifferentials in Stochastic Optimization and Learning}


\author*{\fnm{Feng} \sur{Ruan}}\email{fengruan@northwestern.edu}

\affil*{\orgdiv{Department of Statistics and Data Science}, \orgname{Northwestern University}}%, \orgaddress{\city{Evanston}, \postcode{60202}, \state{Illinois}, \country{United States of America}}}


\abstract{
\phantom{.....}
%In nonsmooth, nonconvex stochastic optimization, understanding the uniform convergence of subdifferential mappings is crucial for analyzing stationary points of sample average approximations of risk as they approach the population risk.
%However, characterizing this convergence remains challenging, particularly in the absence of smoothness or strong distributional assumptions.
%This work provides a general analytical reduction: for weakly convex stochastic objectives, we show that a uniform bound on the convergence of selected subgradients—chosen arbitrarily from subdifferential sets—yields a corresponding uniform bound on the Hausdorff distance between the subdifferentials themselves over any open subsets of the domain. This deterministic result enables the study of subdifferential set convergence via simpler vector-valued subgradient convergence. We apply this reduction to derive tight uniform convergence rates for subdifferential sets in stochastic convex-composite optimization, without relying on differentiability assumption on the population risk. These guarantees clarify the landscape of nonsmooth empirical objectives in finite samples and provide new insight into the geometry of optimization problems encountered in robust statistics and related applications.
We investigate the uniform convergence of subdifferential mappings from empirical risk to population risk in nonsmooth, nonconvex stochastic optimization. This question is key to understanding how empirical stationary points approximate population ones, yet characterizing this convergence remains a fundamental challenge due to the set-valued and nonsmooth nature of subdifferentials. This work establishes a general reduction principle: for weakly convex stochastic objectives,  over any open subset of the domain, we show that a uniform bound on the convergence of selected subgradients—chosen arbitrarily from subdifferential sets—yields a corresponding uniform bound on the Hausdorff distance between the subdifferentials. This deterministic result reduces the study of set-valued subdifferential convergence to simpler vector-valued subgradient convergence. We apply this reduction to derive sharp uniform convergence rates for subdifferential mappings in stochastic convex-composite optimization, without relying on differentiability assumptions on the population risk. These guarantees clarify the landscape of nonsmooth empirical objectives and offer new insight into the geometry of optimization problems arising in robust statistics and related applications.
}




\maketitle

\section{Introduction} 
\indent\indent
%At the core of nonsmooth, nonconvex stochastic optimization lies a fundamental yet underexplored challenge---characterizing 
%the statistical uniform convergence behavior of set-valued subdifferential mappings~\cite{ShapiroXu07}. This convergence is 
%essential for understanding how the empirical risk, derived from sample average approximation, approximates the true population risk. 
%While results on the uniform convergence of objective values provide important 
%insights, they do not directly characterize the convergence behaviors of stationary points formulated via subdifferentials. 
%Analyzing the uniform convergence of these set-valued subdifferential mappings of the empirical risk is crucial for a comprehensive understanding 
%of statistical estimation and the behaviors of numerical algorithms in stochastic programming~\cite[Section 7.2]{ShapiroDeRu21}. 
%Despite its importance, this characterization of subdifferential set convergence remains elusive, warranting further investigation
%(see Section~\ref{sec:related-works} for related works).
At the core of nonsmooth, nonconvex stochastic optimization lies a fundamental challenge---characterizing 
the statistical uniform convergence behavior of set-valued subdifferential mappings~\cite{ShapiroXu07}. 
This characterization is not only key to understanding the convergence of stationary points defined by subdifferentials but also essential for analyzing statistical estimation and the performance of numerical algorithms.
%While existing results on the uniform convergence of objective values offer valuable insights, they often fall short of addressing the more intricate convergence behaviors of stationary points characterized by subdifferentials.
Consequently, understanding the uniform convergence of these subdifferential mappings for the empirical risk plays a pivotal role in stochastic programming~\cite[Section 7.2]{ShapiroDeRu21}. 
Nevertheless, existing approaches to this problem often require the population risk to be continuously differentiable. Such smoothness assumptions can obscure the geometric structure of nonsmooth models and limit the applicability of theoretical guarantees in real-world problems (see Section~\ref{sec:related-works} for detailed discussions on related works).

In this work, we provide a new perspective on this problem by establishing a general deterministic principle that links the uniform convergence of subgradient selections to the uniform convergence of subdifferential sets. Specifically, we show that for any pair of weakly convex functions, the uniform distance between their subdifferential mappings (measured in Hausdorff distance) is bounded above by the uniform distance between any pair of subgradient selections over any open domain  (Theorem~\ref{theorem:bound-HD-by-MS}). This result bypasses any smoothness assumptions and allows us to characterize subdifferential set convergence by analyzing simpler vector-valued subgradient mappings. Importantly, it enables the analysis of the uniform convergence of subdifferentials in stochastic optimization settings via the uniform convergence of subgradients, leading to sharp statistical rates. % (Theorem~\ref{theorem:main-stochastic-minimization-result}).

%Towards addressing this gap, this work provides new understandings of the uniform convergence behavior of subdifferential sets by presenting a general principle that, notably, extends to settings without relying on the smoothness of population risks. Specifically, our work targets at  
%stochastic weakly convex objectives---a prominent class of nonsmooth and nonconvex objectives with widespread 
%applications~\cite{PoliquinRo92, PoliquinRo96, DavisDr20}. Our approach hinges on a fundamental property of weakly convex functions. Specifically, for any pair of 
%weakly convex functions, we establish that over any arbitrary open set within their domains, the uniform bound on the distance between their 
%respective set-valued subdifferential mappings, measured via the Hausdorff metric, is upper bounded by the uniform bound on the distance 
%between any corresponding vector-valued subgradient mappings of these functions (Theorem~\ref{theorem:bound-HD-by-MS}). 
%This deterministic result---which does not rely on differentiability---immediately implies a 
%crucial connection between the uniform convergence of subdifferentials and the uniform convergence of subgradient mappings 
%in sample average approximations, 
%enabling the analysis of the former through the latter (Theorem~\ref{theorem:main-stochastic-minimization-result}).

We instantiate this principle in the context of stochastic weakly convex objectives, a broad class of nonsmooth, nonconvex functions encompassing many widely studied models~\cite{PoliquinRo92, PoliquinRo96, DavisDr20}. Our main application is to stochastic convex-composite objectives~\cite{Burke85, DuchiRu18}, which include robust phase retrieval~\cite{DuchiRu19}, blind deconvolution~\cite{CharisopoulosDaDiDr21}, 
matrix sensing~\cite{LiZhMaVi20, CharisopoulosChDaDiDiDr21}, 
conditional Value-at-Risk~\cite{RockafellarUr00}, and nonlinear quantile regressions~\cite{KoenkerPa96}. 
These models appear across imaging, statistics, machine learning, and risk management. %(Theorem~\ref{theorem:main-stochastic-convex-composition-result}). 
For this class, we derive uniform convergence rates for subdifferentials by studying the convergence behavior of stochastic subgradient mappings (Theorem~\ref{theorem:basic-statistical-learning-result} and Theorem~\ref{theorem:final-result}). Our analysis complements prior work by (i) removing differentiability assumptions on the population risk, and (ii) establishing tight convergence rates under the Hausdorff distance, a topologically weaker notion than graphical convergence that more directly captures the geometry of subdifferential sets.
%And working under a topologically weaker notion of subdifferential convergence.
%Our result, which does not rely on a key assumption in the literature requiring the population risk to be continuously differentiable, provides tight convergence rates in the context of many existing 
%studies (see Section~\ref{sec:related-works} for a detailed explanation).
For instance, in the setting of robust phase retrieval~\cite{DuchiRu19, DavisDrPa20}, we establish a 
sharp convergence rate of $\sqrt{d/m}$ (modulo logarithmic factors on $d, m$) for the uniform convergence of subdifferential mappings under the Hausdorff distance, 
where $d$ is the problem's dimension and $m$ is the sample size (Corollary~\ref{corollary:phase-retrieval}). This is complementary to the state-of-the-art $\sqrt[4]{d/m}$ rate, which is 
established under the graphical distance of subdifferentials using an alternative approach based on Attouch's 
epigraphical convergence theorem~\cite{Attouch77, DavisDr22}. 
%Therefore, our results not only enhance the 
%convergence rate's dependence on the sample size $m$ but also achieve guarantees under a topologically weaker notion of convergence
%than existing work (see Section~\ref{sec:related-works} for a detailed explanation).
%Our result is essentially tight (modulo logarhmic factors), established under a topologically weaker notion of convergence of subdifferentials. 
Our results lead to new understanding into nonsmooth empirical risk landscapes, including how the locations of stationary points evolve with sample size—providing a sharper understanding of finite-sample behavior in high-dimensional stochastic optimization.

%e.g., the landscape of the finite sample robust phase 
%retrieval initially studied in~\cite{DavisDrPa20}, including the locations and convergence properties of stationary points.

%\vspace{0.65cm}

\subsection{Related Work and Contributions} 
\label{sec:related-works}
This work contributes to the extensive literature on characterizing the behavior of uniform convergence of subdifferentials in sample 
average approximation settings, for which we review. Notably, we saturate our contributions to existing work focused on 
settings where the objective functions defining the risks are simultaneously nonconvex and nonsmooth.

A significant body of literature has utilized set-valued variational analysis~\cite{RockafellarWe98} and random set theory~\cite{Molchanov05} 
to study the uniform convergence of subdifferentials from empirical to population risks~\cite{ShapiroXu07}. 
This technique is popular because subdifferential mappings of many nonsmooth and nonconvex 
functions—including weakly convex functions—often exhibit structural properties such as outer 
semicontinuity, and the subdifferentials themselves form convex sets, unlike discontinuous subgradient mappings~\cite{RockafellarWe98}.
This technique discretizes the parameter space, achieving initial subdifferential set convergence at discrete points 
via Artstein and Vitale's strong law of large numbers for random convex sets~\cite{ArtsteinVi75}. 
The core assumption is then the subdifferential mapping of the population risk is 
\emph{continuous} under the Hausdorff metric, which, in many cases (e.g., weakly convex functions), is equivalent to
continuous differentiability. 
This continuity—not merely outer semicontinuity—is 
crucial, see~\cite[Remark 2]{ShapiroXu07}, 
as it enables the extension of the initial convergence from discrete points to uniform convergence of subdifferential sets over the entire parameter space.
Without continuity, using outer-semicontinuity this approach can still ensure uniform convergence of \emph{enlarged} empirical subdifferentials to
\emph{enlarged} population subdifferentials; however, this does not apply to the \emph{original} subdifferentials~\cite{ShapiroXu07}.

%alone does not guarantee uniform convergence of the original subdifferentials, although it is enough to ensure uniform convergence of the 
%enlarged subdifferential of the sample average approximation to an expanded population subdifferential~\cite{ShapiroXu07}.

Therefore, a common and core assumption of this approach to attain uniform convergence of 
subdifferential sets is the %continuity of the population subdifferentials (which in many cases, equivalent to the
continuous differentiability of the population objective; 
while it facilitates many applications, 
it is sometimes restrictive and does not capture the full picture of subdifferential convergence~\cite[Remark 2]{ShapiroXu07}. 
%To attain uniform convergence (without rate guarantees), a continuity assumption is typically required~\cite{ShapiroXu07}. 
For achieving uniform convergence with a rate guarantee, stronger assumptions, such as H\"{o}lder continuity of 
population gradients, or $H$-calmness conditions---a H\"{o}lder type assumption---on subdifferential sets, are often necessary~\cite{Xu10}. 
Concrete analyses of uniform convergence rates under various smoothness assumptions of population risk have been explored in a varierty of statistical and machine-learning contexts, e.g.,~\cite{BaiJiSu19, MaFa21, LamWaWuZh22}. For instance, in multi-task learning, \cite{LamWaWuZh22} leverages the assumption of smooth and convex population objectives, using convex duality arguments to establish uniform convergence of subdifferentials. Other research aims to relax the smoothness requirement to piecewise continuity of subdifferential sets over 
disjoint regions, but this still necessitates the underlying probability measure being nonatomic 
to ensure that samples almost surely avoid the discontinuity boundaries of the subdifferentials~\cite{RalphXu11}.
%Therefore, while valuable, these methods do not fully address the fundamental challenge of uniform convergence of subdifferential mappings, particularly in broader settings where these smoothness assumptions no longer hold.

%in the context of nonsmooth and nonconvex population risks.
%Therefore, while these methods provide valuable insights under smoothness assumptions, they often fail to address the complexities of nonsmooth and nonconvex population risks, leaving fundamental questions unanswered about the uniform convergence of subdifferential mappings under weaker regularity conditions.

%In this work, we provide a more general understanding beyond the instances analyzed in these previous works by 
%presenting a novel perspective to establish uniform convergence of subdifferential set mappings, setting it apart from existing literature. 
In this work, we provide understanding beyond the specific instances of these prior works by presenting a % novel 
generalized principle that, most notably, encompasses settings where the aforementioned smoothness assumptions do not hold.
Specifically, our work targets the domain of stochastic weakly-convex classes, pivotal to numerous data-driven applications~\cite{DavisDr20}—while not relying on smoothness assumptions for the population risk 
often imposed in prior approaches. %, a restriction often imposed in prior approaches. %Unlike prior approaches that rely on smoothness assumptions for the population risk~\cite{LamWaWuZh22}, our framework accommodates nonsmooth and weakly convex objectives, a setting that arises naturally in many modern optimization problems.
The fundamental principle of our approach is that, for stochastic weakly convex objectives,
achieving uniform convergence—and the associated rates—of subdifferential sets 
using sample average approximation requires only the uniform convergence 
and rates of any chosen pair of the subgradient mappings of the corresponding empirical and population risk functions.
The proof of our principle is built on top of variational analytic tools such as Rademacher's theorem and Clarke's subdifferential characterizations. 
%This principle builds on Clarke's subdifferential characterization, where subdifferential maps can be effectively determined by carefully chosen subgradient selections on dense sets of differentiable points.
Leveraging this principle, we derive uniform convergence (and also rates) for a specific weakly-convex subclass termed stochastic convex-composite objectives, without requiring the population objectives to be continuously differentiable often required in the existing literature.
Importantly, obtaining uniform convergence (and rates) for discontinuous subgradient mappings in stochastic convex-composite objectives 
still demands sophisticated tools from statistical learning theory, such as chaining, the concept of shattering, and the notion of 
Vapnik–Chervonenkis dimension~\cite{Vapnik13}.

Notably, there is a recent advance in the literature that also targets at establishment of uniform subdifferential convergence 
within the domain of stochastic weakly convex classes in nonsmooth settings~\cite{DavisDrPa20, DavisDr22}. However, there are significant differences between 
our results and their results. 
First, there is a subtle difference in terms of the criteria for convergence. Our notion of uniform convergence on any open domain implies their 
notion of graphical convergence of subdifferential mappings on the same domain, but the reverse 
is not generally true (see Remark of Theorem~\ref{theorem:bound-HD-by-MS}). Therefore, our results are established 
under a topologically weaker notion of convergence.
Second, there is an important difference in the approaches and the established rates. Their approach, 
grounded in the utilization of Moreau envelope smoothing~\cite{Moreau65}, 
and Attouch's epigraphical convergence theorem~\cite{Attouch77}, achieves a convergence rate whose dependence on the sample 
size $m$ is often at $m^{-1/4}$. The reason is that the core principle behind their approach, which aims to derive graphical 
subdifferential distance bounds through the closeness of objective values, often inherently results in suboptimal 
guarantees in terms of the sample size~\cite[Section 5]{DavisDr22}. Our results are notably precise regarding the sample size 
$m^{-1/2}$, although they apply mainly to 
stochastic convex-composite objectives and introduce additional logarithmic factors. This indicates that further research is 
essential to fully delineate the advantages and limitations of both approaches.  

Finally, there are many interesting alternative ideas in the literature based on directly smoothing the nonsmooth
objective to attain uniform convergence of gradients of the smoothed objectives~\cite{XuZh09}, also 
used extensively in the statistics literature, e.g., in the context of quantile regressions~\cite{Horowitz98}. These results, 
however, do not directly characterize the convergence behavior of the original subdifferential sets. %Other results in the literature have addressed special cases of this problem under the assumption that the population objective is smooth~\cite{LamWaWuZh22}, but these settings do not apply to the broader, nonsmooth and nonconvex formulations considered here.

\subsection{Roadmap and Paper Organization} 
Section~\ref{sec:notation} provides the basic notation and definitions. 
Section~\ref{sec:general-theory} describes the general principle for analyzing the uniform convergence behavior for the 
subdifferential sets for weakly convex functions. Section~\ref{sec:explicit-uniform-convergence-rates} characterizes the explicit
uniform convergence rates of the subdifferential sets for stochastic convex-composite minimization problems. Section~\ref{sec:applications}
gives concrete applications illustrating our techniques on the problem of robust phase retrieval. The remainders are proofs and 
discussions.

\newcommand{\dom}{{\rm dom}}
\section{Notation and Basic Definitions} 
\label{sec:notation}
For $x, y \in \R$, we let $x \wedge y = \min\{x, y\}$ and $x \vee y = \max\{x, y\}$. In $\R^d$, we use $\norm{\cdot}$ and 
$\langle \cdot, \cdot \rangle$ to denote the standard Euclidean norm and inner product respectively. 
We let $\B(x_0; r) = \{x: \norm{x-x_0} < r\}$ denote the open $\ell_2$ ball with center at $x_0 \in \R^d$ and radius $r$. 
For matrices in $\R^{d \times d}$, we identify each such matrix with a vector in $\R^{d^2}$ 
and define $\langle A, B \rangle$ as ${\rm tr}(AB^T)$, where $\rm tr$ denotes the trace. For a set $X$, we use 
${\rm cl}, {\rm int}$ to denote the closure and interior respectively for the set $X$. For a closed convex set $X$, 
we use $\iota_X$ to denote the $+\infty$-valued indicator for the set $X$, that is $\iota_X(x) = 0$ if $x \in X$ and 
$+\infty$ otherwise. The normal cone to $X$ at $x$ is $\mathcal{N}_X(x):= \{v \in \R^d: \langle v, y-x\rangle \le 0~\text{for all}~y\in X\}$.

For a function $f \colon \R^d \to \R \cup\{+\infty\}$, its epigraph is the set $\{(x, \alpha) \in \R^d \times \R \mid \alpha \ge f(x)\}$.
It is called closed if its epigraph is closed in $\R^d \times \R$. It is called subdifferentially regular at $x$ if $f(x)$ is 
finite and its epigraph is Clarke regular at $(x, f(x))$ as a subset of $\R^d \times \R$
\cite[Definition 7.25]{RockafellarWe98}. A function $f\colon \R^d \to \R$ is said to be $\mathcal{C}^1$ smooth on $\R^d$ if it has continuous 
gradients $\grad f \colon \R^d \to \R^d$.

\subsection{Weakly Convex Functions and Subdifferentials}

We say $f\colon \R^d \to \R \cup \{+\infty\}$ is $\lambda$ locally 
weakly-convex near $x$
(also known as lower-$\mathcal{C}^2$~\cite{RockafellarWe98} or semiconvex~\cite{BolteDaLeMa10}) if there 
exists $\eps > 0$ such that 
\begin{equation*}
	y \mapsto f(y) + \frac{\lambda}{2} \norm{y}^2, ~~y \in \B(x; \eps)
\end{equation*}
is convex~\cite[Chapter 10.G]{RockafellarWe98}. We say a function $f \colon \R^d \to \R \cup \{+\infty\}$ is 
a locally weakly convex function on $\openset \subseteq \R^d$ if at every $x \in \openset$, there is $\lambda < \infty$ 
such that $f$ is $\lambda$ locally weakly-convex near $x$.
For a function $f\colon \R^d \to \R \cup \{+\infty\}$ and a point $x$ with $f(x)$ finite, we let 
$\partial f(x)$ denote the Fr\'{e}chet subdifferential (or regular~\cite[Chapter 8.B]{RockafellarWe98}) of $f$ at the point $x$, 
\begin{equation*}
	\partial f(x) = \left\{g \in \R^d  \colon f(y) \ge f(x) + \langle g, y-x \rangle + o(\norm{y-x})~~\text{as}~~y \to x\right\}.
\end{equation*}
Here, any $g \in \partial f(x)$ is referred to as a subgradient of $f$ at the point $x$. We say $v$ is a \emph{horizon
subgradient} of $f$ at $x$ with $f(x)$ finite, written as $v \in \partial^\infty f(x)$ if there exist sequences $x_i, v_i \in \partial f(x_i)$ and 
$\tau_i \to 0$ satisfying $(x_i, f(x_i), \tau_i v_i) \to (x, f(x), v)$~\cite[Chapter 8.B]{RockafellarWe98}. 

Notably, when $f$ is smooth, the subdifferential $\partial f(x)$ consists only of the gradient $\{\grad f(x)\}$, while 
for convex function, it coincides with the subdifferential in convex analysis~\cite{Rockafellar70}.  
For a weakly convex function $f$, the subdifferential $\partial f(x)$ is non-empty, compact, and convex for every $x$ 
in the interior of $\dom f$. Notably, these follow from the corresponding results for convex 
functions~\cite[Chapter 8]{RockafellarWe98}.

\subsection{Set-valued Analysis}  
Our definitions follow the references of Rockafellar and Wet~\cite{RockafellarWe98}, and
Aubin and Frankowska~\cite{AubinHe99}. 
For a set $A \subseteq \R^d$, we denote $\norm{A} = \sup_{a\in A} \norm{a}$. 

For a set $A \subseteq\R^d$ 
and $y \in \R^d$, we denote by $\dist(y, A) = \inf_{z \in A} \norm{y-z}$ the 
distance from $y$ to $A$ with respect to the Euclidean norm $\norm{\cdot}$. 
For two sets $A_1, A_2 \subseteq \R^d$, we denote by 
\begin{equation*}
	\D(A_1, A_2) = \sup_{x \in A_1} \dist(x, A_2)
\end{equation*}
the deviation of the set $A_1$ from the set $A_2$, by 
\begin{equation*}
	\H(A_1, A_2) = \max\{\D(A_1, A_2), \D(A_2, A_1)\}
\end{equation*}
the Hausdorff distance between $A_1$ and $A_2$. 
Given a sequence of sets $A_n \subseteq \R^d$, the limit supremum of the sets consists of limit points of 
subsequences $y_{n_k} \in A_{n_k}$, that is, 
\begin{equation*}
	\limsup_n A_n = \{y \colon \exists n_k,  \exists y_{n_k} \in A_{n_k}~\text{s.t.}~y_{n_k} \to y~\text{as}~k \to \infty\}.
\end{equation*}
The limit infimum of the sets consists of limit points of sequences $y_n \in A_n$, that is: 
\begin{equation*}
	\liminf_n A_n = \{y \colon \exists y_n \in A_n~\text{s.t.}~y_n \to y~\text{as}~n \to \infty\}.
\end{equation*}
We let $G \colon X \rightrightarrows \R^d$ denote a set-valued mapping from $X$ to $\R^d$, and
$\dom(G):= \{x\colon G(x) \neq \emptyset\}$. A function $g \colon X \to \R^d$ is said to be a selection of 
$G\colon X \rightrightarrows \R^d$ if $g(x) \in G(x)$ for every $x \in X$. 
We say $G$ is outersemicontinuous if for any sequence $x_n \to x \in \dom(G)$,
we have $\limsup_n G(x_n) \subseteq G(x)$. We say $G$ is innersemicontinuous if
for any $x_n \to x \in \dom(G)$, we have $\liminf_n G(x_n) \supseteq G(x)$. We say $G \colon X \rightrightarrows \R^d$
is continuous if it is both outersemicontinuous and innersemicontinuous.  

For a weakly convex function $f \colon \R^d \to \R \cup\{+\infty\}$, the subgradient 
mapping $\partial f \colon  {\rm int}~\dom(f) \rightrightarrows \R^d$ is outersemicontinuous~\cite[Chapter 8]{RockafellarWe98}. 
%For $G: X \rightrightarrows \R^d$ and a measure $\mu$ on $X$, we define integrals of set-valued functions, following
%Aumann~\cite{Aumann65}: 
%\begin{equation*}
%	\int_X Gd\mu := \left\{\int_X g(x) \mu(dx) \mid g(x) \in G(x)~\text{for}~x \in X,~g~\text{integrable}\right\}. 
%\end{equation*}

\subsection{Probability Space and Random Set-Valued Map}
\label{sec:probability-space-and-random-set-valued-map}
%Our definitions follow closely the references of Billingsley~\cite{Billingsley86}, 
%Aubin and Frankowska~\cite[Chapter 8]{AubinHe99} and~Shapiro et.al.~\cite[Chapter 7]{ShapiroDeRu21}. 
Let $(\Xi, \mathscr{G}, \P)$ denote a probability space. %We assume the measure $\P$ 
%is complete, meaning any subset of a measure zero set is also measurable~\cite{Billingsley86}. 
We frequently consider random set-valued mapping of the form $\mathcal{A}  \colon X \times \Xi \rightrightarrows \R^d$. 
Let $\mathscr{B}$ denote the space of nonempty, compact subsets of $\R^d$. 
By~\cite[Theorem 14.4]{RockafellarWe98}, we say the mapping $\xi \mapsto \mathcal{A}(x, \xi)$ is measurable if and only if 
for every $\mathcal{S} \in \mathscr{B}$, $\mathcal{A}(x, \cdot)^{-1}\mathcal{S} = \{\xi \colon \mathcal{A}(x, \xi) \cap \mathcal{S} \neq \emptyset\}$ is 
a measurable set under $(\Xi, \mathscr{G}, \P)$.  

By a selection of the random set $\mathcal{A}(x, \xi)$, we refer to a random vector $a(x, \xi) \in \mathcal{A}(x, \xi)$,
meaning that $\xi \mapsto a(x, \xi)$ is measurable. Note that such selection exists if 
$\xi \mapsto \mathcal{A}(x, \xi)$ is closed-valued (meaning that $\mathcal{A}(x, \xi)$ is a closed set for every $\xi$) 
and measurable~\cite{Aumann65}, see 
also~\cite[Theorem 8.1.3]{AubinHe99}.
For a map $\mathcal{A} \colon X \times \Xi \rightrightarrows \R^d$ and a probability measure $\P$, we define its expectation, following
Aumann~\cite{Aumann65}. At every $x \in X$: 
\begin{equation*}
	\int_\Xi \mathcal{A}(x, \xi) P(d\xi) := \left\{\int_\Xi a(x, \xi) P(d\xi) \mid a(x, \xi) \in \mathcal{A}(x, \xi)~
		\text{for}~\xi \in \Xi,~a(x, \cdot)~\text{integrable}\right\}. 
\end{equation*}

We say $f \colon \R^d \times \Xi \to \R \cup \{+\infty\}$ is a random function when $\xi \mapsto f(x, \xi)$ is 
measurable at every $x \in X$~\cite[Chapter 7]{ShapiroDeRu21}. Suppose a random function $f$ also 
satisfies $x\mapsto f(x, \xi)$ is locally weakly convex, and real-valued near a point $x \in \R^d$, then the
subdifferential $\partial f(x, \xi)$ satisfies $\xi \mapsto \partial f(x, \xi)$ is measurable at that
$x \in \R^d$, following~\cite[Section 7.2.6]{ShapiroDeRu21} and~\cite[Proposition 2.1]{XuZh09}.

\subsection{Measurability Issues} 
\label{sec:probability-measure-and-measurability}
This subsection can be mostly skipped on a first read, but it is essential for addressing measurability 
issues that arise when considering the supremum of random functions over an uncountable index 
set~\cite{Billingsley86}. This is important because we are interested in the supremum of the Hausdorff 
distance between stochastic subdifferential mappings over a given domain.

%In this paper, we often take the supremum over an uncountable set, which can lead to functions that are not necessarily measurable. 
To address these measurability issues, we use the concept of outer measure $\P^*$ and  inner measure $\P_*$ 
frequently adopted in empirical process theory~\cite[Section 1.2-5]{VanDerVaartWe96}.
Given a probability space $(\Xi, \mathscr{G}, \P)$, every subset $B \subseteq \Xi$ is assigned with an 
outer measure: 
\begin{equation*}
	\P^*(B) = \inf \{\P(A) \colon B\subseteq A, A\in \mathscr{G}\}.
\end{equation*} 
Notably, $\P^*$ is subadditive: $\P^*(B_1 \cup B_2) \le \P^*(B_1) + \P^*(B_2)$, following the union bound. 
Correspondingly, the inner measure is $\P_*(B) = \sup\{\P(A) \colon B\supseteq A, A \in \mathscr{G}\}$. Clearly, $\P_*(B) = 1-\P^*(\Xi \backslash B)$.
Finally, when $B \in \mathscr{G}$ is measurable, then the inner and outer measure agree: $\P^*(B) = \P_*(B) = \P(B)$.

\section{General Theory}
\label{sec:general-theory}
\subsection{Uniform Subdifferential Bounds via Selections} 
In this section, we introduce Theorem~\ref{theorem:bound-HD-by-MS}, a general technique
for obtaining uniform bounds on the subdifferentials between two real-valued, locally weakly convex functions. 
The theorem states that the supremum of the Hausdorff distance between their subdifferentials over 
any open set is upper bounded by the supremum of the norm differences between \emph{any} selected 
subgradient mappings. 

Theorem~\ref{theorem:bound-HD-by-MS} might initially seem surprising because at any single point 
$x$, the Hausdorff distance between the subdifferentials at $x$ can exceed the norm difference 
between the selected subgradients. To further appreciate this result, we will provide counterexamples 
in the remark 
to highlight why the assumption of $\openset$ being an open set is crucial.

%To the best of our knowledge, we are unaware of such a result documented in the literature. 
%In the context of convex settings, it is typical to establish the uniform convergence of subdifferentials under the Hausdorff metric, often in the sense of graphical convergence. This is usually achieved through epi-convergence of convex functions, as highlighted by Attouch's renowned theorem 
%we are unaware if such a result is formally documented in the literature. 
We note that in the (weakly) convex settings,  graphical distance
between subdifferentials under the Hausdorff metric %(achieved as a consequence of the stronger notion of graphical convergence) 
is typically established through controlling the
epi-distance of functions due to Attouch's celebrated theorem~\cite{Attouch77, AttouchBe93, DavisDr22}.  See, e.g.,
\cite[Theorem 12.35]{RockafellarWe98}. We shall make comparisons to clarify two distance metrics of subdifferential mappings 
in the second remark: the supremum of the Hausdorff distance between subdifferential mappings over $x$ in an open set $\openset$,
and 
the graphical distance between subdifferential mappings over the same open set $\openset$.

\vspace{.5cm}

\begin{theorem}
\label{theorem:bound-HD-by-MS}
Let $f_1$ and $f_2$ be locally weakly convex functions from $\openset$ to $\R$, 
where $\openset$ is an open set in $\R^d$. Let
$g_1$ and $g_2$ be selections of the subdifferentials $\partial f_1$ and $\partial f_2$ on $\openset$, 
respectively, i.e., obeying $g_1(x) \in \partial f_1(x)$ and $g_2(x) \in \partial f_2(x)$ for all $x \in \openset$. 
Then the following inequality holds: 
\begin{equation}
	\sup_{x \in \openset} \H(\partial f_1(x), \partial f_2(x)) \le \sup_{x \in \openset} 
		\norm{g_1(x) - g_2(x)}.
\end{equation} 
\end{theorem} 

\newcommand{\conv}{{\rm conv}}
\begin{proof}
By Rademacher's theorem (e.g.,\cite[Theorem 9.6.0]{RockafellarWe98}), both functions $f_i$ are almost everywhere differentiable in $\openset$. Clarke's subdifferential characterization (e.g., \cite[Theorem 9.6.1]{RockafellarWe98}) further states that the subdifferential set $\partial f_i(x)$ can be expressed as:
\begin{equation}
\label{eqn:variational-characterization-of-the-subdifferential-set} 
\begin{split} 
	\partial f_i(x) &= \conv\{\mathcal{G}_i(x)\}
\end{split} 
\end{equation} 
where 
\begin{equation*}
	\mathcal{G}_i(x) = \left\{\lim_{k \to \infty} \grad f_i(x_k) \colon x_k \in D_i, x_k \to x, \lim_{k \to \infty} 
		\grad f_i(x_k)~\text{exists}\right\}. 
\end{equation*} 
Here, $D_i$ denotes the set of differentiable points of $f_i$, which is dense in the open set $\openset$
(meaning that $\openset \backslash D_i$ has Lebesgue measure $0$), and $\conv$ 
represents the convex hull of a set. By definition, the sets $\mathcal{G}_i(x)$ are closed at every $x$, 
and the set-valued 
mappings $x \mapsto \mathcal{G}_i(x)$ are locally bounded as the functions $f_i$ are locally Lipschitz. 

To upper bound the Hausdorff distance between the 
subdifferential sets $\partial f_1(x)$ and $\partial f_2(x)$,  we use a basic property of the Hausdorff distance:
\begin{equation*}
	\H(\conv\{A_1\}, \conv\{A_2\}) \le \H(A_1, A_2)~~~\text{for any sets $A_1$ and $A_2$}.
\end{equation*}
Applying this property to $\partial f_1(x)$ and $\partial f_2(x)$, we obtain for every $x \in \openset$: 
\begin{equation}
	\H(\partial f_1(x), \partial f_2(x)) = \H(\conv\{\mathcal{G}_1(x)\}, \conv\{\mathcal{G}_2(x)\})
		\le \H(\mathcal{G}_1(x), \mathcal{G}_2(x)).
\end{equation} 
Therefore, the problem reduces to bounding the Hausdorff distance between $\mathcal{G}_1(x)$ and $\mathcal{G}_2(x)$. By definition, the Hausdorff distance between these sets is given by:
\begin{equation*}
	\H(\mathcal{G}_1(x), \mathcal{G}_2(x))
		= \max\left\{\sup_{v_1 \in \mathcal{G}_1(x)} \inf_{v_2 \in \mathcal{G}_2(x)}\norm{v_1 - v_2},
			\sup_{v_2 \in \mathcal{G}_2(x)} \inf_{v_1 \in \mathcal{G}_1(x)}\norm{v_1 - v_2}\right\}.
\end{equation*} 
Our specific goal is to bound the first term, showing that for every $x \in \openset$: 
\begin{equation}
\label{eqn:specific-goal}
	\sup_{v_1 \in \mathcal{G}_1(x)} \inf_{v_2 \in \mathcal{G}_2(x)}	\norm{v_1 - v_2}
		\le \sup_{x \in \openset} \norm{g_1(x) - g_2(x)},
\end{equation} 
since a completely symmetric argument would show the same statement holds to the second term.

Fix $x \in \openset$. To establish~\eqref{eqn:specific-goal}, let $v_1^* \in \mathcal{G}_1(x)$ be the subgradient 
achieving the supremum, which exists due to the compactness of the set $\mathcal{G}_1(x)$. By definition, there 
exists a sequence $\{x_k\}$ such that $x_k \to x$, $x_k \in D_1$ (the set of differentiable points of $f_1$) and 
$\grad f_1(x_k) \to v_1^*$. 

We now argue that we can, without loss of generality, assume that for every $k$, $x_k \in D_2$ (the set of differentiable points 
of $f_2$) as well. If some $x_k \not\in D_2$, applying Clarke's subdifferential characterization to $f_1$ at $x_k$~\eqref{eqn:variational-characterization-of-the-subdifferential-set}, we can replace such $x_k$ with 
some $x_k'$ belonging to the subset $D_1 \cap D_2$ dense in $\openset$ (note $\openset \backslash 
(D_1 \cap D_2)$ has Lebesgue measure $0$)
such that $\norm{x_k' - x_k} \to 0$ and $\norm{\grad f_1(x_k') - \grad f_1(x_k)} \to 0$ as $k \to \infty$. As a result, 
$x_k'$ preserves all the properties of $x_k$, including $x_k' \to x$, $x_k' \in D_1$ and $\grad f_1(x_k') \to v_1^*$, 
while guaranteeing that $x_k' \in D_2$ for all $k$. 

Let $v_2^*$ be any accumulation point of $\grad f_2(x_k)$ (which exists since $\mathcal{G}_2$ are locally bounded
near $x$). By taking a subsequence if necessary, we can assume that $\grad f_2(x_k) \to v_2^*$.
Then, 
\begin{equation}
\label{eqn:specific-goal-two}
	\sup_{v_1 \in \mathcal{G}_1(x)} \inf_{v_2 \in \mathcal{G}_2(x)}	\norm{v_1 - v_2}
		\le \norm{v_1^* - v_2^*} = \lim_{k \to \infty} \norm{\grad f_1(x_k) - \grad f_2(x_k)}.
\end{equation} 
At differentiable points $x_k \in D_1\cap D_2$, the subdifferential is a singleton: $\partial f_i(x_k) = \{\grad f_i(x_k)\}$, so the subgradient selection must obey $g_i(x_k) = \grad f_i (x_k)$. Therefore, 
\begin{equation*}
	\sup_{v_1 \in \mathcal{G}_1(x)} \inf_{v_2 \in \mathcal{G}_2(x)}	\norm{v_1 - v_2}  \le 
		\lim_{k \to \infty} \norm{g_1(x_k) - g_2(x_k)} \le \sup_{x \in \openset} \norm{g_1(x) - g_2(x)}.
\end{equation*} 
This establishes \eqref{eqn:specific-goal}. Since the argument is symmetric, we conclude:
\begin{equation}
	\H(\partial f_1(x), \partial f_2(x)) %= \H(\conv\{\mathcal{G}_1(x)\}, \conv\{\mathcal{G}_2(x)\})
		\le \H(\mathcal{G}_1(x), \mathcal{G}_2(x)) \le \sup_{x \in \openset} \norm{g_1(x) - g_2(x)}.
\end{equation} 
Finally, since the above inequality holds for all $x \in \openset$, this completes the proof of
Theorem~\ref{theorem:bound-HD-by-MS}.

\end{proof} 
\renewcommand{\theremark}{\Alph{remark}}

\begin{remark}[Necessity of $\openset$ being an open set]

\emph{
Theorem~\ref{theorem:bound-HD-by-MS} would not hold if $\openset$ were not an open set. Consider 
the convex functions $f_1(x) = |x|$ and $f_2(x) = 2|x|$. Notably, the subdifferential sets are
\begin{equation*}
	\partial f_1(x) = 
		\begin{cases}
			1 & x \ge 0\\ 
			[-1, 1] & x = 0 \\
			-1 & x \le 0
		\end{cases},
	~~~~~~
	\partial f_2(x) = 
	\begin{cases}
			2 & x \ge 0\\ 
			[-2, 2] & x = 0 \\
			-2 & x \le 0
	\end{cases}.
\end{equation*}
Let the subgradient selections $g_1$ and $g_2$ of $\partial f_1$ and $\partial f_2$ respectively be as follows: 
\begin{equation*}
	g_1(x) = 
		\begin{cases}
			1 & x > 0\\ 
			0 & x = 0\\
			-1 & x < 0
		\end{cases},
	~~~~~~
	g_2(x) = 
	\begin{cases}
			2 & x > 0\\ 
			0 & x = 0\\
			-2 & x < 0
	\end{cases}.
\end{equation*}
Consider the case where $\openset = \{0\}$, which is a singleton, and thus a closed set. Then 
\begin{equation*}
	\sup_{x \in \openset} \H(\partial f_1(x), \partial f_2(x)) = \H(\partial f_1(0), \partial f_2(0)) = 1 > 0 = |g_1(0) - g_2(0)| = \sup_{x \in \openset} \norm{g_1(x) - g_2(x)}.
\end{equation*}
This illustrates the necessity for $\openset$ to be an open set for Theorem~\ref{theorem:bound-HD-by-MS} to hold. 
\emph}
\end{remark}

%\noindent\noindent\textbf{Remark}~(Comparison between Graphical Distance and Supremum of Hausdorff Distance)~
\renewcommand{\d}{{\rm d}}
\newcommand{\gph}{{\rm gph}}
\newcommand{\half}{\frac{1}{2}}

\vspace{.5cm}

\begin{remark}[Comparison of Metrics on Subdifferential Mappings]
\emph{
Let $\openset \subseteq \R^d$ be open. 
For a pair of locally weakly convex functions $f_i \colon \openset \to \R$ for $i= 1, 2$, we define: 
%Let $f_1$ and $f_2$ denote two locally weakly convex functions from $\openset$ to $\R$. We define: 
\begin{equation}
\begin{split} 
	\d_1(\partial f_1, \partial f_2) &:= \sup_{x \in \openset} \H(\partial f_1(x), \partial f_2(x)),
	~~~\d_2(\partial f_1, \partial f_2) := \H(\gph_\openset \partial f_1, \gph_\openset \partial f_2).
\end{split} 
\end{equation}
In defining $\d_2$, for a locally weakly convex $f\colon \openset \to \R$, we recall the graph of its subdifferential: 
\begin{equation*}
	\gph_\openset \partial f = \{(x, y) \in \openset \times \R^d: y \in \partial f(x)\},
\end{equation*}
where the Cartesian product $\openset \times \R^d \subseteq \R^d \times \R^d \cong \R^{2d}$ is equipped with Euclidean distance. 
By definition, $\d_1, \d_2$ are metrics on $\mathbb{S}_\openset = 
\{g\colon \openset \rightrightarrows \R^d \colon g = \partial f~\text{for some locally weakly convex $f \colon \openset \to \R$}\}$.
The metric $\d_2$ is known as the graphical distance, aligning well with applications of 
Attouch's epigraphical convergence theorem~\cite{Attouch77}, see also~\cite[Theorem 5.1]{DavisDr22}.
}
%\emph{
%Following the definition of Hausdorff distance, there is a more transparent way to express $\d_2$: 
%\begin{equation*}
%	\d_2(\partial f_1, \partial f_2) = \max\left\{\sup_{\substack{(x_1, y_1) \in \gph \partial f_1 \\ x_1 \in \openset,y_1 \in \R^d}} 
%		\dist((x_1, y_1), \gph \partial f_2), 
%		\sup_{\substack{(x_2, y_2) \in \gph \partial f_2 \\ x_2 \in \openset,y_2 \in \R^d}} \dist((x_2, y_2), \gph \partial f_1)\right\}.
%\end{equation*}
%}

\emph{
We document relations between $\d_1$ and $\d_2$. For every open set $\openset$, and every pair of 
locally weakly convex $f_i \colon \openset \to \R$ for $i= 1, 2$, the following bound holds from the definition:
\begin{equation}
\label{eqn:stronger-metric}
	\d_2(\partial f_1, \partial f_2) \le \d_1(\partial f_1, \partial f_2).
\end{equation}
In other words, the metric $\d_1$ is topologically stronger than the metric $\d_2$ on the space $\mathbb{S}_\openset$. On the 
other hand, when $\openset = \R$, we can construct locally weakly convex functions $f_{1, n}$ and $f_{2, n}$ such that 
\begin{equation}
\label{eqn:non-equivalence}
	\lim_{n\to \infty} \d_2(\partial f_{1, n}, \partial f_{2, n}) = 0~~\text{while}~~\lim_{n\to \infty} \d_1(\partial f_{1, n}, \partial f_{2, n}) = 1,
\end{equation}
meaning the metrics $\d_1$ and $\d_2$ are not equivalent in general. To construct such a sequence, we can take 
convex functions $f_{1, n}(x) = |x|$, and $f_{2, n}(x) = \half \left|x-\frac{1}{n}\right| + \half \left|x+\frac{1}{n}\right|$, which yields: 
\begin{equation*}
	\partial f_{1, n}(x) = 
		\begin{cases}
			1 & x \ge 0\\ 
			[-1, 1] & x = 0 \\
			-1 & x \le 0
		\end{cases},
	~~~\text{and}~~~
	\partial f_{2, n}(x) = 
		\begin{cases}
			1 & x > 1/n\\ 
			[0, 1] & x  = 1/n \\
			0 & x \in (-1/n, 1/n) \\
			[-1, 0] & x = -1/n \\
			-1 & x < -1/n
		\end{cases}.
\end{equation*}
A direct verification gives $\d_1(\partial f_{1, n}, \partial f_{2, n}) = 1$ and $\d_2(\partial f_{1, n}, \partial f_{2, n})  = 1/n$ for every $n \in \N$. 
}

\emph{To conclude, the metric $\d_1$ is topologically stronger than the metric $\d_2$ on $\mathbb{S}_\openset$ for every open set $\openset$.
An upper bound on $\d_1(\partial f_1, \partial f_2)$ induces an upper bound on $\d_2 (\partial f_1, \partial f_2)$. 
Nevertheless, the two metrics are not equivalent in general.}
\end{remark}

\newcommand{\oo}{\iota}
\newcommand{\X}{\mathbf{X}}
\newcommand{\interior}{{\rm int}}

\subsection{Implications to Stochastic Weakly Convex Minimization Problems} 
\label{sec:implications-to-stochastic-weakly-convex-minimizations}
Theorem~\ref{theorem:bound-HD-by-MS} provides key insights into the uniform convergence 
of subdifferentials in stochastic minimization problems.
Theorem~\ref{theorem:main-stochastic-minimization-result}, which builds on top of 
Theorem~\ref{theorem:bound-HD-by-MS}, shows that for stochastic minimizing objectives with 
certain weak convexity requirements, the uniform convergence of subdifferential sets can be 
effectively understood by analyzing the uniform convergence of any pair of selections from the 
subdifferentials of both population and empirical objectives.

 %Our main finding is that the uniform convergence of any measurable 
%selection from the subdifferentials of population and empirical objectives, within any open set, 
%ensures the convergence of the entire subdifferentials. 

To formalize our results, we begin by introducing our setup. Let $f(\cdot, \cdot) \colon \R^d \times \Xi \to \R$ denote 
a real-valued random function. %, where $\Xi$ is a Euclidean space. 
We are sampling $\xi \sim \P$ where $\P$ is supported on $\Xi$, assumed to be a Euclidean space. 
%Since we can always enlarge the sigma-algebra and extend the probability measure in such a way as to get a 
%complete measure space, 
%we always assume the underlying probability measure is a complete measure to avoid some cumbersome measurability issues. 

Under this setup, the \emph{population risk} is given by: 
\begin{equation*}
%\label{eqn:population-weak-convex-objective}
	\min_{x\in \R^d} \phi(x) = f(x) + R(x)+ \oo_X(x) ~~~\text{where}~~~f(x) = \E_{\xi \sim \P}[f(x, \xi)] = \int_{\Xi} f(x, \xi) P(d\xi),
\end{equation*}
whereas its associated \emph{empirical risk} is given by: 
\begin{equation*}
%\label{eqn:empirical-weak-convex-objective}
	\min_{x \in \R^d} \phi_S(x) =  f_S(x) + R(x)+ \oo_X(x)
	~~~\text{where}~~~f_S(x) = \E_{\xi \sim \P_m}[f(x, \xi)] = \frac{1}{m}\sum_{i=1}^m f(x, \xi_i).
\end{equation*}
In the above, $\xi_1, \xi_2, \ldots, \xi_m$ are i.i.d. samples drawn from the distribution $\P$, where 
$\P_m$ denotes the empirical distribution of $\xi_1, \xi_2, \ldots, \xi_m$. Formally, we start with 
a single random variable $\xi$ defined on a probability space $(\Xi, \mathcal{G}, \P)$. Utilizing this space, 
we can construct the countable product space $(\Xi, \mathcal{G}, \P)^\N$ where the sequence of i.i.d. random variables 
$(\xi_1, \xi_2, \ldots, \xi_m, \ldots)$ can be properly defined on this product probability space~\cite[Section 36]{Billingsley86},
with each $\xi_i$ acting as a coordinate projection. In this product space, 
we then extend the sigma-algebra and the probability measure to obtain a complete measure, following 
standard procedures in measure theory~\cite{Billingsley86}, to avoid some cumbersome measurability issues~\cite[Section 7]{ShapiroDeRu21}. 
The function $R \colon \R^d\to \R$
is a real-valued, closed and convex function. The constraint set $X$ is a 
nonempty, closed, and convex set. 

We are interested in establishing an upper bound for the rate of 
convergence of subdifferentials of the empirical objective $\phi_S$ to the population objective $\phi$ 
over some subset of $X$, frequently modeled as $X \cap \openset$ where $\openset$
is an open set in $\R^d$.  This openness requirement for $\openset$ originates from 
Theorem~\ref{theorem:bound-HD-by-MS} and, despite this condition, 
our result is relevant to a broad range of interesting applications. %, and notably, $\openset \cap X = X$ when $\openset = \R^d$. 
%\begin{equation*}
%	\sup_{x \in \openset} \H(\phi(x), \phi_S(x)).
%\end{equation*}

To ease our discussions, we assume: 
%To deal with the constraint set $X$, we make a second assumption.  
\begin{assumption}
\label{assumption:avoid-boundary}
$f(x) = \int_{\Xi} f(x, \xi) P(d\xi) < \infty$ for all $x \in \openset$ where $\openset$ is an open set. 
\end{assumption}
Assumption~\ref{assumption:avoid-boundary} implies, with probability one, $f$ and $f_S$ are 
real-valued functions on $\openset$. In many practical scenarios, the objective $f$ has its domain $\R^d$, 
for which Assumption~\ref{assumption:avoid-boundary} naturally holds.

%, given we have the In many data-driven applications, $f$ is a real-valued function 
%across the entire $\R^d$, so this assumption would naturally hold. 
%%which applies to many data-driven applications. 
%Notably, our results can be easily extended to more broad scenarios where $X \subseteq \interior(\dom(f))$.
%This is because the objective values $\phi(x), \phi_S(x)$ and subdifferentials $\partial \phi(x), \partial \phi_S(x)$ 
%only depend on the values of $f$ and $f_S$ within any open set containing the constraint set $X$. 

%Thus, we make Assumption~\ref{assumption:avoid-boundary} for notational convenience.

%, where the function $f$ typically has its domain $\interior(f) = \R^d$,  aligning with the scenario where 
%$X \subseteq \R^d$ represents constraints on the parameters. 

Our results also require the following locally weak convexity assumption on $f(\cdot, \xi)$. 
%\begin{assumption}
%\label{assumption:avoid-boundary}
%$X \subseteq \interior(\dom f)$.
%\end{assumption}
%Assumption~\ref{assumption:avoid-boundary} is relevant to many large-scale data science applications, 
%where the function $f$ typically has its domain $\interior(f) = \R^d$,  aligning with the scenario where 
%$X \subseteq \R^d$ represents constraints on the parameters. Building on this framework for stochastic weakly
%convex minimization, we further require that 
%the function $f(\cdot, \xi)$ is locally weakly convex.
%This is formally documented in Assumption~\ref{assumption:weak-convexity}.
\begin{assumption}
\label{assumption:weak-convexity}
For all $x$ in $\openset$, there exists $\eps(x) > 0$
and $\lambda(x, \xi) \ge 0$ such that
\begin{equation*}
	y \mapsto f(y, \xi) + \frac{\lambda(x, \xi)}{2} \norm{y}^2
\end{equation*}  
is convex on the open ball $\B(x; \eps(x))$ and $\E_{\xi \sim \P}[\lambda(x, \xi)] < \infty$. 
\end{assumption} 

\vspace{.2cm}
\begin{lemma}
\label{lemma:basic-weak-convexity-probability-one-lemma}
Let Assumptions~\ref{assumption:avoid-boundary} and~\ref{assumption:weak-convexity} hold. 
Then with probability one, the function $f_S(x)$ is locally weakly convex on the open set $\openset$.
Also, the function $f(x) = \E[f_S(x)]$ is locally weakly convex on $\openset$. 
\end{lemma}
\begin{proof}
By Assumption~\ref{assumption:weak-convexity}, for each $x \in \openset$, there exists a radius 
$\eps(x) > 0$ and $\lambda(x, \xi) \ge 0$ such that 
\begin{equation*}
	y \mapsto f(y, \xi) + \frac{\lambda(x, \xi)}{2} \norm{y}^2
\end{equation*}  
is convex on the open ball $\B(x; \eps(x))$. The collection of these balls, $\{\B(x; \eps(x))\}_{x \in \openset}$ forms an open cover 
of $\openset$. By Lindel\"{o}f's covering Theorem, see, e.g.,~\cite[Page 50]{Armstrong13}, and noting that $\R^d$  with its norm topology
is second countable, there is a countable
subcover $\{\B(x; \eps(x))\}_{x \in \openset^o}$ of $\openset$, where $\openset^o$ is a countable subset of $\openset$. 
In particular, $\openset$ obeys the inclusion: 
\begin{equation*}
	\openset \subseteq \bigcup_{x \in \openset^o} \B(x; \eps(x)).
\end{equation*}

For each $x \in \openset^o$, the function
\begin{equation*}
	y \mapsto f_S(y) +  \frac{\E_{\xi \sim \P_m}[\lambda(x, \xi)]}{2} \norm{y}^2 = \frac{1}{m}  \sum_{i=1}^m \left(f(y, \xi_i)+\frac{\lambda(x, \xi_i)}{2} \norm{y}^2\right)
\end{equation*}
is convex on the open neighborhood $\B(x; \eps(x))$. Given that $\E_{\xi \sim \P}[\lambda(x, \xi)] < \infty$ holds, for each
$x \in \openset^o$ it follows that 
$\E_{\xi \sim \P_m}[\lambda(x, \xi)] < \infty$ with probability one. Using a union bound over the countable set $\openset^o$
this implies that with probability one, $\E_{\xi \sim \P_m} [\lambda(x, \xi)] < \infty$ simultaneously holds for every $x \in \openset^o$.
Thus, with probability one $f_S(x)$ is locally weakly convex on the set $\openset$, as $\openset$ is included in the countable union 
of these open balls: $ \openset \subseteq \bigcup_{x \in \openset^o} \B(x; \eps(x))$.
A similar reasoning yields that the function $f(x) = \E[f_S(x)]$ is locally weakly convex within the open set $\openset$. 
\end{proof}

%To see this, for each $x \in \openset$, 
%\begin{equation*}
%	y \mapsto f_S(y) +  \frac{\E_{\xi \sim \P_m}[\lambda(x, \xi)]}{2} \norm{y}^2 = \frac{1}{m}  \sum_{i=1}^m \left(f(y, \xi_i)+\frac{\lambda(x, \xi_i)}{2} \norm{y}^2\right)
%\end{equation*}
%is convex on the open neighborhood $\B(x; \eps(x))$. %, with $f_S(x) < \infty$ holding with probability one. 

Building on Theorem~\ref{theorem:bound-HD-by-MS} and Lemma~\ref{lemma:basic-weak-convexity-probability-one-lemma} then
Theorem~\ref{theorem:main-stochastic-minimization-result} naturally follows. 

%Recall that, for a random set-valued map $\mathcal{A}(x, \xi)$, 
%a function $a(x, \xi)$ is considered a selection of $\mathcal{A}$ over the domain $x \in X$ if $a(x, \xi)$ is a random
%function satisfying $a(x, \xi) \in \mathcal{A}(x, \xi)$ for all $x \in X$ and all $\xi$.

\newcommand{\set}{{A}}

\vspace{.5cm}

\begin{theorem}
\label{theorem:main-stochastic-minimization-result}
Let Assumptions~\ref{assumption:avoid-boundary} and~\ref{assumption:weak-convexity} hold.  
Let $G$ and $G_S$ be selections of the subdifferentials $\partial f$ and $\partial f_S$ 
over $x \in \openset$ respectively, i.e., obeying the inclusions $G(x) \in \partial f(x)$ and 
$G_S(x) \in \partial f_S(x)$ for every $x \in \openset$. 
%. This means that 
%$G(x) \in \partial f(x)$ for every $x \in \openset$, and for every sample instance $\xi_1, \xi_2, \ldots, \xi_m$,
%$G_S(x) \in \partial f_S(x)$ for every $x \in \openset$. 
Then, the following inequality holds with probability one: 
\begin{equation*}
	\sup_{x \in \openset \cap X} \H(\partial \phi(x), \partial \phi_S(x))  \le \sup_{x \in \openset}\norm{G(x) - G_S(x)}.
\end{equation*} 
%where $\openset$ is any open set such that $X \subseteq \openset$.
\end{theorem} 

\begin{proof}
Given that $X$ is closed and convex, $\oo_X$ is subdifferentially regular at any $x \in X$~\cite[Exercise 8.14]{RockafellarWe98}. The functions 
$f(x)$ and $f_S(x)$ are real-valued, closed and weakly convex on $\openset$ by Assumptions~\ref{assumption:avoid-boundary}
and~\ref{assumption:weak-convexity}. Thus they are subdifferentially regular on $\openset$~\cite[Corollary 8.11]{RockafellarWe98},
%Furthermore, Assumption~\ref{assumption:avoid-boundary} ensures that with probability one, 
%$X \in \interior(\dom(f))$ and $X \in \interior(\dom(f_S))$ hold.  
with $\partial^\infty f(x) = \partial^\infty f_S(x) = \{0\}$ for all $x \in \openset$. Similarly, the
real-valued, closed, and convex function $r$, is subdifferentially regular 
at $x \in\openset$ with  $\partial^\infty R(x) = \{0\}$.

These regularity properties of the functions enable the application of basic subdifferential 
calculus~\cite[Corollary 10.9]{RockafellarWe98}.  Accordingly, the subdifferentials at $x\in X$ are given by:
\begin{equation*}
	\partial \phi(x) = \partial f(x) + \partial R(x) + \mathcal{N}_X(x),~~~~\partial \phi_S(x) = \partial f_S(x) + \partial R(x) + \mathcal{N}_X(x).
\end{equation*} 
Here, the addition is the Minkowski sum. A fundamental property of Hausdorff distance is that 
\begin{equation*}
	\H(\set_1 + \set_3, \set_2 + \set_3) \le \H(\set_1, \set_2)
\end{equation*}
holds for any sets  $\set_1, \set_2, \set_3$. Consequently, this implies for every $x \in X$: 
\begin{equation*}
	\H(\partial \phi(x), \partial \phi_S(x)) \le \H(\partial f(x), \partial f_S(x)).
\end{equation*}
Since both functions $f(x)$ and $f_S(x)$ are real-valued and locally weakly convex functions on 
$\openset$ due to Assumptions~\ref{assumption:avoid-boundary} and~\ref{assumption:weak-convexity}, 
the following chain of inequalities then follows: 
\begin{equation*}
\begin{split} 
\sup_{x \in \openset \cap X} \H(\partial\phi(x), \partial\phi_S(x))  
	&\le \sup_{x \in \openset \cap X} \H(\partial f(x), \partial f_S(x))\\ 
	&\le
		\sup_{x \in \openset}  \H(\partial f(x), \partial f_S(x)) \le \sup_{x \in \openset}\norm{G(x) - G_S(x)}.
\end{split} 
\end{equation*} 
The last inequality is 
due to Theorem~\ref{theorem:bound-HD-by-MS}, acknowledging that, with probability one, 
$f$ and $f_S$ are real-valued and weakly-convex on the open set $\openset$ by  
Lemma~\ref{lemma:basic-weak-convexity-probability-one-lemma}. This inequality, which involves the 
supremum over uncountable random variables, represents a measurable event with probability one 
due to the completeness of the probability measure on the product space (see the setup on the probability space 
at the beginning of this subsection).
\end{proof} 

\subsection{A Class of Stochastic Weakly Convex Minimization Problems}
\label{sec:a-class-of-stochastic-weakly-convex-minimizations}
One important example of stochastic weakly convex minimization problems corresponds to the 
setting where
\begin{equation}
\label{eqn:convex-compositions}
	f(x, \xi) = \loss(c(x; \xi)).
\end{equation}
Here, $\loss\colon \R^k \to \R$ is a closed, convex and real-valued function, and $c(\cdot, \xi)\colon \R^d \to \R^k$ is 
$\mathcal{C}^1$ smooth on $\R^d$ for every $\xi$. Under this scenario, the population and empirical objectives become
\begin{equation}
\label{eqn:stochastic-convex-composite-objectives}
\begin{split} 
	\min_{x\in \R^d} \phi(x) = f(x) + R(x)+ \oo_X(x) ~~~&\text{where}~~~f(x) = \E_{\xi \sim \P}[\loss(c(x; \xi))] = \int_{\Xi} \loss(c(x; \xi)) P(d\xi)\\
	\min_{x\in \R^d} \phi_S(x) = f_S(x) + R(x) + \oo_X(x) ~~~&\text{where}~~~f_S(x) = \E_{\xi \sim \P_m}[\loss(c(x; \xi))] = \frac{1}{m} \sum_{i=1}^m \loss(c(x;\xi_i))
\end{split}.
\end{equation}
 In the literature, it is known that under mild conditions, 
the convex composition $f(\cdot, \xi) = \loss(c(\cdot; \xi))$ satisfies weak-convexity requirements in Assumption~\ref{assumption:weak-convexity}, 
resulting in locally weakly convex objectives $f(x)$ and $f_S(x)$ that fit into our diagrams in 
Section~\ref{sec:implications-to-stochastic-weakly-convex-minimizations}, 
see, e.g.,~\cite[Claim 1]{DuchiRu18} and~\cite{DavisDr19}. We will 
revisit these conditions subsequently before we present our main results in Theorem~\ref{theorem:main-stochastic-convex-composition-result}. 

We give examples of stochastic convex-composite objectives that appear in 
statistics, machine learning, imaging, and risk management. More examples can be found 
in~\cite[Section 2.1]{DavisDr19}.

%In these applications, 
%$\loss$ typically acts as the loss function, and $c(\cdot; \xi)$ models the data fitting. 
%These objectives can lack both convexity and differentiability, as we illustrate with the examples below.

\vspace{.2cm}

\begin{example}[Robust Phase Retrieval]
\label{example:robust-phase-retrieval}
\emph{Phase retrieval is a computational problem 
with applications across various fields, including imaging, X-ray crystallography, and speech processing.
The (real-valued) phase retrieval seeks to detect a point $x$ satisfying $|\langle a_i, x\rangle|^2 = b_i$
where $a_i \in \R^d$, and $b_i \in \R$ for $i=1,2,\ldots, m$. We can choose $\xi = (a, b) \in \R^d \times \R$, 
$\loss(z) = |z|$ and $c(x; \xi) = (a^T x)^2-b$, and $X = \R^d$, in which case the 
form~\eqref{eqn:stochastic-convex-composite-objectives} gives an exact penalty formulation for solving 
the collection of quadratic equations, which
yields strong statistical recovery and robustness guarantees~\cite{EldarMe14, DuchiRu19, DavisDrPa20},
among other nonconvex formulations~\cite{CandesLiSo15, WangGiEl17, SunQuWr18}.
}
\end{example} 

\vspace{.2cm}

\newcommand{\tr}{{\rm tr}}
\begin{example}[Robust Matrix Sensing]
\label{example:robust-matrix-sensing}
\emph{This problem can be viewed as a variant of phase retrieval. Let $A_1, A_2, \ldots, A_{m} \in \R^{D\times D}$
be measurement matrices. Given the measurement $b_i = \langle A_i, M_\sharp\rangle + \eta_i \in \R$, where $M_\sharp \in \R^{D \times D}$
is the true matrix, and $\eta_i \in \R$ is the noise corruption, the goal is to recover a low-rank approximation 
of $M_\sharp$, which is modeled through $XX^T$ where $X \in \R^{D \times r_0}$ where $1 \le r_0 \le D$. 
Recently, there has been a series of efforts showing that the following potential function has strong stability 
guarantees under certain assumptions on $A$ and $\eta$~\cite{LiZhMaVi20, CharisopoulosChDaDiDiDr21, DingJiChQuZh21}:
\begin{equation*}
	\min_{X \in \R^{d \times r}} \frac{1}{m}\sum_{i=1}^m |\langle A_i, XX^T \rangle - b_i |.
\end{equation*}
This falls into the form~\eqref{eqn:stochastic-convex-composite-objectives} by simply setting $x = X \in \R^{D \times r_0}$, 
$\xi = (A, b)$ where $A \in \R^{D\times D}, b\in \R$, $\loss(z) = |z|$ and $c(x; \xi) = c(x; (A,b)) = \langle A, XX^T \rangle - b$.
} 
\end{example} 

\vspace{.2cm}

\begin{example}[Robust Blind Deconvolution]
\emph{
Blind deconvolution aims to recover two vectors $x_1, x_2 \in \R^d$ from the observed measurements  
$b_i = \langle u_i, x_1\rangle \langle v_i, x_2\rangle$ from $i = 1, 2, \ldots, n$ where $u_i, v_i \in \R^d$
are known measurement vectors.  This problem has applications in fields like astronomy and computer vision.
A recent formulation for recovering $x_1$ and $x_2$ minimizes the reconstruction error~\cite{CharisopoulosDaDiDr21, Diaz19}: 
\begin{equation*}
	\min_{x_1 \in \R^d, x_2 \in \R^d}  \frac{1}{m} \sum_{i=1}^m |b_i - \langle u_i, x_1\rangle \langle v_i, x_2\rangle|
\end{equation*}
This falls into the form~\eqref{eqn:stochastic-convex-composite-objectives} by simply setting $x = (x_1, x_2)$, 
$\xi = (u, v, b)$ where $u, v \in \R^{d}, b\in \R$, $\loss(z) = |z|$ and $c(x; \xi) = c((x_1, x_2); (u, v, b)) = b - \langle u, x_1\rangle \langle v, x_2 \rangle$.
}
\end{example}

\vspace{.2cm}
\begin{example}[Conditional Value-at-Risk]
\emph{
Let $\ell(w, \xi)$ represent a decision rule parameterized by $w \in \R^d$ for a data point $\xi$, where $\xi \sim \P$.
Instead of minimizing the expected value $\E_{\xi \sim \P}[\ell(w, \xi)]$, it is often preferable to minimize the 
conditional expectation of the random variable $\ell(w, \cdot)$ over its $\alpha$-tail, where $\alpha \in (0, 1)$ is given.  
This quantity is termed the Conditional Value-at-Risk \nolinebreak (cVaR)~\cite[Section 6]{ShapiroDeRu21}.  
Remarkably,  a seminal work shows that minimizing cVaR can be formulated as~\cite{RockafellarUr00}: 
\begin{equation*}
	\min_{\gamma \in \R, w \in \R^d} \E_{\xi \sim P} [(1-\alpha) \gamma + (\ell(w, \xi) - \gamma)_+].
\end{equation*} 
Now we suppose $\ell(\cdot, \xi)$ is $\mathcal{C}^1$ smooth. 
This falls into the form~\eqref{eqn:stochastic-convex-composite-objectives} by setting $x = (\gamma, w) \in \R \times \R^d$, 
$\loss(z) = (z)_+$, $c(x; \xi) =  c((\gamma, w); \xi) = \ell(w, \xi) - \gamma$, and $R(x) = R((\gamma, w)) = (1-\alpha)\gamma$.
}
\end{example}

%In many applications, convex-composite objectives are neither convex nor differentiable. 
%Despite their widespread use, important questions remain open regarding tight 
%bounds on the difference between subdifferentials of the empirical and population objectives. 
%This gap limits our understanding of the generalization gap between empirical and population 
%risk solutions. 
%Since we only have access to finite sample data, and our algorithms are only guaranteed to 
%converge to stationary points, not the global minimum, in these nonconvex scenarios, establishing 
%tight bounds on subdifferentials is critical for accurately assessing our algorithm performances.

%Despite their widespread use, important questions remain open regarding tight 
%bounds on the difference between subdifferentials of the empirical and population objectives. 

We now revisit the conditions on the convex function $\loss$ and smooth function 
$c(x, \xi)$ that guarantees the weak-convexity assumptions for the composition 
$f(x, \xi) = \loss(c(x, \xi))$, ensuring the regularity required for 
subdifferential calculus rules. 
We take the following conditions---local Lipschitzian and integrability conditions on $\loss, c$---from~\cite[Section 2.2]{DuchiRu18}. 

%These conditions are sufficient to guarantee that 
%$f(x) = \E[f(x; \xi)] = \E[\loss(c(x, \xi))]$ and $f_S(x) = \E[f(x; \xi)]$ are 
%locally weakly convex on $\R^d$, obeying the chain rule when taking their subdifferentials. 
\renewcommand{\theassumption}{\Alph{assumption}'}
\setcounter{assumption}{0}

\vspace{.2cm}
\begin{assumption}
\label{assumption:avoid-boundary-cc}
$f(x) = \int_{\Xi} \loss(c(x; \xi)) P(d\xi)< \infty$ for every $x \in \openset$ where $\openset$ is an open set. 
\end{assumption}

\vspace{.2cm}
\begin{assumption}
\label{assumption:weak-convexity-for-composition}
For every $x \in \openset$, there is $\eps(x) > 0$ such that 
\begin{itemize}
\item $\sup_{y \in \B(x; \eps(x))}\norm{\grad c(y; \xi)\partial \loss(c(y; \xi))} \le L_x(\xi)$ for some random variable $L_x(\xi)$ that is measurable, and integrable with respect to $\xi \sim \P$. 
\item $\norm{\grad c(y; \xi) - \grad c(y'; \xi)} \le \beta_\eps(x, \xi) \norm{y-y'}$ 
for $y, y' \in \B(x; \eps(x))$, where $\E[\beta_\eps(x, \xi)] < \infty$.
\end{itemize}
%$c(\cdot, \xi)$ has $\beta_\eps(\cdot; \xi)$-Lipschitz gradients 
%within an $\eps$ neighborhood around $x$---meaning that $\norm{\grad c(y; \xi) - \grad c(y'; \xi)} \le \beta_\eps(x, \xi) \norm{y-y'}$ 
%holding for $y, y' \in x+\eps\B$---and if $\E[\beta_\eps(x, \xi)] < \infty$
\end{assumption} 

\vspace{.2cm}

Given Assumptions~\ref{assumption:avoid-boundary-cc} and~\ref{assumption:weak-convexity-for-composition}, the following basic property
on the stochastic convex-composite objectives follows from~\cite[Claim 1+Lemma 3.6]{DuchiRu18}. We use 
$\grad c(x; \xi)$ to denote the gradient of a smooth function $c$ with respect to $x$ for a given $\xi$. 

\vspace{.5cm}
\begin{lemma}
\label{lemma:regularity-convex-composite}
Let Assumptions~\ref{assumption:avoid-boundary-cc} and~\ref{assumption:weak-convexity-for-composition} hold.
Then 
$f(x; \xi) = \loss(c(x; \xi))$ satisfies Assumption~\ref{assumption:weak-convexity}.

Moreover, 
\begin{equation*}
	\partial f(x; \xi) = \partial \loss(c(x; \xi)) \grad c(x; \xi),
\end{equation*}
 with 
\begin{equation*}
\begin{split}
	\partial f(x) = \E_{\xi \sim \P}[\partial f(x; \xi)],~~~\text{and}~~~\partial f_S(x) = \E_{\xi \sim \P_m}[\partial f(x; \xi)].
\end{split} 
\end{equation*}
\end{lemma}

For ease of reference in future discussions, we establish a result that stems directly from 
Theorem~\ref{theorem:main-stochastic-minimization-result} and Lemma~\ref{lemma:regularity-convex-composite}
for stochastic convex-composite minimization problems. 

\vspace{.5cm}
\begin{theorem}
\label{theorem:main-stochastic-convex-composition-result}
Let Assumptions~\ref{assumption:avoid-boundary-cc} and~\ref{assumption:weak-convexity-for-composition} hold
for $f(x, \xi) = \loss(c(x; \xi))$. Suppose at every $x \in \openset$, the functions $G(x)$ and $G_S(x)$ obey:  
\begin{equation*}
	G(x) \in \E_{\xi \sim \P}[\partial f(x; \xi)]~~~\text{and}~~~G_S(x) \in \E_{\xi \sim \P_m}[\partial f(x; \xi)].
\end{equation*} 
Then, the following inequality holds with probability one: 
\begin{equation*}
	\sup_{x \in X \cap \openset} \H(\partial \phi(x), \partial \phi_S(x)) \le \sup_{x \in \openset}\norm{G(x) - G_S(x)}.
\end{equation*} 
In the above, $\phi(x)$ and $\phi_S(x)$ refer to the convex-composite objectives in 
equation~\eqref{eqn:stochastic-convex-composite-objectives}. 
\end{theorem} 

\section{Explicit Uniform Convergence Rates} 
\label{sec:explicit-uniform-convergence-rates}
In this section, we will delve deep into the convergence rate of the subdifferentials 
\begin{equation}
\label{eqn:explicit-target}
	\sup_{x \in X\cap \openset} \H(\partial \phi(x), \partial \phi_S(x))
\end{equation}
for a class of stochastic convex-composite objectives where 
\begin{equation}
	\phi(x) = \E_{\xi \sim \P}[\loss(c(x; \xi))] + R(x) + \oo_X(x),~~~~\phi_S(x) = \E_{\xi \sim \P_m}[\loss(c(x; \xi))] + R(x) + \oo_X(x).
\end{equation}
Here, the crucial assumption is that $\loss \colon \mathbb{R} \to \mathbb{R}$ is a one-dimensional convex function.
The function $c(\cdot; \xi) \colon \mathbb{R}^d \to \mathbb{R}$ is $\mathcal{C}^1$ smooth on $\R^d$ for every $\xi$, 
aligning with the framework in~Section~\ref{sec:a-class-of-stochastic-weakly-convex-minimizations}.  As before,
$R \colon \mathbb{R}^d \to \mathbb{R}$ is closed and convex, and $X$ is nonempty, closed and convex.
This setting is particularly relevant as it covers a broad spectrum of practical objectives, including all 
examples previously discussed in 
Section~\ref{sec:a-class-of-stochastic-weakly-convex-minimizations}. In our considerations, the set 
$\openset$ is usually an open ball $\B(x_0; r)$ in the Euclidean space where $x_0 \in \R^d$ is the center and 
$r > 0$ is the radius. 

Despite their wide applications, tight bounds on the subdifferential difference of these objectives 
remain largely open in the literature~\cite[Section 5]{DavisDr22}, which motivates our investigations. 
We will show in Section~\ref{sec:applications} how our main result in this section, Theorem~\ref{theorem:final-result}, 
yields tight bounds on the subdifferential differences, and leads to 
sharp quantitative characterizations on the nonsmooth landscape of stochastic convex-composite formulations in finite samples. 

%Nonetheless, tight bounds that uniformly bound the subdifferential difference, namely the quantity in equation~\eqref{eqn:explicit-target}, 
%between the empirical and population objectives remain largely open in the literature, which motivates our investigations.

%In Section~\ref{sec:measure-of-discontinuities-function-class-complexities}, 
%we will first set up the notation needed for our main results, which include a measure on the 
%discontinuities of the convex function $\loss$, and a measure of the complexities of the smooth function classes 
%$\{x \mapsto c(x; \xi) \mid x \in \R^d\}$ that 
%appear in our rate upper bounds. Section~\ref{sec:main-result-uniform-convergence-rates} will present 
%our main results on the rates of uniform convergence of subdifferentials. Possible extensions to certain higher-dimensional 
%scenarios, such as $\loss$ being polyhedral functions will not be explored in this paper to maintain focus 
%and clarity. See Section~\ref{sec:discussions} for the discussion. 

\subsection{Subgradient Selections}
Theorem~\ref{theorem:main-stochastic-convex-composition-result} shows that, under regularity 
Assumptions~\ref{assumption:avoid-boundary-cc} and~\ref{assumption:weak-convexity-for-composition}, 
there is the bound: 
\begin{equation}
\label{eqn:bound-basic}
	\sup_{x \in X\cap \openset} \H(\partial \phi(x), \partial \phi_S(x)) \le \sup_{x \in \openset}\norm{G(x) - G_S(x)}
\end{equation}
where $\openset$ is an open set. 
Remarkably, $G$ and $G_S$ can be any pair of selections from the subdifferentials (cf. Lemma~\ref{lemma:regularity-convex-composite}
and Theorem~\ref{theorem:main-stochastic-convex-composition-result}):  
\begin{equation}
\label{eqn:requirements-for-subgradient-selections} 
\begin{split} 
	G(x) \in \E_{\xi \sim P}[\partial \loss(c(x; \xi)) \grad c(x; \xi)], ~~~~% ~~= \E_{\xi \sim P}[\partial(\loss(c(x; \xi))\grad c(x; \xi)], \\
	G_S(x) &\in \E_{\xi \sim P_m}[\partial \loss(c(x; \xi)) \grad c(x; \xi)]. %= \E_{\xi \sim P_m}[\partial(\loss(c(x; \xi))\grad c(x; \xi)].
\end{split} 
\end{equation}	
%where we recall 
%\begin{equation*}
%	\partial \loss(c(x; \xi)) = \partial \loss(c(x; \xi)) \grad c(x; \xi).
%\end{equation*} 
The main goal of this subsection is to select a pair of subgradients $G(x)$ and $G_S(x)$
that are simple to analyze---from a probabilistic sense---regarding their uniform differences over an open set $\openset$. 

This selection process begins with identifying a subgradient for the nonsmooth convex function $\loss$, which we decompose 
into a sum $\loss = \loss\sm+\loss\ns$ of a smooth component $\loss\sm$ and a nonsmooth component $\loss\ns$. To do so, we use
$\loss^\prime_+$ and $\loss^\prime_-$ to denote the right-hand and left-hand derivative of $\loss$, respectively. 
We enumerate the set of nondifferentiable points $\{t \in \R: \loss^\prime_+(t) \neq \loss^\prime_-(t)\}$ as $\{t_j\}_{j =1}^\infty$,
given the fact that these nondifferentiable points are always countable~\cite[Theorem 2.1.2]{BorweinVa10}.

For the majority of Section~\ref{sec:explicit-uniform-convergence-rates}, 
we will work under the assumption that $\sum_{j=1}^\infty (\loss^\prime_+(t_j) - \loss^\prime_-(t_j)) < \infty$
so that $\loss\ns$ and $\loss\sm$ are well-defined,
although this assumption will be unnecessary in our final bound in Theorem~\ref{theorem:final-result}. 

Under this assumption, we can decompose this nonsmooth $h$ into two distinct parts:
%The function $x \mapsto c(x; \xi)$ is $\mathcal{C}^1$ smooth almost surely for $\xi \sim \P$. 
%The function $\loss:  \R \mapsto \R$ is convex and comprises two distinct parts: 
\begin{equation}
	\loss(z) = \loss\sm(z) + \loss\ns(z).
\end{equation} 
Here, $\loss\ns \colon  \R  \to \R$ (where ``ns" stands for nonsmooth) denotes the nonsmooth component of $\loss$, which 
is defined by: 
\begin{equation}
\label{eqn:loss-ns-definition}
	\loss\ns(z) = \sum_{j=1}^\infty a_j \cdot (z-t_j)_+,
\end{equation}
where the coefficients are defined by
$a_j = \loss^\prime_+(t_j) - \loss^\prime_-(t_j) > 0$. Notably, $\loss\ns$ is well-defined because  
$\sum_{j=1}^\infty a_j = \sum_{j=1}^\infty (\loss^\prime_+(t_j) - \loss^\prime_-(t_j)) < \infty$.
The smooth component $\loss\sm: \R \to \R$ (where ``sm" stands for smooth) is defined as 
$\loss\sm = \loss - \loss\ns$. By construction, $\loss\sm$ is a $\mathcal{C}^1$ smooth and convex function, 
as one can easily verify its derivative $(\loss\sm)'$ is continuous and monotonically increasing on $\R$. 

To select a subgradient of the convex function $\loss$, we first note that the smooth component $\loss\sm$ possesses a unique subgradient, which 
corresponds to its derivative:
\begin{equation}
	g\sm(z) = (\loss\sm)^\prime(z).
\end{equation}
For the nonsmooth component $\loss\ns$, a subgradient at every $z \in \R$ is chosen as follows: 
\begin{equation}
\label{eqn:g-ns-part}
	g\ns(z) = \sum_{j=1}^\infty a_j \mathbf{1}(z \ge t_j) \in \partial \loss\ns(z),
\end{equation}
where $z \mapsto \mathbf{1}(z \ge t_j)$  is an indicator function that is a subgradient for the mapping $z \mapsto (z-t_j)_+$, 
taking the value $1$ when $z \ge t_j$ and $0$ otherwise. Combining these, we define a subgradient $g$ for $\loss$ at every $z$: 
\begin{equation*}
	g(z) := (g\sm + g\ns) (z) \in \partial \loss(z).
\end{equation*} 
With $g(z)$ defined, we can then pick the subgradients $G$ and $G_S$ that satisfy equation~\eqref{eqn:requirements-for-subgradient-selections}:
\begin{align}
\begin{split}
	G(x)  &:=\E_{\xi \sim P}\left[g(c(x, \xi)) \cdot \grad c(x, \xi)\right] \\
	G_S(x) &:=\E_{\xi \sim P_m}\left[g(c(x, \xi)) \cdot \grad c(x, \xi)\right]
\end{split}. 
\end{align} 
To summarize, $G$ and $G_S$ are well-defined functions under the condition $\sum_{j=1}^\infty (\loss^\prime_+(t_j) - \loss^\prime_-(t_j)) < \infty$, provided that regularity Assumptions~\ref{assumption:avoid-boundary-cc} and~\ref{assumption:weak-convexity-for-composition} are also satisfied.

In the next subsection, we will provide a uniform convergence result that establishes a high probability upper bound on 
$\norm{G_S(x)- G(x)}$ over $x$ in a Euclidean ball in $\R^d$. We will address the following question. 

\vspace{.1cm}
\textbf{Question}: Given a center $x_0 \in \R^d$, a radius $r$, a threshold $\delta \in (0, 1)$, what values of $u$ can we choose such that
\begin{equation}
\label{eqn:statistical-learning-question}
	\P^*\left(\sup_{x: x\in \B(x_0; r)} \norm{G_S(x) - G(x)} \ge u\right) \le \delta.  
\end{equation}
Answers to this question will automatically yield a high probability upper bound on the subdifferential difference 
$\sup_{x\in X \cap\B(x_0; r)}\H(\phi(x), \phi_S(x))$ by applying equation~\eqref{eqn:bound-basic} to the open set 
$\openset = \B(x_0; r)$. Recall $\P^*$ denotes the outer probability (Section~\ref{sec:probability-measure-and-measurability}), 
which is required because the supremum in equation~\eqref{eqn:statistical-learning-question} may not be measurable
under the probability space.

\subsection{A Uniform Bound on Subgradient Selections}
\label{sec:main-result-uniform-convergence-rates}
This section examines the uniform rate of convergence of $G_S$ to $G$ over any Euclidean ball of radius $r$, 
a question that is arised in equation~\eqref{eqn:statistical-learning-question}.  Recall our definitions: 
\begin{equation}
\label{eqn:detail-definition-of-G-S-G}
\begin{split} 
	G_S(x) %&=\E_{\xi \sim P_m}\bigg[g\ns(c(x, \xi)) \cdot \grad c(x, \xi)\bigg] + \E_{\xi \sim P_m}\bigg[g\sm(c(x, \xi)) \cdot \grad c(x, \xi)\bigg]\\
	&= \E_{\xi \sim P_m}\bigg[\sum_{j=1}^\infty a_j \one\{c(x, \xi) \ge t_j\} \cdot \grad c(x, \xi)
	 +  (\loss\sm)^\prime(c(x, \xi)) \cdot \grad c(x, \xi)\bigg] \\
	G(x) %&=\E_{\xi \sim P_m}\bigg[g\ns(c(x, \xi)) \cdot \grad c(x, \xi)\bigg] + \E_{\xi \sim P_m}\bigg[g\sm(c(x, \xi)) \cdot \grad c(x, \xi)\bigg]\\
	&= \E_{\xi \sim P}~\bigg[\sum_{j=1}^\infty a_j \one\{c(x, \xi) \ge t_j\} \cdot \grad c(x, \xi)
	+  (\loss\sm)^\prime(c(x, \xi)) \cdot \grad c(x, \xi)\bigg]
\end{split}.
\end{equation}

A key challenge in examining the rate of uniform convergence from $G_S$ to $G$ 
is due to the nonsmooth nature of $G_S$ and $G$,  whose definition involves indicator functions, as outlined in equation
\eqref{eqn:detail-definition-of-G-S-G}. To resolve this challenge, our approach 
leverages the notion of Vapnik–Chervonenkis dimension (VC dimension)~\cite{VapnikCh71}, a concept familiar to experts in statistical learning theory.
In short words, VC dimension measures the complexity of sets used in defining these indicator functions (termed as classifiers in statistical learning theory), 
and this complexity measure governs the uniform generalizability from empirical averages to 
expectations for these indicator functions.

We state the definition of VC dimension for completeness~\cite[Section 3.6]{Vapnik13}. 

\vspace{.2cm}

\newcommand{\VC}{{\rm vc}}
\newcommand{\F}{\mathcal{F}}

\begin{definition}[VC Dimension]
Let $\mathcal{H}$ denote a set family, and $E$ a set. Let
$\mathcal{H} \cap E = \{H \cap E \mid H \in \mathcal{H}\}$. We say a set $E$ is \emph{shattered} if 
$\mathcal{H} \cap E$ contains all the subsets of $E$, i.e., $|\mathcal{H} \cap E| = 2^{|E|}$. The VC 
dimension $\VC(\mathcal{H})$ is the cardinality of the largest set that is \emph{shattered} by $\mathcal{H}$.
If arbitrarily large sets can be shattered, the VC dimension of $\mathcal{H}$ is $\infty$.
\end{definition}

The collection of sets that appear in the indicator functions will be denoted as: 
\begin{equation}
\label{eqn:indicator-set-class-F}
	\F = \{\{\xi \in \Xi \colon c(x, \xi) \ge t\} \mid x \in \R^d, t\in \R\}.
\end{equation}
More transparently, the VC dimension of $\F$, denoted by $\VC(\F)$, is the largest integer $N$ such that there exist
$\xi_1, \xi_2, \ldots, \xi_N \in \Xi$ so that for any binary labeling sequence
$b_1, b_2, \ldots, b_N \in \{0, 1\}$, it is possible to find parameters $x \in \R^d, t \in \R$ that satisfy: 
\begin{equation*}
	b_i = \one \{c(x; \xi_i) \ge t\}~~\text{for every $1 \le i\le N$}.
\end{equation*}
%The VC dimension, as a complexity measure, characterizes the generalization capabilities of the indicators~\cite{Vapnik13}. 
%In statistical learning theory, the VC dimension measures the generalization capabilities of a set of classifiers~\cite{Vapnik13}. 
%In our context, $\VC(\F)$ plays a key role in establishing a high-probability upper bound on the deviation 
%between the finite sample objective $G_S$ and the population objective $G$. 
The complexity measure $\VC(\F)$ characterizes generalization capabilities of $\F$ and bounds the uniform rate of convergence from 
$G_S$ to $G$. For common techniques of upper bounding VC dimensions of set families, see~\cite[Chapter 2]{VanDerVaartWe96}. 
When the mapping $x \mapsto c(x; \xi)$ are polynomials in $x \in \R^d$ for any given $\xi$---a condition met in our applications---we 
establish a result giving a tight upper bound of the VC dimension, based on results in real algebraic geometry. 
See Section~\ref{sec:VC-dimension-bounds} for details.

%A basic property of VC dimension is: for any integer $m \ge \VC(\F)$, 
%and any data $\xi_1, \xi_2, \ldots, \xi_m$,  consider the set of all possible binary label sequences that can be 
%generated by the model for these $m$ data points: 
%\begin{equation*}
%	\{b = (b_1, b_2, \ldots, b_m)|~~b_i =  \one \{c(x; \xi_i) \ge t\}~~\text{for every $1\le i\le N$}\} \subseteq \{0, 1\}^N
%\end{equation*} 
%has cardinality at most $(N+1)^{\VC(\F)}$. 

%denoted by $\VC(\F)$, which controls the complexity of these sets, 
%will appear in our high probability upper bound.

Our main results rely on Lipschitz continuity and regularity assumptions concerning the tails of 
the random vectors $\nabla c(x, \xi)$ and $(\loss\sm)'(c(x, \xi)) \cdot \nabla c(x, \xi)$. These random 
vectors appear in our subgradient selection~\eqref{eqn:detail-definition-of-G-S-G}, and 
our assumptions about them align with standard treatments in the literature. To state the assumptions, we 
first recall the 
concept of a subexponential tail random vector. We define this in terms
of Orlicz norms, following~\cite[Section 3.4.4]{Vershynin18}.

\vspace{.2cm}
%concept of a subexponential tail random vector which requires that the tail 
%distribution of any projection of the random vector decays subexponentially. Our definition below follows
%standard reference, see, e.g.,~\cite[Section 3.4.4 and Section 2.7]{Vershynin18}.

\begin{definition}
\label{definition:sub-exponential}
A random vector $Z$ is said to be $\sigma$-subexponential if for every $v$ with $\norm{v} = 1$: 
\begin{equation*}
	\E\left[\exp\left(\frac{|\langle Z, v\rangle|}{\sigma}\right)\right] \le 2. 
	%\P(|\langle Z, v\rangle| \ge \sigma \cdot u) \le 2e^{-u}~~~\text{for all $u \ge 0$}.
\end{equation*}
\end{definition} 
Examples of $\sigma$-sub-exponential random vectors include a normal distribution $Z \sim \normal(0, \Sigma)$
where the largest eigenvalue of $\Sigma$ is bounded above by $c\sigma$, or distributions with bounded support, 
such as $\norm{Z} \le c\sigma$ almost surely, with the numerical constant $c=1/2$~\cite[Section 2]{Vershynin18}. 

By Markov's inequality,  $\P(|\langle v, Z\rangle| \ge \sigma \cdot u) \le 2 e^{-u}$ for any unit vector $v$ and $u \ge 0$.
Furthermore, by Bernstein's inequality~\cite[Theorem 2.8.1]{Vershynin18} and the centering property of sub-exponential 
random variables~\cite[Exercise 2.7.10]{Vershynin18}, 
for i.i.d. $\sigma$-subexponential random variables $Z_1, Z_2, \ldots, Z_m$,
there is the \emph{exponential} tail bound for the mean $\frac{1}{m} \sum_{i=1}^m Z_i$ 
that holds for every $v$ with $\norm{v} = 1$, and $u \ge 0$: 
\begin{equation}
\label{eqn:sub-exponential-immediate}
	\P\left(\frac{1}{m}\left|\sum_{i=1}^m \langle Z_i - \E[Z], v \rangle\right| \ge \sigma \cdot u\right) \le 2\exp\left(- c m\min\{u^2, u\}\right).
\end{equation}
In the above, $c > 0$ is an absolute constant. 

%By a standard result~\cite[Proposition 2.7.1]{Vershynin18}, there is a universal constant $C > 0$ such that 
%if $Z$ is $\sigma$-exponential, then the moment generating function 
%$\E[\exp(|\langle Z, v\rangle|/(C\sigma))] \le 2$ holds for every $v$ with $\norm{v} = 1$.  Conversely, if 
%$\E[\exp(|\langle Z, v\rangle|/(C\sigma))] \le 2$, then Markov's inequality implies that 
%$\P(|\langle Z, v\rangle| \ge  C \sigma \cdot u) \le 2e^{-u}$ for all $u \ge 0$.

We first assume the random vectors $\nabla c(x_0, \xi)$ and $(\loss\sm)'(c(x_0, \xi)) \cdot \nabla c(x_0, \xi)$ are subexponential, 
where we recall that $x_0$ is the center of the ball of interest (see equation~\eqref{eqn:statistical-learning-question}). 
This facilitates an initial \emph{exponential} tail probability bound on the difference between $G_S(x_0)$ and $G(x_0)$. 

\vspace{.2cm}
\renewcommand{\theassumption}{C.\arabic{assumption}}
\begin{assumption}
\label{assumption:sub-exponential-assumption}
Each of the two random vectors
%$\nabla c(x_0, \xi)$ and $(\loss\sm)'(c(x_0, \xi)) \cdot \nabla c(x_0, \xi)$ is $\sigma$-subexponential.
%For every fixed $x \in \B(x_0; r)$, each of the random vectors 
\begin{equation*}
\text{$\nabla c(x_0, \xi)$ ~~and~~ $(\loss\sm)'(c(x_0, \xi)) \cdot \nabla c(x_0, \xi)$}
\end{equation*} 
is $\sigma_0$-subexponential.
\end{assumption} 

To further ensure a uniform \emph{exponential} tail probability bound on the difference between  $G_S(x)$ and $G(x)$
over all possible $x \in \B(x_0; r)$, we will assume that the increments of the random processes $\{\grad c(x, \xi)\}_{x \in \B(x_0; r)}$ and 
$\{(\loss\sm)'(c(x, \xi))\nabla c(x, \xi)\}_{x \in \B(x_0; r)}$ are subexponential in the following sense. 

\vspace{.2cm}
\begin{assumption}
\label{assumption:Lipschitz-condition}
For every $x_1, x_2 \in \B(x_0; r)$, each of the two random vectors 
\begin{equation*}
\text{$\grad c(x_1, \xi) - \grad c(x_2, \xi)$ ~~and~~ $e(x_1, \xi) - e(x_2, \xi)$}
\end{equation*} 
is $\sigma \norm{x_1 - x_2}$ subexponential. In the above, $e(x, \xi) = (\loss\sm)'(c(x, \xi)) \nabla c(x, \xi)$ for every $x \in \B(x_0; r)$. 
%There is a measurable function $L:\Xi \to \R$ with $L_* = \E_{\xi \sim \P}[L(\xi)] < \infty$ so that
%\begin{equation*}
%\begin{split} 
%	\norm{\grad c(x_1, \xi) - \grad c(x_2, \xi)} &\le \sigma \norm{x_1-x_2} \\
%	\norm{e(x_1, \xi) - e(x_2, \xi)} &\le  L(\xi) \norm{x_1-x_2}
%	%\norm{(\loss\sm)'(c(x_1, \xi)) \nabla c(x_1, \xi) - (\loss\sm)'(c(x_2, \xi)) \nabla c(x_2, \xi)} \le L(\xi) \norm{x_1-x_2}
%\end{split} 
%\end{equation*}
%hold for every $x_1, x_2 \in \B(x_0; r)$, where $e(x, \xi) = (\loss\sm)'(c(x, \xi)) \nabla c(x, \xi)$.
%In words, both mappings $x \mapsto \grad c(x, \xi)$ and $x \mapsto (\loss\sm)'(c(x, \xi))\nabla c(x, \xi)$ 
%are $L(\xi)$ Lipschitz on xxx
\end{assumption} 
This assumption is commonly used in empirical process theory to extend high probability bounds from individual points to 
the uniform control of the supremum of stochastic processes, utilizing the chaining technique. 
Notably, this assumption naturally holds in a variety of 
statistical applications. For a reference, see~\cite[Chapter 2]{VanDerVaartWe96}.

We are now ready to present our main result. The proof of Theorem~\ref{theorem:basic-statistical-learning-result}
employs empirical process theory tools, namely, the calculus rules for outer probability 
measure~\cite[Chapter 1.2-5]{VanDerVaartWe96}, the chaining 
argument~\cite{Dudley67}, and the Sauer-Shelah Lemma---a fundamental combinatorial principle using the 
VC dimension to control function class complexities~\cite{Sauer72, Shelah72}. 
The proof is relatively standard for experts in statistical learning theory 
and is detailed in Section~\ref{sec:proof-theorem-basic-statistical-learning-result}. 

\vspace{.5cm}
\begin{theorem}
\label{theorem:basic-statistical-learning-result}
Assume Assumptions~\ref{assumption:sub-exponential-assumption} and~\ref{assumption:Lipschitz-condition}. 
Assume $\sum_{j=1}^\infty (\loss^\prime_+(t_j) - \loss^\prime_-(t_j)) < \infty$.

There exists a universal constant $C > 0$ such that for every $\delta \in (0, 1)$: 
\begin{equation}
\label{eqn:uniform-G-G-S-result}
	\P^* \left(\sup_{x:x \in \B(x_0; r)} \norm{G(x) - G_S(x)} \ge C(\sigma_0+\sigma r)\zeta\cdot 
		\max\{\Delta, \Delta^2\}\right) \le \delta.
\end{equation}
In the above, 
\begin{equation}
	\zeta = 1+ \sum_{j=1}^\infty (\loss_+^\prime(t_j) - \loss_-^\prime(t_j))
\end{equation}
 where $t_j$ enumerate points where $\loss$ is non-differentiable, and 
\begin{equation}
\label{eqn:error-definition}
	\Delta =  \sqrt{\frac{1}{m} \cdot \Big(d+\VC(\F)\log m +\log( \frac{1}{\delta})\Big)},
\end{equation} 
where $\VC(\F)$ is the VC dimension of the set family $\F$ defined in equation~\eqref{eqn:indicator-set-class-F}. 
\end{theorem} 

Recall $\P^*$ denotes the outer probability (Section~\ref{sec:probability-measure-and-measurability}), also 
used to address measurability issues with supremums in empirical process theory~\cite{VanDerVaartWe96}. 
This is technically required because the supremum in equation~\eqref{eqn:uniform-G-G-S-result} may not be 
measurable under the probability space.

Theorem~\ref{theorem:basic-statistical-learning-result} establishes a uniform closeness between the selected subgradient 
$G_S$ and $G$ over the Euclidean ball $\B(x_0; r)$. In the final bound~\eqref{eqn:uniform-G-G-S-result},  %$\sigma$ captures 
%the size of the fluctuations, and 
$\zeta$ captures how the nondifferentiability of $\loss$ could potentially deteriorate
the convergence rate. Crucially, the error term $\Delta$ in equation~\eqref{eqn:error-definition} demonstrates how the dimensionality $d$ and 
the complexity measure $\VC(\F)$ of the set family $\F$ influence the uniform convergence rate.

 The presence of $\max\{\Delta, \Delta^2\}$ in the final bound 
instead of just $\Delta$ stems from the sub-exponential assumption  (cf. $\min\{u, u^2\}$ in the Bernstein 
inequality~\eqref{eqn:sub-exponential-immediate}). When $\Delta \le 1$, which is typical for large sample 
size $m$, this term simplifies to $\Delta$, which is $\sqrt{(d+\VC(\F))/m}$ up to logarithmic factors. Our convergence rates
align well with the standard uniform convergence rate $\sqrt{d/m}$ for averages of smooth functions and 
$\sqrt{\VC(\F)/m}$ for indicator functions in the statistics literature~\cite{Wainwright19}. 
This suggests our bounds are in general tight up to logarithmic factors in $m$.

Finally, the $\log m$ term arises from the subexponential tail assumption and can be removed if a norm boundedness 
assumption replaces it in Assumption~\ref{assumption:Lipschitz-condition}, or by using 
more advanced arguments involving additional probabilistic control of the envelope function of 
$\grad c(x, \xi)$ and $(\loss\sm)'(c(x, \xi)) \cdot \nabla c(x, \xi)$~\cite[Chapter 2]{VanDerVaartWe96}. For conciseness and clarity, 
this paper does not pursue such removal.

\subsection{Implication: A Uniform Bound on Subdifferentials} 
We integrate Theorem~\ref{theorem:main-stochastic-convex-composition-result} and 
Theorem~\ref{theorem:basic-statistical-learning-result} to document our ultimate result, Theorem~\ref{theorem:final-result}, on 
uniform convergence of subdifferentials.  At this point, the only work is to collect and streamline all the assumptions in 
Theorem~\ref{theorem:main-stochastic-convex-composition-result} and~\ref{theorem:basic-statistical-learning-result}. To fulfill
Assumption~\ref{assumption:weak-convexity-for-composition} required in Theorem~\ref{theorem:main-stochastic-convex-composition-result}, 
we identify an additional condition beyond 
Assumptions~\ref{assumption:sub-exponential-assumption}---\ref{assumption:Lipschitz-condition}. This Lipschitz-type 
condition is very mild, and holds in many applications in stochastic programming~\cite[Chapter 5]{ShapiroDeRu21} and statistics~\cite[Chapter 5]{Van00}.
%This is Assumption~\ref{assumption:real-Lipschitz-condition}, 
%a mild Lipschitz-type assumption common  in 
%stochastic programming~\cite[Chapter 5]{ShapiroDeRu21} and statistics~\cite[Chapter 5]{Van00}.
 
\vspace{.2cm} 
 
\begin{assumption}
\label{assumption:real-Lipschitz-condition}
There is a measurable function $L:\Xi \to \R$ with $\E_{\xi \sim \P}[L(\xi)] < \infty$ such that 
\begin{equation*}
\begin{split} 
	\norm{\grad c(x_1, \xi) - \grad c(x_2, \xi)} &\le L(\xi) \norm{x_1-x_2} \\
	\norm{e(x_1, \xi) - e(x_2, \xi)} &\le  L(\xi) \norm{x_1-x_2}
	%\norm{(h\sm)'(c(x_1, \xi)) \nabla c(x_1, \xi) - (h\sm)'(c(x_2, \xi)) \nabla c(x_2, \xi)} \le L(\xi) \norm{x_1-x_2}
\end{split} 
\end{equation*}
hold for every $x_1, x_2 \in \B(x_0; r)$, where $e(x, \xi) = (h\sm)'(c(x, \xi)) \nabla c(x, \xi)$. 
 \end{assumption} 

\vspace{.2cm}

 Theorem~\ref{theorem:final-result} documents the final result on uniform convergence of subdifferentials, recognizing that 
 Assumptions~\ref{assumption:sub-exponential-assumption}---\ref{assumption:real-Lipschitz-condition} 
imply that Assumption~\ref{assumption:weak-convexity-for-composition} holds for $\openset = \B(x_0; r)$.

\vspace{.5cm}
\begin{theorem}
\label{theorem:final-result}
Assume 
Assumptions~\ref{assumption:sub-exponential-assumption}--\ref{assumption:real-Lipschitz-condition},
and $f(x) = \E_{\xi \sim P}[\loss(c(x; \xi))]< \infty$ for $x \in \B(x_0; r)$.

Then there exists a universal constant $C > 0$ such that for every $\delta \in (0, 1)$: 
\begin{equation}
\label{eqn:final-result-conclusion}
	\P^* \left(\sup_{x:x \in X \cap \B(x_0; r)} \H(\partial \phi(x), \partial \phi_S(x)) \ge C(\sigma_0+\sigma r)\zeta\cdot 
		\max\{\Delta, \Delta^2\}\right) \le \delta.
\end{equation}
In the above, $\zeta, \Delta$ follows the same definition in the statement of Theorem~\ref{theorem:basic-statistical-learning-result}.
%In the above, $\zeta = 1+ \sum_{j=1}^\infty (\loss_+^\prime(t_j) - \loss_-^\prime(t_j))$ where $t_j$ enumerate points where 
%$\loss$ is non-differentiable, and 
%\begin{equation}
%	\Delta =  \sqrt{\frac{d+\VC(\F)}{m} \cdot \log\Big(\frac{n }{\delta}\cdot (r L_*+1)\Big)}
%\end{equation} 
%where $\VC(\F)$ is the VC dimension of the set family $\F$ defined in equation~\eqref{eqn:indicator-set-class-F}. 
\end{theorem} 

\begin{proof}
We can assume without loss of generality that $\zeta = 1+\sum_{j=1}^\infty (\loss_+^\prime(t_j) - \loss_-^\prime(t_j)) < \infty$.
If this sum were infinite, then equation~\eqref{eqn:final-result-conclusion} would be trivially satisfied.

Given Theorem~\ref{theorem:main-stochastic-convex-composition-result} and 
Theorem~\ref{theorem:basic-statistical-learning-result}, and with Assumption~\ref{assumption:avoid-boundary-cc}
assumed for $\openset = \B(x_0; r)$ in the statement of Theorem~\ref{theorem:final-result}, 
the remaining task is to show that 
Assumption~\ref{assumption:weak-convexity-for-composition} is met for $\openset = \B(x_0; r)$. 
This can be achieved by first recognizing that for $y \in \B(x_0; r)$: 
\begin{equation*}
\begin{split} 
	\norm{\grad c(y; \xi)\partial \loss(c(y; \xi))} &\le \norm{\grad c(y; \xi) \cdot \partial \loss\ns(c(y; \xi))} + \norm{\grad c(y; \xi) \cdot (\loss\sm)^\prime(c(y; \xi))} \\
	& \le \norm{\grad c(y; \xi)} \cdot \sum_{j=1}^\infty (\loss_+^\prime(t_j) - \loss_-^\prime(t_j)) + \norm{\grad c(y; \xi) \cdot (\loss\sm)^\prime(c(y; \xi))}.
\end{split} 
\end{equation*} 
Assumptions~\ref{assumption:sub-exponential-assumption}--\ref{assumption:Lipschitz-condition} imply that 
$\sup_{y \in \B(x_0; r)} \norm{\grad c(y; \xi)}$ and $\sup_{y \in \B(x_0; r)} \norm{ (\loss\sm)^\prime(c(y; \xi))\grad c(y; \xi) }$
are integrable over $\xi \sim \P$ due to the standard chaining argument~\cite[Section 8]{Vershynin18}. 
This implies that $\sup_{y \in \B(x_0; r)} \norm{\grad c(y; \xi)\partial \loss(c(y; \xi))}$
has an integrable upper bound given that $\sum_{j=1}^\infty (\loss_+^\prime(t_j) - \loss_-^\prime(t_j)) < \infty$, thus confirming the first condition of 
Assumption~\ref{assumption:weak-convexity-for-composition} holds for $\openset = \B(x_0; r)$. The second condition 
of Assumption~\ref{assumption:weak-convexity-for-composition} follows trivially from 
Assumption~\ref{assumption:real-Lipschitz-condition}.
\end{proof} 

\begin{remark}
\emph{
Theorem~\ref{theorem:final-result} establishes a rate for uniform convergence of subdifferential mappings for stochastic 
convex-composite objectives at $\sqrt{\max\{d, \VC(\F)\}/m}$, modulo logarithmic factors.
Our results are different from approaches that depend on the \emph{continuous} differentiability of the 
population objectives, 
or require a nonatomic probability distribution $\P$ to achieve a $1/\sqrt{m}$ rate~\cite{ShapiroXu07, RalphXu11}
(see Section~\ref{sec:related-works} for a more detailed literature review). Indeed, 
Theorem~\ref{theorem:final-result} achieves this tight rate without requiring continuity type assumptions 
in the subdifferential mappings of the 
objective or specific conditions on the data distribution $\P$. 
%The result is also nonasymptotic, 
%attains a tight rate of $\sqrt{\max\{d, \VC(\F)\}/m}$, modulo logarithmic factors.
}

\emph{
In the literature on uniform convergence rates of subdifferential mappings for stochastic weakly-convex objectives, 
methods using Moreau envelope smoothing~\cite{Moreau65} and Attouch's epigraphical convergence 
theorem~\cite{Attouch77} achieve a convergence rate of $\sqrt[4]{1/m}$ suboptimal in terms of the sample size $m$~\cite{DavisDrPa20, DavisDr22}. 
This $\sqrt[4]{1/m}$ convergence is measured using graphical distance, a notion which is topologically stronger than the uniform 
convergence criterion we consider in Theorem~\ref{theorem:final-result}
(see Remark of Theorem~\ref{theorem:bound-HD-by-MS}). Our work complements these studies by showing a convergence 
rate's dependence on sample size $m$ to be $\sqrt{1/m}$ (modulo logarithmic factors) under 
a topologically weaker notion of convergence.
}

%Another approach, which relies on assuming and utilizing the 
%piecewise continuity of the subdifferentials of population objectives, requires a stringent condition on the 
%underlying distribution, specifically that $\P$ is non-atomic, to ensure that the samples almost surely do not lie on the boundaries
%where the subdifferentials may be discontinuous. Moreover, the result is asymptotic and does not capture dependencies of rates
%on the dimension and the function class complexity~\cite{RalphXu11}. 
\emph{
In summary, for the class of nonsmooth, nonconvex, stochastic convex-composite objectives we study, our results often 
provide more precise convergence rates or require fewer assumptions about the distribution  $\P$. 
Examples in Section~\ref{sec:applications} make this concrete.
}
\end{remark}

\renewcommand{\D}{\mathcal{D}}

\section{VC Dimension Bounds} 
\label{sec:VC-dimension-bounds}

In Theorem~\ref{theorem: VC-bounds-on-polynomials} below, 
we establish an upper bound on the VC dimension of the set $\F$ defined in equation~\eqref{eqn:indicator-set-class-F}
under the condition that the functions $x \mapsto c(x; \xi)$ are polynomials in $x$ for all $\xi$. Notably, our result does not require $\xi \mapsto c(x; \xi)$ to be polynomials in the data $\xi$. 
 This general bound will be particularly useful for future applications of Theorem~\ref{theorem:final-result} to concrete problems such as matrix sensing, phase retrieval, and blind deconvolution, where $x \mapsto c(x; \xi)$ naturally takes the form of polynomials in $x$. For further details on these applications, see Section~\ref{sec:applications}.

\vspace{.5cm}
\begin{theorem}
\label{theorem: VC-bounds-on-polynomials}
Consider $x \in \R^d$ and $\xi \in \Xi$, with each mapping $x \mapsto c(x; \xi)$ being a polynomial of degree 
at most $K$ for every $\xi \in \Xi$, and $K \ge 1$. The VC dimension of the set $\F$ (cf. equation~\eqref{eqn:indicator-set-class-F})
\begin{equation*}
	\F = \{\{\xi \in \Xi \colon c(x, \xi) \ge t\} \mid x \in \R^d, t\in \R\}
\end{equation*}
is bounded by $Cd\log(Kd)$, where $C > 0$ is a universal constant.  
\end{theorem} 

\vspace{.2cm}
\begin{definition}
Let $p_1, p_2, \ldots, p_N$ be a sequence of functions where each function $p_i: \R^d \to \R$. 
A vector $\sigma \in \{-1, 0, +1\}^N$ is called a sign pattern of $p_1, p_2, \ldots, p_N$ if there 
exists an $x \in \R^d$ such that the sign of $p_i(x)$ is $\sigma_i$ for all $i= 1, 2, \ldots, n$, where 
$\sigma_i$ is the $i$th coordinate of $\sigma$. 
\end{definition} 

The following theorem, rooted in the field of real algebraic geometry, is referenced from 
Matousek's monograph on discrete geometry~\cite{Matousek13}. 
See~\cite{Milnor64, Thom65} and~\cite[Theorem 3]{Warren68}.

\vspace{.5cm}
\begin{theorem}[{\cite[Theorem 6.2.1]{Matousek13}}]
\label{theorem:classical-result-on-sign-patterns}
Let $p_1, p_2, \ldots, p_N$ be $d$-variate real polynomials of degree at most $K$. The number 
of sign patterns of $p_1, p_2, \ldots, p_N$ is bounded by 
\begin{equation*}
	\left( \frac{50 KN}{d}\right)^d.
\end{equation*}
\end{theorem}

\begin{proof}[Proof of Theorem~\ref{theorem: VC-bounds-on-polynomials}]
Define $z = (x, t) \in \R^d \times \R \cong \R^{d+1}$.
Given a dataset $\xi_1, \xi_2, \ldots, \xi_N$, let us define the functions $q_i(z)= q_i(x, t) = c(x; \xi_i) -t$
for $1\le i \le N$. Notably, each $q_i$ is a $d+1$-variate polynomial with degree at most $K$.  

Consider the set of binary patterns $b$ defined by: 
\begin{equation*}
\begin{split} 
	\mathcal{B} &= \{(b_1, b_2, \ldots, b_N) \in \{0, 1\}^N \colon \\
		&~~~~~~~~~~\text{there exists $z \in \R^{d+1}$ such that}~b_i = \one\{q_i(z) \ge 0\}~\text{for all $1\le i\le N$}\}.
\end{split} 
\end{equation*}
By definition, the number of possible binary patterns, $|\mathcal{B}|$, cannot exceed the number of sign patterns for the 
same polynomials $q_1, q_2, \ldots, q_N$. According to Theorem~\ref{theorem:classical-result-on-sign-patterns}, the number 
of sign patterns is bounded by $(50KN/(d+1))^{d+1}$. Therefore, we conclude:
\begin{equation*}
|\mathcal{B}| \le 	\left( \frac{50 KN}{d+1}\right)^{d+1}.
\end{equation*}  
Remarkably, this bound on $|\mathcal{B}|$ is valid for any choice of the data points  $\xi_1, \xi_2, \ldots, \xi_N$. To establish an 
upper limit on the VC dimension of $\F$, it remains to analyze the set of possible integers $N$ that satisfies the following inequality:
\begin{equation*}
	\left( \frac{50 KN}{d+1}\right)^{d+1} < 2^N. 
\end{equation*} 
A straightforward algebraic manipulation reveals that there exists a universal constant $C > 0$ such that setting $N > C d \log(Kd)$ 
ensures the inequality holds. The conclusion is that for any dataset $\xi_1, \xi_2, \ldots, \xi_N$ with size $N >  C d \log(Kd)$, the 
total number of binary patterns generated by $q_i(z) = c(x; \xi_i) -t$ with $x$ ranging over $\R^d$ and $t$ over $\R$ is strictly fewer 
than $2^N$.  

This proves the bound on VC dimension: $\VC(\F) \le Cd\log(Kd)$ as desired. 
\end{proof}

\section{Applications} 
\label{sec:applications}
%In numerous data-driven applications, we do not have access to the true population models $\phi$. 
%Rather, we rely on finite sample objectives $\phi_S$ to make decisions. Given nonconvexity of 
%these objectives, finding a global minimizer is challenging, and often, we must settle for finding 
%stationary points using iterative algorithms. Recent literature highlights the importance of understanding 
%how these stationary points and their local geometry in $\phi_S$ compare to those in $\phi$, 
%which is essential for evaluating the generalization abilities of solutions and the local convergence speed of algorithms.
%
%In this section, we illustrate through concrete examples how a uniform upper bound on the subdifferential gap: 
%\begin{equation*}
%	\sup_{X \cap \openset} \H(\phi(x), \phi_S(x))
%\end{equation*}
%facilitates our understanding 
%of the finite sample objective $\phi_S$'s landscape through the lens of the population objective $\phi$.
%Section~\ref{sec:general-results} provides the guiding principle behind the analysis of the nonsmooth landscape
%of $\phi_S$, while Section~\ref{sec:phase-retrieval} and Section~\ref{sec:matrix-sensing} provide concrete 
%illustrations of our results on two important problems in modern data science: phase retrieval and matrix sensing. 
%
%
In this section, we return to some of the examples previously discussed in Section~\ref{sec:a-class-of-stochastic-weakly-convex-minimizations} 
to demonstrate the application of our main Theorem~\ref{theorem:final-result}. We will concretely evaluate the 
bounds in Theorem~\ref{theorem:final-result} for some of these objectives, compare these to existing 
results in the literature, and use them to derive new findings.

\subsection{Illustration I: Phase Retrieval} 
\label{sec:phase-retrieval}
%\subsection{Illustration II: Noisy Phase Retrieval} 
We start with the (real-valued) robust phase retrieval objective (Example~\ref{example:robust-phase-retrieval}): 
\begin{equation}
\label{eqn:phase-retrieval}
\begin{split} 
	\min_x \Phi(x)~~&\text{where}~~\Phi(x)~= \E_{(a, b) \sim \P} |\langle a, x\rangle^2 - b|. \\
	\min_x \Phi_S(x)~~&\text{where}~~\Phi_S(x)= \E_{(a, b) \sim \P_m} |\langle a, x\rangle^2 - b| = \frac{1}{m} \sum_{i=1}^m |\langle a_i, x\rangle^2- b_i|.
\end{split} 
\end{equation}
This objective has been studied in the literature~\cite{EldarMe14, DuchiRu19, DavisDrPa20}, yet tight bounds on the 
gap between the empirical and population subdifferential maps have yet to be established. 

In the above, $x \in \R^d$ denotes the unknown signal vector, $a \in \R^d$ denotes the measurement vector, $b \in \R$ denotes 
the response, and $\P_m$ refers to the empirical distribution of $(a_1, b_1), (a_2, b_2), \ldots, (a_m, b_m)$ which are i.i.d. 
sampled from the distribution $\P$. 
This objective falls into the class of stochastic convex-composite objective in 
Section~\ref{sec:explicit-uniform-convergence-rates}, as we can pick $\xi = (a, b) \in \R^d \times \R$, $\loss(z) = |z|$ and $c(x; \xi) = (a^T x)^2 - b$.
Given a radius $r > 0$, we are interested in a high probability bound onto the subdifferential gap for the 
empirical and population phase retrieval objectives over $x \in {\rm cl}(\B(0; r))$: 
\begin{equation*}
	\sup_{x: \norm{x} \le r} \H(\partial \Phi(x), \partial \Phi_S(x)).
\end{equation*}

We need to specify conditions on the distribution of measurements $(a, b) \sim \P$ that allow us to apply 
Theorem~\ref{theorem:final-result}. We recall the concept of a subgaussian tail random vector
\cite[Section 2.5]{Vershynin18}. 

\vspace{.2cm}
\begin{definition}
\label{definition:sub-gaussian}
A random vector $Z$ is said to be $\sigma$-subgaussian if for every vector $v$ with $\norm{v} = 1$: 
\begin{equation*}
	\E \left[\exp\left(\frac{\langle Z, v \rangle^2}{\sigma^2}\right)\right] \le 2. 
\end{equation*}
\end{definition}
Examples of a $\sigma$-subgaussian distribution include normal distributions $Z \sim \normal(0, (c\sigma)^2 I)$, 
and uniform distribution on the hypercube $Z \sim {\rm unif}\{\pm (c\sigma)\}^{d}$ where $c> 0$ can be any constant
obeying $c \le 1/4$. 
Applying Theorem~\ref{theorem:final-result}, we derive an upper bound on the subdifferential gap for the robust 
phase retrieval. Remarkably, our result, which achieves an almost optimal $m^{-1/2}$ rate (up to logarithmic factors), 
holds regardless of whether the random variables $a$ and $b$ are discrete or continuous. This contrasts with existing 
literature, which either requires a continuity assumption or attains a slower $m^{-1/4}$ rate (see Remark of Theorem~\ref{theorem:final-result}).

%Importantly, our convergence result makes no assumptions about the conditional distribution of $b$ given $a$. 
%It could be $b = \langle a, x\opt\rangle^2$ for some fixed vector $x\opt$, or $b$ could be independent of $a$.

\vspace{.5cm}
\begin{corollary}
\label{corollary:phase-retrieval}
Assume the measurement vector $a$ is $\sigma$-subgaussian, and $\E[|b|] < \infty$. 

Then there exists a universal constant $C > 0$ such that for every $\delta \in (0, 1)$ and $r > 0$: 
\begin{equation}
\label{eqn:high-probability-bound-phase-retrieval}
	\P^* \left( \sup_{x: \norm{x} \le r} \H(\partial \Phi(x), \partial \Phi_S(x)) \ge C\sigma^2 r\cdot 
		\max\{\Delta_\Phi, \Delta_\Phi^2\}\right) \le \delta.
\end{equation}
In the above,  
\begin{equation*}
	\Delta_\Phi = \sqrt{\frac{1}{m} \cdot \left(d \log d\log m+ \log(\frac{1}{\delta})\right)}.
\end{equation*}
\end{corollary} 

\begin{proof}
We apply Theorem~\ref{theorem:final-result}. We first check the assumptions. 
For the nonsmooth $\loss(z) = |z|$, it is decomposed 
as $\loss = \loss \ns + \loss\sm$ where $\loss\sm(z) = -z$ and $\loss\ns(z) = 2(z)_+$.
\begin{itemize}
\item (Assumption~\ref{assumption:sub-exponential-assumption}). 
	For the phase retrieval problem, 
	\begin{equation*}
		\grad c(x, \xi) = 2\langle a, x\rangle a,~~~(\loss\sm)'(c(x; \xi)) \grad c(x; \xi) = -2\langle a, x\rangle a.
	\end{equation*}
	Thus, at the center $x = 0 \in \R^d$, both gradients vanish: $\grad c(0, \xi) = (\loss\sm)'(c(0; \xi)) \grad c(0; \xi) = 0$. 
%	This verifies Assumption~\ref{assumption:sub-exponential-assumption}.
%	For every $x$ with $\norm{x} \le r$,  the scalar random variable 
%	$\langle a, x \rangle$ is $r$-subgaussian, and for every $v$ with $\norm{v} \le 1$, $\langle a, v\rangle$ is $1$-subgaussian. 
%	Given the property that the product of two subgaussian random variables is 
%	subexponential~\cite[Lemma 2.7.7]{Vershynin18}, we derive that for some universal constant $C > 0$,
%	$\langle a, x \rangle \langle a, v\rangle$ must be $Cr$ subexponential for every unit vector $v$ with $\norm{v}=1$. 
%	Thus, both $\grad c(x, \xi)$ and  $(\loss\sm)'(c(x; \xi))$ are $Cr$ subexponential by definition.
\item (Assumption~\ref{assumption:Lipschitz-condition}). It is easy to verify that 
	\begin{equation*}
	\begin{split} 
		\grad c(x_1, \xi) - \grad c(x_2, \xi) &= 2 \langle a, x_1 - x_2 \rangle a \\
		e(x_1, \xi) - e(x_2, \xi) &= 2 \langle a, x_2 - x_1 \rangle a
	\end{split},
	\end{equation*}
	where $e(x, \xi) = (\loss\sm)'(c(x; \xi)) \grad c(x; \xi)$. 
	Notably, for every vector $v$, $\langle a, v\rangle$ is $\sigma \norm{v}$-subgaussian. 
	Given the property that the product of two subgaussian random variables is 
	subexponential~\cite[Lemma 2.7.7]{Vershynin18}, we derive that for some universal constant $C > 0$, each of 
	\begin{equation*}
	\begin{split} 
		\langle \grad c(x_1, \xi) - \grad c(x_2, \xi), v\rangle = 2 \langle a, x_1 - x_2 \rangle \langle a, v\rangle \\
		e(x_1, \xi) - e(x_2, \xi) = 2 \langle a, x_2 - x_1 \rangle \langle a, v\rangle
	\end{split} 
	\end{equation*} 
	must be $C \sigma^2 \norm{x_1 - x_2}$ subexponential for every unit vector $v$ with $\norm{v} = 1$. 

\item (Assumption~\ref{assumption:real-Lipschitz-condition}). It is easy to verify that 
	\begin{equation*}
	\begin{split} 
		\norm{\grad c(x_1, \xi) - \grad c(x_2, \xi)} &= 2 \norm{\langle a, x_1 - x_2 \rangle a} \le 2 \norm{a}^2 \norm{x_1 - x_2} \\
		\norm{e(x_1, \xi) - e(x_2, \xi)} &= 2 \norm{\langle a, x_1 - x_2 \rangle a} \le 2 \norm{a}^2 \norm{x_1 - x_2}
	\end{split},
	\end{equation*}
	where $e(x, \xi) = (h\sm)'(c(x; \xi)) \grad c(x; \xi)$. 	Since $a$ is $\sigma$-subgaussian, 
	$\E[L(\xi)]  = 2\E[\norm{a}^2] \le C\sigma^2 d$ for some universal constant $C > 0$.
	Thus Assumption~\ref{assumption:real-Lipschitz-condition} is satisfied 
	with $L(\xi) = 2\norm{a}^2$.

\item (Integrability). Notably $\E[h(c(x; \xi))]  = \E[|b-\langle a, x \rangle^2|] < \infty$ for every $x \in \R^d$. 
\end{itemize}
We then compute the bound in Theorem~\ref{theorem:final-result}. Namely, we need to compute $\zeta$ and $\VC(\F)$.

For the nonsmooth function $\loss(z) = |z|$, the corresponding value of $\zeta$ is given by 
\begin{equation*}
	\zeta = 1+\loss_+^\prime(0) - \loss_-^\prime(0) = 3,
\end{equation*}
as $0$ is the only nondifferentiable point of $\loss$. The corresponding $\F$ is given by 
\begin{equation*}
	\F = \left\{ \{a \in \R^d, b \in \R \colon \langle a, x\rangle^2 - b \ge t \}\mid x \in \R^d, t \in \R \right\}.
\end{equation*}
We now bound its VC dimension using Theorem~\ref{theorem: VC-bounds-on-polynomials}. 
Note $x \mapsto \langle a, x\rangle^2 - b$ 
is a degree $2$ polynomial in $x \in \R^d$ for every $a, b$. 
Using Theorem~\ref{theorem: VC-bounds-on-polynomials}, we 
obtain that for some universal constant $C > 0$: 
\begin{equation*}
	\VC(\F) \le C d \log (d). 
\end{equation*}
Given $r > 0$, we apply Theorem~\ref{theorem: VC-bounds-on-polynomials} to the open ball $\B(0; 2r)$, and we obtain that 
for some universal constant $C > 0$: 
\begin{equation}
%\label{eqn:high-probability-bound-phase-retrieval}
	\P^* \left( \sup_{x: x \in \B(0; 2r)} \H(\partial \Phi(x), \partial \Phi_S(x)) \ge C\sigma^2 r\cdot 
		\max\{\Delta_\Phi, \Delta_\Phi^2\}\right) \le \delta.
\end{equation}
Corollary~\ref{corollary:phase-retrieval} then follows by recognizing the set inclusion: $ {\rm cl}(\B(0; r)) \subseteq \B(0; 2r)$.
%To summarize, if we plug in all the estimates above into the formula of $\Delta$ in equation~\eqref{eqn:error-definition}, we obtain 
%that if we choose 
%\begin{equation*}
%\Delta_\Phi = \sqrt{\frac{1}{m} \cdot \left(d \log d\log m+ \log(\frac{1}{\delta})\right)}.
%\end{equation*} 
%then the high probability bound~\eqref{eqn:high-probability-bound-phase-retrieval} holds when we choose a large absolute constant $C$. 
%The collection of half planes $\mathcal{G} = \{\{ v \in \R^d, w\in \R: \langle v, x\rangle + w \ge 0\} \mid x \in \R^d\}$
\end{proof} 

In the literature on optimization, statistics and machine learning, 
establishing a uniform upper bound on the convergence of subdifferentials is important because it 
facilitates our understanding 
of the landscape of the nonsmooth, nonconvex finite sample objective $\phi_S$, e.g., the location of stationary points, growth conditions, etc.,
through the lens of the population objective $\phi$.
We shall give further applications in Section~\ref{sec:noiseless-phase-retrieval} 
how Corollary~\ref{corollary:phase-retrieval} contributes to existing results to refine
our understanding of the landscape of the nonsmooth phase retrieval objective $\Phi_S$ in finite samples. 

Notably, characterization of the nonsmooth landscape of empirical risk function is an important and 
profound topic, and Section~\ref{sec:noiseless-phase-retrieval} presents only a very basic example of how to apply our general theorem.
Further results on this topic will be detailed in an upcoming manuscript.

\subsubsection{Noiseless Phase Retrieval}
\label{sec:noiseless-phase-retrieval}
We consider the noiseless phase retrieval, where the measurement $b$ obeys: 
\begin{equation*}
	b = \langle a, \bar{x}\rangle^2
\end{equation*}
for some fixed vector $\bar{x} \in \R^d$. In this case, the objectives $\Phi_S$ and $\Phi$ become: 
\begin{equation*}
	\Phi_S(x) =  \frac{1}{m} \sum_{i=1}^m |\langle a_i, x\rangle^2- \langle a_i, \bar{x}\rangle^2|,
	~~~
	\Phi(x) =  \E_{a \sim \P}[ |\langle a, x\rangle^2- \langle a, \bar{x}\rangle^2|].
\end{equation*}
Let us denote the set of stationary points of $\Phi_S$ and $\Phi$ to be: 
\begin{equation*}
	\mathcal{Z}_S = \{x \in \R^d \colon 0 \in \partial \Phi_S(x)\}~~~\text{and}~~~\mathcal{Z} = \{x \in \R^d \colon 0 \in \partial \Phi(x)\}.
\end{equation*}

For standard normal measurement vectors $a \sim \normal(0, I)$, a recent important and interesting study first characterizes $\mathcal{Z}$, 
the locations of stationary points of the population objective~$\Phi(x)$~\cite{DavisDrPa20}. Remarkably, the authors further
develop a quantitative version of Attouch's epi-convergence theorem to show that the stationary points 
of the finite samples objective $\Phi_S$ converge to those of the population objective $\Phi$ at a rate of $\sqrt[4]{d/m}$~\cite{DavisDrPa20}.
More precisely, \cite[Theorem 5.2, Corollary 6.3]{DavisDrPa20} prove that, there exist numerical constants $c, C > 0$
such that when $m \ge Cd$, with an inner probability at least $1-C\exp(-cd)$ (recall the definition 
of inner probability from Section~\ref{sec:probability-measure-and-measurability}): 
\begin{equation*}
	\mathbb{D}(\mathcal{Z}_S, \mathcal{Z}) \le C\sqrt[4]{\frac{d}{m}} \norm{\bar{x}}.
\end{equation*}
In words, with high probability, for any stationary point of $\Phi_S$, there is a stationary point of $\Phi$ that is 
at most a distance of $C\sqrt[4]{d/m} \norm{\bar{x}}$ in the $\ell_2$ norm away from it. It is important to note that 
the slow rate in $m$, namely $1/\sqrt[4]{m}$, appears intrinsic to the approach based on the
quantitative version of Attouch's epi-convergence theorem~\cite[Section 5]{DavisDr22}.

In this paper, we show a stronger result for large sample size $m$: the stationary points of the finite samples objective $\Phi_S$ converge to 
those of the population objective $\Phi$ at a rate of $\sqrt{d/m}$ up to logarithmic factors in $d, m$. This rate improvement in $m$
is due to our bound on uniform convergence of subdifferentials of $\Phi$ and $\Phi_S$ 
in Corollary~\ref{corollary:phase-retrieval}, and our proof of Proposition~\ref{theorem:tight-nonsmooth-phase-retrieval-landscape} 
benefits from the characterization of approximate stationary 
points of the population objective $\Phi$~\cite[Corollary 5.3]{DavisDrPa20}. Nonetheless, our result, 
Proposition~\ref{theorem:tight-nonsmooth-phase-retrieval-landscape}, comes with 
a tradeoff of additional logarithmic factors in $d$ and $m$.

\vspace{.5cm}
\begin{proposition}
\label{theorem:tight-nonsmooth-phase-retrieval-landscape}
Assume $a \sim \normal(0, I)$. Then 
there exist numerical constants $c, C > 0$ such that when $m \ge Cd$, with an inner probability at least $1-C\exp(-cd)$: 
\begin{equation*}
	\mathbb{D}(\mathcal{Z}_S, \mathcal{Z})  \le C\left(\sqrt{\frac{d}{m}\log (d) \log (m)} + \frac{d}{m} \log (d) \log(m)\right)\norm{\bar{x}}.
\end{equation*}
\end{proposition} 

The formal proof of Proposition~\ref{theorem:tight-nonsmooth-phase-retrieval-landscape} is deferred to Appendix, 
Section~\ref{sec:theorem-tight-nonsmooth-phase-retrieval-landscape}.

\subsection{Illustration II: Blind Deconvolution} 
\label{sec:blind-deconvolution}
Our second example studies the (real-valued) blind deconvolution objective: 
\begin{equation*}
\begin{split}
	\min_{y, w} \tilde{\Phi}(y, w)~~&\text{where}~~\tilde{\Phi}(y, w) = \E_{(u, v, b) \sim \P}|\langle u, y \rangle \langle v, w \rangle - b|. \\
	\min_{y, w} \tilde{\Phi}_S(y, w)~~&\text{where}~~\tilde{\Phi}(y, w) = \E_{(u, v, b) \sim \P_m}
		|\langle u, y\rangle \langle v, w \rangle - b| = \frac{1}{m} \sum_{i=1}^m |\langle u_i, y\rangle \langle v_i, w\rangle-b_i|. 
\end{split} 
\end{equation*} 
This objective proposed in the literature~\cite{CharisopoulosDaDiDr21, Diaz19}
falls into the stochastic convex-composite minimization problems in 
Section~\ref{sec:explicit-uniform-convergence-rates}, as we can simply set $x = (y, w) \in \R^{d_1} \times \R^{d_2}$, 
$\xi = (u, v, b) \in \R^{d_1} \times \R^{d_2} \times \R$, $\loss(z) = |z|$ and $c(x; \xi) = \langle u, y\rangle \langle v, w\rangle - b$.

Another application of Theorem~\ref{theorem:final-result} yields an upper bound on the subdifferential gap for the robust blind deconvolution objective.

\vspace{.2cm} 
\begin{corollary}
\label{corollary:blind-deconvolution}
Assume the measurements $u, v$ are independent $\sigma$-subgaussian vectors, and $\E[|b|] < \infty$. 

Then there exists a universal constant $C > 0$ such that for every $\delta \in (0, 1)$ and $r > 0$: 
\begin{equation}
\label{eqn:high-probability-bound-phase-retrieval}
	\P^* \left( \sup_{x: \norm{x} \le r} \H(\partial \Phi(x), \partial \Phi_S(x)) \ge C\sigma^2 r\cdot 
		\max\{\Delta_{\tilde{\Phi}}, \Delta_{\tilde{\Phi}}^2\}\right) \le \delta.
\end{equation}
In the above,  
\begin{equation*}
	\Delta_{\tilde{\Phi}} = \sqrt{\frac{1}{m} \cdot \left((d_1+d_2) \log (d_1+d_2)\log m+ \log(\frac{1}{\delta})\right)}.
\end{equation*}
\end{corollary} 

We wish to comment that, the recent work~\cite{Diaz19} has studied the location of the stationary points 
of the finite sample objective $\tilde{\Phi}_S$ for the noiseless blind deconolution problem, assuming independent measurements 
$u \sim \normal(0, I_{d_1})$, and $v \sim \normal(0, I_{d_2})$, and no corruptions in the measurement outcome 
$b = \langle u, \bar{y}\rangle \langle v, \bar{w}\rangle$. Their analysis employs a version of Attouch's epi-convergence theorem, similar to its use in the noiseless phase retrieval problem~\cite{DavisDrPa20}, to identify the location of finite sample stationary points. In Section~\ref{sec:noiseless-phase-retrieval}, we demonstrated that our approach achieves tighter convergence rates in the context of phase retrieval. Although we expect that similar improvements can be achieved for blind deconvolution, a full analysis remains to be explored and is beyond the scope of the current paper.

\subsection{Illustration III: Matrix Sensing}
\label{sec:matrix-sensing}
Our final example demonstrates the flexibility of our framework. Consider the robust matrix sensing 
objective~\cite{LiZhMaVi20, CharisopoulosChDaDiDiDr21, DingJiChQuZh21}:
\begin{equation}
\label{eqn:matrix-sensing}
\begin{split} 
	\min_X \bar{\Phi}(X)~~&\text{where}~~\bar{\Phi}(X)~= \E_{(a, b) \sim \P} |\langle A, XX^T\rangle - b|. \\
	\min_X \bar{\Phi}_S(X)~~&\text{where}~~\bar{\Phi}_S(X)= \E_{(a, b) \sim \P_m} |\langle A, XX^T\rangle - b| = 
		\frac{1}{m} \sum_{i=1}^m |\langle A_i, XX^T\rangle - b_i|.
\end{split} 
\end{equation}
Here, $A \in \R^{D \times D}$ represents the measurement matrices, $b \in \R$ the measurements, and 
$X \in \R^{D \times r_0}$ a low-rank matrix, typically with $r_0$ significantly smaller than $D$ in applications.
This objective falls into the stochastic convex-composite objectives in 
Section~\ref{sec:explicit-uniform-convergence-rates}, as we can simply set $x = X \in \R^{D \times r_0}$, 
$\xi = (A, b) \in \R^{D \times D} \times \R$, $\loss(z) = |z|$ and $c(x; \xi) = c(x; (A,b)) = \langle A, XX^T \rangle - b$.

Another application of Theorem~\ref{theorem:final-result} yields an upper bound on the subdifferential gap for the robust matrix 
sensing objective. Recall a matrix $A \in \R^{D \times D}$ is a $\sigma$-subgaussian random matrix if the vectorized matrix 
$\textbf{{\rm vec}}(A) \in \R^{D \times D}$ is a $\sigma$-subgaussian vector in the sense of Definition~\ref{definition:sub-gaussian}~\cite{Vershynin18}.
Below, we use $\norm{\cdot}_F$ to denote the Frobenius norm of a matrix in $\R^{D\times D}$. %This norm is 
%identical to the Euclidean norm on $\R^{D^2}$ when we identify each matrix in $\R^{D\times D}$ as a vector in $\R^{D^2}$ . 

\vspace{.5cm}
\begin{corollary}
\label{corollary:matrix-sensing}
Assume the measurement matrix $A \in \R^{D \times D}$ is $\sigma$-subgaussian, and $\E[|b|] < \infty$. 

Then there exists a universal constant $C > 0$ such that for every $\delta \in (0, 1)$ and $r > 0$: 
\begin{equation}
\label{eqn:high-probability-bound-matrix-sensing}
	\P^* \left( \sup_{X: \norm{X}_F \le r} \H(\partial \bar{\Phi}(X), \partial \bar{\Phi}_S(X)) \ge C\sigma r\cdot 
		\max\{\Delta_{\bar{\Phi}}, \Delta_{\bar{\Phi}}^2\}\right) \le \delta.
\end{equation}
In the above,  
\begin{equation*}
	\Delta_{\bar{\Phi}} = \sqrt{\frac{1}{m} \cdot \left(Dr_0 \log (Dr_0)\log m+ \log(\frac{1}{\delta})\right)}.
\end{equation*}
\end{corollary} 

The proof is provided in Section~\ref{sec:proof-corollary-matrix-sensing} and is analogous to the phase retrieval case. 
The primary challenge, which comes from evaluating the VC dimension, is addressed using 
Theorem~\ref{theorem: VC-bounds-on-polynomials}.

%To summarize, if we plug in all the estimates above into the formula of $\Delta$ in equation~\eqref{eqn:error-definition}, we obtain 
%that if we choose 
%\begin{equation*}
%\Delta_{\bar{\Phi}} = \sqrt{\frac{1}{m} \cdot \left((Dr) \log (Dr)\log m+ \log(\frac{1}{\delta})\right)}.
%\end{equation*} 
%then the high probability bound~\eqref{eqn:high-probability-bound-matrix-sensing} holds when we choose a large absolute constant $C$. 
%The collection of half planes $\mathcal{G} = \{\{ v \in \R^d, w\in \R: \langle v, x\rangle + w \ge 0\} \mid x \in \R^d\}$

\section{Discussion}
\label{sec:discussions}
This paper presents a general technique for proving subdifferential convergence in weakly convex functions through the 
establishment of subgradient convergence (Theorem~\ref{theorem:bound-HD-by-MS}).  We demonstrate this approach 
to stochastic convex-composite minimization problems---an important class of weakly convex objectives---and derive concrete 
convergence rates for subdifferentials of empirical objectives
using tools from statistical learning theories (Theorem~\ref{theorem:final-result}). Our results achieve optimal rates
with tight dependence on the sample size and dimensionality (up to logarithmic factors) and do not rely on key distributional assumptions
in the literature requiring the population objectives to be continuously differentiable.
%not requiring the population subdifferentials to be continuous (both 
%innersemicontinuous and outersemicontinuous) in Hausdorff metric, an assumption typically needed to achieve a tight $1/\sqrt{m}$ rate. 
In Section~\ref{sec:applications}, we demonstrate how our results lead to complementary understanding of subdifferential convergence 
and its uses for landscape analysis, in more specific applications such as robust phase retrieval and matrix sensing problems.
 
There are many interesting future directions to extend the scope of the current work. Notable examples include: 
\begin{itemize}
\item Can Theorem~\ref{theorem:final-result} be extended to other weakly-convex models, such as max-of-smooth 
	functions~\cite{Rockafellar81, HanLe23}? Extending this theorem could provide tools for solving other weakly-convex optimization problems.
\item Can Theorem~\ref{theorem:bound-HD-by-MS} be extended to other function classes? An intriguing candidate is the class of difference-of-convex functions~\cite{LiuCuPa22}. %~\cite{FacchineiPa03}.
	%, or subdifferential convergence in certain stochastic minimax optimizations. 
\item Theorem~\ref{theorem:final-result} could facilitate our understanding of the landscape of a certain class of 
	nonsmooth empirical risk minimization problems. Exploring its statistical applications, such as nonlinear quantile 
	regressions~\cite{KoenkerPa96, Koenker05}, may lead to new insights into data-driven models.
\end{itemize}

\bmhead{Acknowledgements}
Feng Ruan extends his gratitude to his friends and advisor---John Duchi, Qiyang Han, 
X.Y. Han, Kate Yao and Yichen Zhang---for their support and patience as 
he tackled this once seemingly challenging problem. 
A shift in perspective during his time at Northwestern led him to this solution.
He is also grateful to Damek Davis and Dmitriy Drusvyatskiy for their valuable 
feedback on the initial manuscript draft posted on arXiv. Feng Ruan would also like to express gratitude to 
Johannes Royset, Alexander Shapiro and Lai Tian, and the anonymous reviewers for their valuable discussions on literature 
and insights regarding the proof of Theorem~1 using variational analytic tools and also its limitations.

\newpage
\begin{appendices}

%\section{Construction of Probability Space}
%\label{sec:construction-of-the-probability-space}
%In this section, we provide the measure-theoretical details for the construction of the i.i.d. 
%random samples $\xi_1, \xi_2, \ldots, \xi_m, \ldots$ described in the main text (e.g., 
%Section~\ref{sec:implications-to-stochastic-weakly-convex-minimizations}). Notably, our construction 
%of these random variables follow the standard treatments~e.g.,\cite[Section 7]{ShapiroDeRu21} and~\cite[Chapter 7]{Billingsley86}.
%
%We define a single random variable $\xi$ on a probability space $(\Xi, \mathcal{G}, \P)$. Utilizing this space, 
%we construct the countable product space $(\Xi, \mathcal{G}, \P)^\N$ where the sequence of i.i.d. random variables 
%$(\xi_1, \xi_2, \ldots, \xi_m, \ldots)$ can be properly defined on this measurable space~\cite[Section 36]{Billingsley86},
%with each $\xi_i$ acting as a coordinate projection. In this product space, 
%we extend the sigma-algebra and the probability measure to obtain a complete measure, following standard procedures in measure theory.

\section{Proof of Theorem~\ref{theorem:basic-statistical-learning-result}}
\label{sec:proof-theorem-basic-statistical-learning-result}

Our proof starts with a basic triangle inequality bound: 
\begin{equation}
\label{eqn:basic-inequality-bounds-on-G}
	\sup_{x \in \B(x_0; r)} \norm{G_S(x) - G(x)} \le \Big( \sum_{j=1}^\infty a_j \Big) \cdot \Gamma_1  + \Gamma_2.
\end{equation} 
where the two quantities $\Gamma_1, \Gamma_2$ are defined by: 
\begin{equation*}
\begin{split} 
	\Gamma_1&= \sup_{x \in \B(x_0; r), t \in \R} \norm{\E_{\xi \sim \P_m}[\one\{c(x; \xi)  \ge t\} \cdot \grad c(x; \xi)] - 
		\E_{\xi \sim \P}[\one\{c(x; \xi)  \ge t\} \cdot \grad c(x; \xi)]}   \\
	\Gamma_2 &= \sup_{x \in \B(x_0; r)}~~~~ \norm{\E_{\xi \sim \P_m}[(\loss\sm)^\prime(c(x; \xi)) \cdot \grad c(x; \xi)] - 
		\E_{\xi \sim \P}[(\loss\sm)^\prime(c(x; \xi)) \cdot \grad c(x; \xi)]}
\end{split}.
\end{equation*} 
We will deduce high probability upper bounds on $\Gamma_1$ and $\Gamma_2$ using Theorem
\ref{theorem:uniform-convergence-theorem}. The proof of Theorem~\ref{theorem:uniform-convergence-theorem}
is technical and lengthy, and deferred to Section~\ref{sec:proof-theorem-uniform-convergence-theorem}.

\vspace{.5cm}
\newcommand{\Z}{\mathcal{Z}}
\newcommand{\W}{\mathcal{W}}
\begin{theorem}
\label{theorem:uniform-convergence-theorem}
Consider two subsets of $\R^d$, $\Z_1$ and $\Z_2$, where $\Z_2$ is a Euclidean ball centered at $z_{2, 0}$ with radius $r$.
Let $q \colon \mathcal{Z}_1 \times \Xi \to \R$ and $w\colon \mathcal{Z}_2 \times \Xi \to \R^d$. 
Suppose $\xi_1, \xi_2, ..., \xi_m$ are i.i.d. samples from a distribution $\P$ with support $\Xi$. For any $z_1 \in \Z_1$ and $z_2 \in \Z_2$, we 
define:
\begin{equation*}
\begin{split} 
	K_m(z_1, z_2, t) &= \frac{1}{m} \sum_{i=1}^m \one \left\{q(z_1, \xi_i) \ge t\right\} w(z_2, \xi_i),~~~~\\
	K(z_1, z_2, t) &=  \E_{\xi \sim \P}\bigg[\one \left\{q(z_1, \xi) \ge t\right\} w(z_2, \xi)\bigg]
\end{split}.
\end{equation*}
 Assume the following: 
\begin{enumerate}[label=(\alph*)]
\item The random variable $w(z_{2, 0}, \xi)$ is sub-exponential with parameter $\sigma_0^2$. 
\item For every $z_2, z_2' \in \mathcal{Z}_2$, the random vector $w(z_2, \xi) - w(z_2', \xi)$ is $\sigma \norm{z_2- z_2'}$
	subexponential. 
\item The random functions $z_2 \mapsto w(z_2, \xi)$ and $z_1 \mapsto q(z_1, \xi)$ are continuous. 
\end{enumerate} 
Under these conditions, there exists a universal constant $C > 0$ such that for every $\delta \in (0, 1)$: 
\begin{equation}
\label{eqn:crazy-ept}
	\P^* \left(\sup_{z_1 \in \Z_1, z_2 \in \Z_2, t\in \R} \norm{K_m(z_1, z_2, t) - K(z_1, z_2, t)} \ge C(\sigma r + \sigma_0)\cdot 
		\max\{\Delta_{\W}, \Delta_{\W}^2\}\right) \le \delta
\end{equation}
where 
\begin{equation*}
	\Delta_{\W} =  \sqrt{\frac{1}{m} \cdot \Big(d + \VC(\W)\log(m) + \log (\frac{1}{\delta}) \Big) }.
\end{equation*} 
Here, $\W:= \left\{\{\xi \in \Xi: q(z_1, \xi) \ge t\} \mid z_1 \in \Z_1, t\in \R\right\}$, and 
$\VC(\W)$ denotes its VC dimension.
\end{theorem} 

\vspace{.5cm}

%Below we will demonstrate how Theorem~\ref{theorem:uniform-convergence-theorem} implies high probability bounds on 
%$\Gamma_1, \Gamma_2$ in equation~\eqref{eqn:Gamma-high-probability-bounds}. 

We go back and prove Theorem~\ref{theorem:basic-statistical-learning-result}.

We deduce high probability bounds on $\Gamma_1, \Gamma_2$ by applying 
Theorem~\ref{theorem:uniform-convergence-theorem} to two different sets of functions $q, w$. 
Recall the following definition of $\Delta$ as given in equation~\eqref{eqn:error-definition}:
\begin{equation*}
	\Delta =  \sqrt{\frac{1}{m} \cdot \Big(d+\VC(\F)\log m +\log( \frac{1}{\delta})\Big)}.
\end{equation*} 
Below we set $\mathcal{Z}_1 = \R^d$ and $\mathcal{Z}_2 = \B(x_0; r)$. 

\begin{enumerate}
\item (Bounds on $\Gamma_1$)
First, we set $q(x, \xi) = c(x, \xi)$
and $w(x', \xi) = \grad c(x', \xi)$. 
Given Assumptions~\ref{assumption:sub-exponential-assumption} and~\ref{assumption:Lipschitz-condition}, 
the conditions of Theorem~\ref{theorem:uniform-convergence-theorem} are satisfied. Applying this theorem 
yields the following guarantee. Let us define
\begin{equation*}
\begin{split} 
	\Gamma_1' = \sup_{x \in \R^d, t \in \R, x' \in \B(x_0; r)} \norm{\E_{\xi \sim \P_m}[\one\{c(x; \xi)  \ge t\} \cdot \grad c(x'; \xi)] - 
		\E_{\xi \sim \P}[\one\{c(x; \xi)  \ge t\} \cdot \grad c(x'; \xi)]}.
\end{split} 
\end{equation*}
Then there exists a universal constant $C > 0$ such that for every $\delta \in (0, 1)$: 
\begin{equation*}
	\P^*(\Gamma_1' \ge C(\sigma r+\sigma_0)\cdot \max\{\Delta, \Delta^2\}) \le \delta.
\end{equation*} 
Since $\Gamma_1 \le \Gamma_1'$ always holds by definition, this establishes the bound: 
\begin{equation}
\label{eqn:Gamma-high-probability-bound-one}
	\P^*(\Gamma_1 \ge C(\sigma r+\sigma_0)\cdot \max\{\Delta, \Delta^2\}) \le \delta.
\end{equation} 
%first inequality of equation~\eqref{eqn:Gamma-high-probability-bounds}. 

\item (Bounds on $\Gamma_2$) Next, we set $q(x, \xi) \equiv 0$ and 
$w(x', \xi) = (\loss\sm)^\prime(c(x'; \xi)) \cdot \grad c(x'; \xi)$. Given Assumptions~\ref{assumption:sub-exponential-assumption} 
and~\ref{assumption:Lipschitz-condition}, the conditions of Theorem~\ref{theorem:uniform-convergence-theorem} are satisfied. 
Applying this theorem provides the following guarantee. Let us define 
\begin{equation*}
\begin{split} 
	\Gamma_2' = \sup_{t \in \R, x' \in \B(x_0; r)} \norm{\E_{\xi \sim \P_m}[\one\{0  \ge t\} \cdot (\loss\sm)^\prime(c(x'; \xi)) \grad c(x'; \xi)] - 
		\E_{\xi \sim \P}[\one\{0  \ge t\} \cdot (\loss\sm)^\prime(c(x'; \xi)) \grad c(x'; \xi)]} .
\end{split} 
\end{equation*}
Then there exists a universal constant $C > 0$ such that for every $\delta \in (0, 1)$: 
\begin{equation*}
	\P^*(\Gamma_2' \ge C(\sigma r+\sigma_0)\cdot \max\{\Delta, \Delta^2\}) \le \delta.
\end{equation*} 
Since $\Gamma_2 \le \Gamma_2'$ by definition, this establishes the bound: 
\begin{equation}
\label{eqn:Gamma-high-probability-bound-two}
	\P^*(\Gamma_2 \ge C(\sigma r+\sigma_0)\cdot \max\{\Delta, \Delta^2\}) \le \delta.
\end{equation} 
\end{enumerate} 

With both high probability bounds on $\Gamma_1$ and $\Gamma_2$ in 
equations~\eqref{eqn:Gamma-high-probability-bound-one}
and~\eqref{eqn:Gamma-high-probability-bound-two}, we can use the inequality 
\eqref{eqn:basic-inequality-bounds-on-G} to deduce that 
\begin{equation*}
	\P^*\left( \sup_{x \in \B(x_0; r)} \norm{G_S(x) - G(x)} \ge C \cdot \Big( \sum_{j=1}^\infty a_j + 1\Big) \cdot (\sigma r + \sigma_0) \cdot \max\{\Delta, \Delta^2\}\right) \le 2\delta. 
\end{equation*} 
By appropriately adjusting the numerical constant from $\delta$ to $\delta/2$, and noting that 
$\zeta = 1+\sum_{j=1}^\infty a_j$ by definition, Theorem~\ref{theorem:basic-statistical-learning-result} follows.

%\newpage

\subsection{Proof of Theorem~\ref{theorem:uniform-convergence-theorem}}
\label{sec:proof-theorem-uniform-convergence-theorem}
%Let $\Z_1^o$, $\Z_2^o$ and $\R^o$ denote a dense countable subset of $\Z_1$, $\Z_2$, and $\R$ respectively.
%Notice that with probability one: 
%\begin{equation}
%\label{eqn:dense-approximation}
%\begin{split} 
%	&\sup_{z_1 \in \Z_1, z_2 \in \Z_2, t\in \R} \norm{K_m(z_1, z_2, t) - K(z_1, z_2, t)}  \\
%	&= \sup_{z_1 \in \Z_1^o, z_2 \in \Z_2^o, t\in \R^o} \norm{K_m(z_1, z_2, t) - K(z_1, z_2, t)}.
%\end{split} 
%\end{equation}
%Hence, the supremum over the dense countable subsets is the same as over the original sets, which ensures 
%the measurability. 
%The reason is for every $z_1, t$ and sample instances $\xi_1, \xi_2, \ldots, \xi_m$, we can always find $z_1^o \in \Z_1^o, t\in \R^o$
%such that $\mathbf{1}\{o(z_1, \xi_i) \ge t\} = \mathbf{1}\{o(z_1^o, \xi_i) \ge t^o\}$ for every $i = 1, 2, \ldots, m$ since 
%$z_1 \mapsto o(z_1, \xi)$ is continuous.

Our proof of Theorem~\ref{theorem:uniform-convergence-theorem} 
%establishes the probabilistic bound 
%in equation~\eqref{eqn:crazy-ept}. We 
combines two well-established techniques 
for proving uniform bounds on the supremum of a random process in the statistical learning theory: 
chaining, which constructs a sequence of progressively finer approximations to the continuous process, 
and Sauer-Shelah Lemma, a combinatorial principle using the VC dimension to control the 
complexity of the function class under consideration~\cite[Chapter 2]{VanDerVaartWe96}.
The chaining argument specifically addresses the continuous parameter $z_2$, and the complexity of its increments is 
then controlled using the VC dimension technique applied to the function class specified by the parameters $z_1, t$. 
To deal with measurability issues, we use the calculus of \emph{outer integral}~\cite[Chapter 1.2]{VanDerVaartWe96}.
A crucial part of our proof demonstrates that our process becomes measurable after symmetrization, building on
ideas in the monograph~\cite[Chapter 2.3]{VanDerVaartWe96}. 

Given the subtlety in dealing with measurability, and differences from existing literature, which typically focuses either on chaining or on VC dimension arguments 
applied to indicator functions, we provide detailed explanations of our combined approach to ensure clarity and completeness. 

\subsubsection{Chaining Argument}
For notational simplicity, below we denote: 
\begin{equation*}
	R_m(z_1, z_2, t) = K_m(z_1, z_2, t) - K(z_1, z_2, t). 
\end{equation*}

The main idea in the chaining argument involves discretizing the parameter $z_2$ into increasingly finer subsets, 
constructing chains to approximate the supremum of $R_m(z_1, z_2, t)$, 
which allows more tractable probabilistic analysis. 

Let $\eps_k = r 2^{-k}$ for every $k \ge 0$. An $\eps_k$-cover of $\mathcal{Z}_2$, denoted by $\mathcal{Z}_{2, \eps_k}$, 
enables the approximation of every point $z_2 \in \mathcal{Z}_2$ by a point $\pi_k(z_2) \in \mathcal{Z}_{2, \eps_k}$ 
such that $\norm{z_2-\pi_k(z_2)} \le \eps_k$. Specifically, we take the initial cover $\mathcal{Z}_{2, \eps_0} = \{z_{2, 0}\}$ to consist
solely of the center $z_{2, 0}$ of the ball $\mathcal{Z}_2$. By employing a standard volume-type argument, we can further select
the cover $\mathcal{Z}_{2, \eps_k}$ such that $\log |\mathcal{Z}_{2, \eps_k}| \le p \log(2+r/\eps_k)$ for every $k \ge 1$
\cite[Corollary 4.2.13]{Vershynin18}.

Using these coverings, we can decompose $R_m(z_1, z_2, t)$ through a chaining argument~\cite{Dudley67}: 
\begin{equation*}
	R_m(z_1, z_2, t) = R_m(z_1, z_{2, 0}, t) + \sum_{k=1}^\infty \Big[R_m(z_1, \pi_k(z_2), t) - R_m(z_1, \pi_{k-1}(z_2), t)\Big].
\end{equation*}
This decomposition is valid because the limit $R_m(z_1, z_2, t) = \lim_{k \to \infty} R_m(z_1, \pi_k(z_2), t)$ holds with probability one,
as $z_2 \mapsto R_m(z_1, z_2, t)$ is continuous by assumption. 
Applying the supremum over  $z_1$, $z_2$, $t$ after taking the norm $\norm{\cdot}$, and using the triangle inequality we obtain: 
\begin{equation}
\label{eqn:chaining-identity}
\begin{split} 
	\sup_{z_1 \in \mathcal{Z}_1, z_2 \in \mathcal{Z}_2, t\in \R} \norm{R_m(z_1, z_2, t)} &\le 
		\sup_{z_1 \in \mathcal{Z}_1, t\in \R} \norm{R_m(z_1, z_{2, 0}, t)} \\
		 &~~~~~~~+  \sum_{k=1}^\infty \sup_{z_1 \in \mathcal{Z}_1, z_2 \in \mathcal{Z}_2, t\in \R} \norm{R_m(z_1, \pi_k(z_2), t) - R_m(z_1, \pi_{k-1}(z_2), t)}.
\end{split} 
\end{equation}
We will bound the tail behavior of every single term in this sum. To this end, the following lemma is useful.  For every $\delta \in (0, 1)$, define: 
\begin{equation}
\label{eqn:definition-Delta-delta}
	\Delta(\delta) := \sqrt{\frac{1}{m}\cdot \left(d + \VC(\W)\log(m)+\log(\frac{1}{\delta}) \right)}.
\end{equation} 

\begin{lemma}
\label{lemma:one-single-term}
There exists a universal constant $C > 0$ such that for any given pair $z_2, z_2' \in \mathcal{Z}_2$, the following occurs with an inner 
probability at least $1-\delta$: 
\begin{equation*}
	\sup_{z_1 \in \mathcal{Z}_1, t\in \R}   \norm{R_m(z_1, z_2, t) - R_m(z_1, z_2', t)}  \le C \sigma \cdot \norm{z_2 - z_2’} \cdot \max\{\Delta(\delta), \Delta(\delta)^2\}.
\end{equation*} 
Similarly, for the center $z_{2, 0}$ of $\mathcal{Z}_2$, the following occurs with an inner probability at least $1-\delta$: 
\begin{equation*}
	\sup_{z_1 \in \mathcal{Z}_1, t\in \R} \norm{R_m(z_1, z_{2, 0}, t)}  \le C \sigma_0 \cdot \max\{\Delta(\delta), \Delta(\delta)^2\}.
\end{equation*}
\end{lemma} 
The proof of Lemma~\ref{lemma:one-single-term} is deferred to Section~\ref{sec:bound-on-E-1-u}. 

Back to our bound of the tail behavior of every term on the RHS of equation~\eqref{eqn:chaining-identity}. Fix $\delta \in (0, 1)$. 
Let us denote $\delta_0 = \delta/2$ and 
\begin{equation*}
	\delta_k = \delta \cdot 2^{-(k+1)} \cdot (|\mathcal{Z}_{2, \eps_k}| |\mathcal{Z}_{2, \eps_{k-1}}|)^{-1}
\end{equation*}
for each $k \ge 1$. Below we use $C$ to denote the universal constant that appears in Lemma~\ref{lemma:one-single-term}.

First, by Lemma~\ref{lemma:one-single-term}, with an inner probability at least $1-\delta/2$: 
\begin{equation*}
	\sup_{z_1 \in \mathcal{Z}_1, t\in \R} \norm{R_m(z_1, z_{2, 0}, t)}  \le C \sigma_0 \cdot \max\{\Delta(\delta_0), \Delta(\delta_0)^2\}.
\end{equation*}
This controls the first term on the RHS of equation~\eqref{eqn:chaining-identity}.

Next, the adjustment of $\delta$ into $\delta_k$ allows the use of a union bound to extend the probability guarantees 
from Lemma~\ref{lemma:one-single-term} from individual pairs to all pairs within the cover sets $\mathcal{Z}_{2, \eps_k}$
and $\mathcal{Z}_{2, \eps_{k-1}}$. Specifically, for each $k \ge 1$, 
with an inner probability at least $1-\delta \cdot 2^{-(k+1)}$, there is the bound: 
\begin{equation*}
	\sup_{z_1 \in \mathcal{Z}_1, z_2 \in \mathcal{Z}_2, t\in \R} \norm{R_m(z_1, \pi_k(z_2), t) - R_m(z_1, \pi_{k-1}(z_2), t)}
		 \le 3C \sigma \eps_k \cdot \max\{\Delta(\delta_k), \Delta(\delta_k)^2\}.
\end{equation*} 
This result holds because the cardinality of pairs of $\{(\pi_k(z_2), \pi_{k-1}(z_2))\}$ over $z_2 \in \mathcal{Z}_2$
is capped by $|\mathcal{Z}_{2, \eps_k}||\mathcal{Z}_{2, \eps_{k-1}}|$, and there is the bound 
$\norm{\pi_k(z_2) - \pi_{k-1}(z_2)} \le 3\eps_k$ that holds for every $z_2 \in \mathcal{Z}_2$. 

By substituting all the above high probability bounds into the RHS of equation~\eqref{eqn:chaining-identity}, and using the union bounds, 
noting $\sum_{k=0}^\infty \delta \cdot 2^{-(k+1)} = \delta$,
we obtain with an inner probability at least $1-\delta$: 
\begin{equation*}
	\sup_{z_1 \in \mathcal{Z}_1, z_2 \in \mathcal{Z}_2, t\in \R} \norm{R_m(z_1, z_2, t)}
		\le C \sigma_0 \cdot \max\{\Delta(\delta_0), \Delta(\delta_0)^2\}
			+ 3C  \sigma \cdot \sum_{k=1}^\infty  \eps_k \cdot \max\{\Delta(\delta_k), \Delta(\delta_k)^2\}.
\end{equation*} 
To accumulate terms on the RHS, we deduce a basic bound whose proof is deferred to Section~\ref{sec:proof-a-basic-bound}.

\vspace{.2cm}
\begin{lemma}
\label{lemma:basic-bound}
For some universal constant $c > 0$: 
\begin{equation*}
	 \sum_{k=1}^\infty  \eps_k \cdot \max\{\Delta(\delta_k), \Delta(\delta_k)^2\} \le c r \max\{\Delta(\delta), \Delta(\delta)^2\}.
\end{equation*}
\end{lemma}
Ultimately, this proves the existence of a universal constant $\bar{C} > 0$ such that
\begin{equation*}
	\sup_{z_1 \in \mathcal{Z}_1, z_2 \in \mathcal{Z}_2, t\in \R} \norm{R_m(z_1, z_2, t)} \le \bar{C} (\sigma_0 + \sigma r) \cdot \max\{\Delta(\delta), \Delta(\delta)^2\}
\end{equation*}
holds  with probability at least $1-\delta$.
The proof of Theorem~\ref{theorem:uniform-convergence-theorem} is then complete.

\subsubsection{Bounding Individual Term using VC-Dimension (Proof of Lemma~\ref{lemma:one-single-term})} 
\label{sec:bound-on-E-1-u}
We give a unifying proof for the two statements in Lemma~\ref{lemma:one-single-term}. 
To simplify the notation, we define:
\begin{equation*}
	\hh_{z_1, t}(\xi) = \one \{q(z_1, \xi) \ge t\}.
\end{equation*}
We will show that it suffices to prove the following more general statement.  

\vspace{.2cm}
\begin{lemma}
\label{lemma:key-idea-behind-the-lemma}
There exists a universal constant $C > 0$ such that the following holds. Define: 
\begin{equation}
\label{eqn:expression-abstract}
	W_m(z_1, t) =  \frac{1}{m} \sum_{i=1}^m \hh_{z_1, t}(\xi_i) u(\xi_i) - \E[\hh_{z_1, t}(\xi) u(\xi)].
\end{equation}
where $u(\xi)$ in $\R^d$ is $\bar{\sigma}$ subexponential. Then it happens with an inner probability at least $1-\delta$: 
\begin{equation*}	
	\sup_{z_1 \in \mathcal{Z}_1, t\in \R}  \norm{W_m(z_1, t)} \le C\bar{\sigma} \cdot \max\{\Delta(\delta), \Delta(\delta)^2\}.
\end{equation*}
where $\Delta(\delta)$ is defined in equation~\eqref{eqn:definition-Delta-delta}.
\end{lemma} 

We will first demonstrate how Lemma~\ref{lemma:key-idea-behind-the-lemma} implies
Lemma~\ref{lemma:one-single-term} and defer the proof of Lemma~\ref{lemma:key-idea-behind-the-lemma} to the end.
Introduce two notation: 
\begin{equation*}
	u_{z_2, z_2'}(\xi) =  w(z_2, \xi) - w(z_2', \xi),~~~u_{z_{2, 0}}(\xi) = w(z_{2, 0}, \xi).
\end{equation*}
%In the above, we suppress the notation dependence of $u$ onto $z_2, z_2'$ for convenience. 
These notation lead us to the following critical representation:
\begin{equation}
\label{eqn:expression-concrete}
\begin{split} 
	 R_m(z_1, z_2, t) - R_m(z_1, z_2', t)  &= \frac{1}{m} \sum_{i=1}^m \hh_{z_1, t}(\xi_i) u_{z_2, z_2'}(\xi_i) - \E\left[ \hh_{z_1, t}(\xi) u_{z_2, z_2'}(\xi)\right] \\
	 			R_m(z_1, z_{2, 0}, t) &= \frac{1}{m} \sum_{i=1}^m \hh_{z_1, t}(\xi_i) u_{z_{2, 0}}(\xi_i) - \E\left[ \hh_{z_1, t}(\xi) u_{z_{2, 0}}(\xi)\right]
\end{split} .
\end{equation}
Note the similarities between the expressions in equation~\eqref{eqn:expression-abstract} and 
in equation~\eqref{eqn:expression-concrete}. Because
$u_{z_2, z_2'}(\xi)$ and $u_{z_{2, 0}}(\xi)$ are subexponential random vectors, the conditions of 
Lemma~\ref{lemma:one-single-term} are met. More precisely, $u_{z_2, z_2'}(\xi)$
is $\sigma \norm{z_2 - z_2'}$ subexponential for every pair of $z_2, z_2' \in \mathcal{Z}_2$, 
and $u_{z_{2, 0}}(\xi)$ is $\sigma_0$ subexponential.
This allows us to leverage Lemma~\ref{lemma:key-idea-behind-the-lemma} to conclude 
that Lemma~\ref{lemma:one-single-term} follows.

In the remainder, we prove Lemma~\ref{lemma:key-idea-behind-the-lemma}. Our proof exploits the fact that the VC 
dimension of the set family $\W$ controls the Rademacher complexity associated with the indicators $\hh_{z_1, t}(\xi)$.
Additionally, we utilize calculus rules for \emph{outer probability measures}~\cite[Section 1.2-5]{VanDerVaartWe96} to address
measurability issues. Recall the notion of \emph{outer integral}~\cite[Section 1.2]{VanDerVaartWe96}. 

\vspace{.3cm}
\begin{definition}
\label{definition:outer-integral}
Given a probability space $(\Xi, \mathscr{G}, \P)$, and 
an arbitrary map $T: \Xi \mapsto \R \cup \{+\infty\}$, we define the outer integral of $T$ by: 
\begin{equation*}
	\E^*[T]= \inf\{\E[U]: U \ge T,~~U: \Xi \mapsto \R 
\cup \{+\infty\}~\text{is measurable},~\text{and}~\E[U]~\text{exists}\}.
\end{equation*} 
\end{definition} 
%Outer integrals exhibit subadditivity but not additivity in general: $\E^*[T_1 + T_2]\le \E^*[T_1]+\E^*[T_2]$. 
Outer integrals exhibit monotone property: if $T_1(\xi) \le T_2(\xi)$ for every $\xi$, then $\E^*[T_1] \le \E^*[T_2]$. 
This is essentially all we need, as the remaining properties of outer integral will be reference from the monograph~\cite{VanDerVaartWe96}.
To ensure clarity when discussing the calculus rules of the \emph{outer integral}, we will explicitly 
reference the relevant chapters in the monograph~\cite{VanDerVaartWe96}.

\begin{proof}[Proof of Lemma~\ref{lemma:key-idea-behind-the-lemma}]
We define the moment-generating function, using the notion of outer integral (see Definition~\ref{definition:outer-integral}): 
\begin{equation}
	U(\tau) \defeq \E^*\left[\exp \left(\tau \cdot 
		\sup_{z_1 \in \mathcal{Z}_1, t\in \R} \norm{W_m(z_1, t)}\right)\right].
\end{equation} 
We aim to find a tight upper bound for $U(\tau)$ for every $\tau > 0$, which will allow us to infer a high probability upper bound
under the outer measure onto the random variable: 
\begin{equation*}
	\sup_{z_1 \in \mathcal{Z}_1, t\in \R} \norm{W_m(z_1, t)}.
\end{equation*}

Let $V_{1/2}$ denote a $1/2$ cover of unit ball in $\R^d$ with $\log |V_{1/2}| \le d \log 6$. Following a standard argument
(see, e.g.,~\cite[Exercise 4.2.2]{Vershynin18}), we know that for every $\bar{v} \in \R^d$: 
\begin{equation*}
	\norm{\bar{v}} \le 2\max_{v \in V_{1/2}} \langle \bar{v}, v \rangle. 
\end{equation*}
Applying this inequality to $W_m(z_1, t)$ for every $z_1 \in \mathcal{Z}_1$ and $t\in \R$, 
we obtain: 
\begin{equation}
\begin{split} 
	U(\tau)
		&\le \E^*\left[\exp\left(
		2\tau \cdot \sup_{z_1 \in \mathcal{Z}_1, t \in \R} ~\max_{v \in V_{1/2}} \langle W_m(z_1, t), v\rangle
		\right) \right] \\
		%&= 
		%2\sup_{\theta_1 \in \Theta} \sup_{\theta_2 \in \Theta_\eps}  \max_{v \in V_{1/2}} \frac{1}{n} \sum_{i=1}^n \hh_{\theta_1}(X_i) u_{\theta_2, v}(X_i).
\end{split} 
\end{equation}
%We aim to establish a high probability upper bound for the random variable on the RHS. To do so, we define its moment-generating function: 
%\begin{equation*}
%	U(\tau) \defeq \E\left[\exp \left(\tau \cdot 
%		\sup_{\theta_1 \in \Theta} \sup_{\theta_2 \in \Theta_\eps} \max_{v \in V_{1/2}} \langle K_n(\theta_1, \theta_2) - K(\theta_1, \theta_2), v\rangle\right)\right].
%\end{equation*} 
%Our goal is to find an upper bound for $U(\tau)$ for every $\tau > 0$. 

We employ symmetrization techniques from empirical process theory to upper bound the RHS. 
To simplify the notation, we introduce a function
$u_v(\xi) = \langle u(\xi), v\rangle$, which 
allows us to express: 
\begin{equation*}
	\langle W_m(z_1, t), v\rangle = \frac{1}{m} \sum_{i=1}^m \hh_{z_1, t}(\xi_i) u_v(\xi_i) - \E\left[\hh_{z_1, t}(\xi) u_v(\xi)\right].
\end{equation*}
Furthermore, we introduce i.i.d. Rademacher random variables $\{\rad_i\}_{i=1}^m$ 
with $\P(\rad_i = \pm 1) = 1/2$, independent of the data samples $\{\xi_i\}_{i=1}^m$. To simplify 
our notation, we use $\rad_{1:m}$ to denote the vector $(\rad_1, \rad_2, \ldots, \rad_m)$ and 
$\xi_{1:m}$ to denote the vector $(\xi_1, \xi_2, \ldots, \xi_m)$.

Applying symmetrization techniques from empirical process theory for the outer integral~\cite[Lemma 2.3.1]{VanDerVaartWe96}, 
we deduce an upper bound for $U(\tau)$: 
\begin{equation}
\label{eqn:U-tau-first-bound}
\begin{split} 
	U(\tau) \le \E^*\left[\exp \left(4\tau \cdot 
		\sup_{z_1 \in \mathcal{Z}_1, t\in \R} \max_{v \in V_{1/2}} 
			\frac{1}{m}\sum_{i=1}^m \rad_i \hh_{z_1, t}(\xi_i) u_v(\xi_i) \right)\right]
\end{split} 
\end{equation} 
where $\E^*$ is understood as the outer integral on the product probability space involving the data 
samples $\xi_{1:m}$ and the independent Rademacher random variables $\rad_{1:m}$. 
To be more formal, we follow 
\cite[Chapter 2.3]{VanDerVaartWe96}. Recall that  a single random variable $\xi$ can be defined on 
the probability space $(\Xi, \mathscr{G}, \P)$. Thus, we can define our data samples $\xi_{1:m}$ on the product probability space 
$(\Xi^{\otimes m}, \sigma(\mathscr{G}^{\otimes m}), \P^{\otimes m})$, where 
$\Xi^{\otimes m}$ represents the $m$-fold Cartesian product of $\Xi$, $\sigma(\mathscr{G}^{\otimes m})$ 
represents the $\sigma$-field generated by the $m$-fold product of $\mathscr{G}$, and $\P^{\otimes m}$ 
is the $m$-fold product of $\P$, where $\xi_1, \xi_2, \ldots, \xi_m$ corresponds to the coordinate projections. 
Then, we introduce 
Rademacher random variables $\rad_{1:m} = (\rad_1, \rad_2, \ldots, \rad_m)$ 
defined on another probability space $(\Xi', \mathscr{H}, \P')$. We finally combine these to 
define the final product probability space of interest: 
\begin{equation}
\label{eqn:product-space}
	(\Xi^{\otimes m} \times \Xi', \sigma(\mathscr{G}^{\otimes m} \otimes \mathscr{H}), \P^{\otimes m} \otimes \P').
\end{equation}
In this product probability space, the sample $\xi_{1:m}$ corresponds to the first $m$ coordinates projections and the additional 
Rademacher random variables $\rad_{1:m}$ depending only on the $(m+1)$st coordinate. The outer integral $\E^*$ in 
equation~\eqref{eqn:U-tau-first-bound} then refers to the outer integral for this probability space. 

Lemma~\ref{lemma:measurability-issue} is the critical component that resolves the measurability issue.  From the empirical process 
theory, we know that measurability of certain forms is needed at this point, as Fubini's theorem is not valid for 
outer integrals~\cite[Chapter 2.3]{VanDerVaartWe96}. Lemma~\ref{lemma:measurability-issue}  follows 
the established ideas to resolve the measurability issues, 
see, e.g., discussions in the monograph on the measurable class~\cite[Definition 2.3.3]{VanDerVaartWe96}. 
The proof of Lemma~\ref{lemma:measurability-issue} is deferred to Section~\ref{sec:tackle-measurability}.
\begin{lemma}
\label{lemma:measurability-issue}
The mapping: 
\begin{equation*}
	(\xi_{1:m}, \rad_{1:m}) \mapsto 
	\sup_{z_1 \in \mathcal{Z}_1, t\in \R} \max_{v \in V_{1/2}} 
			\frac{1}{m}\sum_{i=1}^m \rad_i \hh_{z_1, t}(\xi_i) u_v(\xi_i) 
\end{equation*}
is measurable on the product probability space defined in equation~\eqref{eqn:product-space}.
\end{lemma}

Lemma~\ref{lemma:measurability-issue} provides the justification for replacing the outer integral $\E^*$ on the RHS 
of equation~\eqref{eqn:U-tau-first-bound} with the standard integral $\E$. Indeed, equation~\eqref{eqn:U-tau-first-bound}
now leads to
\begin{equation}
\label{eqn:U-tau-first-bound-nice}
\begin{split} 
	U(\tau) \le \E\left[\exp \left(4\tau \cdot 
		\sup_{z_1 \in \mathcal{Z}_1, t\in \R} \max_{v \in V_{1/2}} 
			\frac{1}{m}\sum_{i=1}^m \rad_i \hh_{z_1, t}(\xi_i) u_v(\xi_i) \right)\right]
\end{split} 
\end{equation} 
where the expectation on the RHS now is the standard expectation, taken jointly over the samples $\xi_{1:m}$ and the independent 
Rademacher random variables $\xi_{1:m}$. 

To further upper bound the expectation on the RHS of inequality~\eqref{eqn:U-tau-first-bound}, 
we first consider conditioning on the samples $\xi_{1:m}$. This means that we condition on
the values of $\xi_1, \xi_2, \ldots, \xi_m$, and study the quantity: 
\begin{equation*}
\begin{split}
	V(\tau, \xi_{1:m}) = \E_{\rad_{1:m}} \left[\exp \left(\tau \cdot \sup_{z_1 \in \mathcal{Z}_1, t\in \R}\max_{v \in V_{1/2}} 
		\frac{1}{m} \sum_{i=1}^m \rad_i \hh_{z_1, t}(\xi_i) u_v(\xi_i)\right) \right]
\end{split} 
\end{equation*}
where the symbol $\E_{\rad_{1:m}}$ means the standard expectation taken solely over the Rademacher variables $\rad_{1:m}$. 
As the Fubini's theorem holds for the standard expectation, under this notation, the inequality \eqref{eqn:U-tau-first-bound-nice} 
then becomes: 
\begin{equation}
\label{eqn:U-tau-second-bound}
	U(\tau) \le \E_{\xi_{1:m}} [V(4\tau, \xi_{1:m})]
\end{equation}
where the expectation on the RHS is taken over the samples $\xi_{1:m}$. That is to say, to upper bound $U(\tau)$, we 
first focus on bounding $V(\tau, \xi_{1:m})$ for every possible value of $\xi_{1:m}$. 

Fix $\xi_{1:m}$. Consider the set:
\begin{equation*}
\begin{split} 
	H(\xi_{1:m}) &:= \left\{(\hh_{z_1, t}(\xi_1), \hh_{z_1, t}(\xi_2), \ldots, \hh_{z_1, t}(\xi_m)): z_1 \in \mathcal{Z}_1, t\in \R\right\} \subseteq \{0, 1\}^m
\end{split} 
\end{equation*}
which represents all binary outcomes of applying $\hh_{z_1}$ to the data points $\xi_{1:m}$ across all $z_1 \in \mathcal{Z}_1$. Given that 
the VC-dimension of the set collection
\begin{equation*}
	\mathcal{W}:= \left\{\{\xi \in \Xi: q(z_1, \xi) \ge t\} \mid z_1 \in \mathcal{Z}_1, t\in \R\right\}
\end{equation*}
is upper bounded by $\VC(\W)$, Sauer-Shelah Lemma then constrains the size of $H(\xi_{1:m})$ to obey 
\begin{equation}
\label{eqn:sauer-lemma-result}
	|H(\xi_{1:m})| \le (2m)^{|\VC(\W)|}.
\end{equation}
For a reference of Sauer-Shelah Lemma, see, e.g.,~\cite{Sauer72, Shelah72}
or~\cite[Theorem 8.3.16]{Vershynin18}. 

Notably, this definition also allows replacing the supremum over continuous $z_1$ values 
with the supremum over the finite binary labeling set $H(\xi_{1:m})$. Specifically, there is the upper bound: 
%for every collection of points $x_1, x_2, \ldots, x_n$, the set
%\begin{equation*}
%\begin{split} 
%	H(x_{1:n}) &= \left\{(\hh_\theta(x_1), \hh_\theta(x_2), \ldots, \hh_\theta(x_n)): \theta \in \R^p\right\}  \subseteq \{0, 1\}^n
%\end{split} 
%\end{equation*}
%has its size bounded by $(n+1)^{|\V|}$. 
%The definition of $H(x_{1:n})$ also leads to: 
%\begin{equation*}
%	\sup_{\theta_1 \in \Theta_1}
%		\frac{1}{n} \sum_{i=1}^n \eps_i \hh_{\theta_1}(x_i) u_{\theta_2, v}(x_i) 
%	= \sup_{b \in H(x_{1:n})} \frac{1}{n} \sum_{i=1}^n \eps_i b_i u_{\theta_2, v}(x_i) 
%\end{equation*} 
%and extending this to the supremum over $\theta_2, v$ leads to:  
\begin{equation*}
	\sup_{z_1 \in \mathcal{Z}_1, t\in \R}  \max_{v \in V_{1/2}} 
		\frac{1}{m} \sum_{i=1}^m \rad_i \hh_{z_1, t}(\xi_i) u_v(\xi_i)
	\le \max_{b \in H(\xi_{1:m})} \max_{v \in V_{1/2}} 
	\frac{1}{m} \sum_{i=1}^m \rad_i b_i u_v(\xi_i).
\end{equation*} 
On the RHS, $b \in H(\xi_{1:m}) \subseteq \{0, 1\}^m$ denotes a vector whose components are $b = (b_1, b_2, \ldots, b_m)$.

This establishes an upper bound on $V(\tau, \xi_{1:m})$ that applies for every $\tau > 0$: 
\begin{equation*}
	V(\tau, \xi_{1:m}) \le \E_{\rad_{1:m}} \left[\exp \left(\tau \cdot  \max_{b \in H(\xi_{1:m})}\max_{v \in V_{1/2}} 
		\frac{1}{m} \sum_{i=1}^m \rad_i b_i u_v(\xi_i)\right) \right].
\end{equation*} 
By bounding the exponential of a supremum with the sum of exponentials across all configurations, we can further upper bound 
$V(\tau, \xi_{1:m})$ and derive: 
\begin{equation}
\label{eqn:expression-complicated}
\begin{split} 
	V(\tau, \xi_{1:m})  &\le \E_{\rad_{1:m}} \left[\sum_{b \in H(\xi_{1:m})} \sum_{v \in V_{1/2}}\exp \left(
		\frac{\tau}{m} \cdot \sum_{i=1}^m \rad_i b_i u_{v}(\xi_i)\right) \right] \\
		%&=  \sum_{\theta_2 \in \Theta_\eps} \sum_{v \in V_{1/2}} \sum_{b \in H(x_{1:n})} \E_{\eps_{1:n}} \left[\exp \left(
		%\frac{2\tau}{n} \cdot \sum_{i=1}^n \eps_i b_i u_{\theta_2, v}(x_i)\right) \right] \\
		&=  \sum_{b \in H(\xi_{1:m})} \sum_{v \in V_{1/2}} ~\prod_{i=1}^m \E_{\rad_{i}} \left[\exp \bigg(
		\frac{\tau}{m} \cdot \rad_i b_i u_{v}(\xi_i)\bigg) \right]
\end{split}
\end{equation}
where the symbol $\E_{\rad_{i}}$ means the expectation is taken solely over the Rademacher variables $\rad_{i}$. 

Given that for any $b_i \in \{0, 1\}$ and $q_i \in \R$, the expectation over $\rad_i$ obeys: 
\begin{equation*}
	\E_{\rad_i} [\exp(\rad_i b_i q_i)] = \cosh (b_i q_i) \le \cosh(q_i) = \E_{\rad_i}[\exp(\rad_i q_i)]
\end{equation*}
this allows us to upper bound the RHS in equation \eqref{eqn:expression-complicated}, removing the dependence on $b_{1:m}$: 
%\begin{equation*}
%	\E_{\eps_{i}} \left[\exp \left(
%		\frac{2\tau}{n} \cdot \eps_i b_i u_{\theta_2, v}(x_i)\right) \right] \le \E_{\eps_{i}} \left[\exp \left(
%		\frac{2\tau}{n} \cdot \eps_i u_{\theta_2, v}(x_i)\right) \right]. 
%\end{equation*} 
%This removes $b_{1:n}$ in the expectation. After we substitute it into inequality~\eqref{eqn:expression-complicated}, we obtain: 
\begin{equation*}
\begin{split}
	V(\tau, \xi_{1:m}) &\le \sum_{v \in V_{1/2}} \sum_{b \in H(\xi_{1:m})} \prod_{i=1}^m \E_{\rad_{i}} \left[\exp \left(
		\frac{\tau}{m} \cdot\rad_i u_{v}(\xi_i)\right) \right] \\
		&= |H(\xi_{1:m})| \sum_{v \in V_{1/2}}  \prod_{i=1}^m \E_{\rad_{i}} \left[\exp \left(
		\frac{\tau}{m} \cdot \rad_i u_{v}(\xi_i)\right) \right].
\end{split}
\end{equation*} 
Given the bound we have established for $|H(\xi_{1:m})| \le (2m)^{|\VC(\W)|}$ in equation~\eqref{eqn:sauer-lemma-result}, we deduce:  
\begin{equation}
	V(\tau, \xi_{1:m}) \le (2m)^{\VC(\W)} \sum_{v \in V_{1/2}}  \prod_{i=1}^m \E_{\rad_{i}} \left[\exp \left(
		\frac{\tau}{m} \cdot \rad_i u_{v}(\xi_i)\right) \right].
\end{equation}
Notably, this bound applies to every possible $\xi_{1:m}$ value. 

Finally, we substitute this bound into inequality~\eqref{eqn:U-tau-second-bound}, resulting in: 
\begin{equation}
\begin{split}
	U(\tau) &\le \E_{\xi_{1:m}}[V(4\tau, \xi_{1:m})]  \le (2m)^{\VC(\W)} \sum_{v \in V_{1/2}}  \prod_{i=1}^m \E_{\rad_{i}, \xi_i} \left[\exp \left(
		\frac{4\tau}{m} \cdot \rad_i u_{v}(\xi_i)\right) \right]
\end{split} 
\end{equation}
where $\E_{\rad_i, \xi_i}$ denotes the joint expectation taken over $\rad_i$ and $\xi_i$. 

%Let us denote $\bar{\sigma} = \sigma \norm{z_2 - z_2'}$. 
Since $u_{v}(\xi)$ is $\bar{\sigma}$-sub-exponential, then $\rad_i u_{v}(\xi_i)$ is also 
$\bar{\sigma}$-sub-exponential  for every $v$ by definition. Using~\cite[Proposition 2.7.1(e)]{Vershynin18}, this implies that,
for every $\theta_2$, $v$, $\tau \in (0, cm/\bar{\sigma})$: 
\begin{equation*}
	\E_{\rad_{i}, \xi_i} \left[\exp \left(
		\frac{4\tau}{m} \cdot \rad_i u_{v}(\xi_i)\right) \right] \le  \exp \left(C \frac{\bar{\sigma}^2\tau^2}{m^2}\right)
\end{equation*}
where $c, C > 0$ are absolute constants. 
Because $|V_{1/2}| \le 6^d$, 
we further obtain for $\tau \in (0, cm/\bar{\sigma})$: 
\begin{equation}
\label{eqn:bound-on-the-MGF}
	U(\tau) \le (2m)^{\VC(\W)} 6^d \exp \left(C \frac{\bar{\sigma}^2\tau^2}{m}\right).
\end{equation} 
Recall the definition of $U(\tau)$, which yields: 
\begin{equation*}
	\E^*\left[\exp \left(\tau \cdot 
		\sup_{z_1 \in \mathcal{Z}_1, t\in \R} \norm{W_m(z_1, t)}\right)\right] \le (2m)^{\VC(\W)} 6^d \exp \left(C \frac{\bar{\sigma}^2\tau^2}{m}\right).
\end{equation*} 
By applying Markov's inequality under the outer integral~\cite{VanDerVaartWe96} (which follows from the monotone property
of \emph{outer integral} discussed below Definition~\ref{definition:outer-integral}), we establish 
that there exists a universal constant $c > 0$ such that for every $u \ge 0$: 
\begin{equation*}
\begin{split} 
	\P^*\left(\sup_{z_1 \in \mathcal{Z}_1, t\in \R}   \norm{W_m(z_1, t)} \ge u\right) %&\le (n+1)^{|\mathcal{V}|} (18r/\eps)^p \cdot  \min_{\tau \in (0, cn/\sigma)} \exp \left(C\sigma^2 \tau^2/n - \tau u/3\right) \\
	&\le (2m)^{\VC(\W)} 6^d \cdot \exp \left(-c\min\left\{\frac{m u^2}{\bar{\sigma}^2}, \frac{mu}{\bar{\sigma}}\right\}\right).
\end{split} 
\end{equation*}
Equivalently, there exists a universal constant $C > 0$ such that for every $\delta \in (0, 1)$, the following 
happens with an inner probability at least $1-\delta$: 
\begin{equation*}
	\sup_{z_1 \in \mathcal{Z}_1, t \in \R}   \norm{W_m(z_1, t)}  \le C \bar{\sigma} \cdot \max\{\Delta(\delta), \Delta(\delta)^2\},
\end{equation*} 
where $\Delta(\delta)$ is defined in equation~\eqref{eqn:definition-Delta-delta}.
\end{proof}

\subsubsection{A Basic Bound} 
\label{sec:proof-a-basic-bound}
This section proves Lemma~\ref{lemma:basic-bound}. 
By definition: 
\begin{equation*}
\begin{split} 
	\Delta(\delta_k) &= \sqrt{\frac{1}{m} \cdot \left(d + \VC(\W) \log m + \log (\frac{1}{\delta}) + \log (2^{k+1}) + 
		\log |\mathcal{Z}_{2, \eps_k}| + \log |\mathcal{Z}_{2, \eps_{k-1}}|\right)}.
\end{split} 
\end{equation*}
Recall that $\log |\mathcal{Z}_{2, \eps_k}| \le d \log(2+r/\eps_k) \le d \log (2+2^k)$. This yields that 
\begin{equation*}
\begin{split} 
		\Delta(\delta_k) %&\le   \sqrt{\frac{1}{m} \cdot \left(p + \VC(\W) \log m + \log (\frac{1}{\delta})\right)}
			%+ \sqrt{\frac{(k+1) \log 2}{m}} + \sqrt{\frac{2p}{m} \log (2+ 2^k)} \\
			&\le  \sqrt{\frac{1}{m} \cdot \left(d + \VC(\W) \log m + \log (\frac{1}{\delta})\right)} + \sqrt{\frac{12 kd}{m}} %+ \sqrt{\frac{2p}{m} \log (2+ 2^k)}
\end{split} 
\end{equation*} 
As a result, there exists a universal constant $c > 0$ such that
\begin{equation*}
\begin{split} 
	\sum_{k=1}^\infty \eps_k \Delta(\delta_k) &\le \left(\sum_{k=1}^\infty \eps_k\right) \cdot 
		\sqrt{\frac{1}{m} \cdot \left(d + \VC(\W) \log m + \log (\frac{1}{\delta})\right)}
		+  \sum_{k=1}^\infty \eps_k\sqrt{\frac{12 kd}{m}} \\
		&\le c r \cdot  \sqrt{\frac{1}{m} \cdot \left(d + \VC(\W) \log m + \log (\frac{1}{\delta})\right)} = cr \Delta(\delta).
\end{split} 
\end{equation*}
Similarly, one can prove the existence of a universal constant $c > 0$ such that
\begin{equation*}
	\sum_{k=1}^\infty \eps_k \Delta(\delta_k)^2 \le  cr \Delta(\delta)^2.
\end{equation*}
This then concludes the existence of a universal constant $c > 0$ such that 
\begin{equation*}
	\sum_{k=1}^\infty \eps_k \max\{\Delta(\delta_k), \Delta(\delta_k)^2\} \le \sum_{k=1}^\infty \eps_k \Delta(\delta_k) + \sum_{k=1}^\infty \eps_k \Delta(\delta_k)^2
		\le 2cr \max\{\Delta(\delta), \Delta(\delta)^2\}.
\end{equation*}

\subsubsection{A Result on Measurability}
\label{sec:tackle-measurability}
We prove Lemma~\ref{lemma:measurability-issue}. Let $\mathcal{Z}_1^o$ and $\R^o$ be any 
countable dense subsets of $\Z_1$ and $\R$, respectively. Our main observation is that the 
supremum over the uncountable sets $\Z_1, \R$ is equal to the supremum over the countable 
sets $\Z_1^o, \R^o$, holding for every data instances $\xi_{1:m}, \rad_{1:m}$: 
\begin{equation}
\label{eqn:reduce-uncountable-to-countable}
\sup_{z_1 \in \mathcal{Z}_1, t\in \R} \max_{v \in V_{1/2}} 
			\frac{1}{m}\sum_{i=1}^m \rad_i \hh_{z_1, t}(\xi_i) u_v(\xi_i) 
	=\sup_{z_1 \in \mathcal{Z}_1^o, t\in \R^o} \max_{v \in V_{1/2}} 
			\frac{1}{m}\sum_{i=1}^m \rad_i \hh_{z_1, t}(\xi_i) u_v(\xi_i) 
\end{equation}
Since the RHS is the pointwise supremum over \emph{countable} measurable functions of $\xi_{1:m}, \rad_{1:m}$, 
it is measurable. Therefore, the LHS must also be measurable under the product space. To complete 
the proof of Lemma~\ref{lemma:measurability-issue}, it suffices to prove equation~\eqref{eqn:reduce-uncountable-to-countable}.

Fix the values of $\xi_{1:m}$ and $\rad_{1:m}$. It suffices to show that for every $z_1 \in \Z_1$ and $t \in \R$, there exist 
$z_1^o \in \Z_1^o$ and $t^o \in \R^o$ such that for every $i = 1, 2, \ldots, m$:
\begin{equation}
\label{eqn:identities-measurability}
	\hh_{z_1, t}(\xi_i) = \hh_{z_1^o, t^o}(\xi_i).
\end{equation}
Recall that $\hh_{z_1, t}(\xi) = \one\{q(z_1, \xi) \ge t\}$.  First,  we note an important observation, which uses the fact that 
$t \mapsto \one\{w \ge t\}$ is left continuous for every real value $w$. For every $z_1$, 
there exists $t^o \in \R^o$ with $t^o < t$ such that for every $i = 1, 2, \ldots, m$:
\begin{equation*}
	\hh_{z_1, t}(\xi_i) = \hh_{z_1, t^o}(\xi_i)~~\text{and}~~q(z_1, \xi_i) - t^o \neq 0 .
\end{equation*} 
Fix this $t^o$. Next, since $z_1 \mapsto q(z_1, \xi)$ is continuous for every $\xi$, this implies the existence of $z_1^o \in \Z_1^o$ 
close to $z_1$ such that for every $i = 1, 2, \ldots, m$:
\begin{equation*}
	 \hh_{z_1, t^o}(\xi_i) =  \hh_{z_1^o, t^o}(\xi_i).
\end{equation*}
This proves that for every $z_1 \in \Z_1$ and $t \in \R$, there is the existence of the pair $z_1^o \in \Z_1^o$ and $t^o \in \R^o$
such that equation~\eqref{eqn:identities-measurability} holds for all $i = 1, 2, \ldots, m$.  This completes the proof of 
Lemma~\ref{lemma:measurability-issue}.

\section{Proof of Proposition~\ref{theorem:tight-nonsmooth-phase-retrieval-landscape}}
\label{sec:theorem-tight-nonsmooth-phase-retrieval-landscape}
Our approach comprises two main steps. First, we demonstrate that every 
stationary point of the empirical objective $\Phi_S$ must be approximately stationary 
for the population objective $\Phi$, leveraging the uniform convergence of 
subdifferentials in Theorem~\ref{corollary:phase-retrieval}. Second, we utilize established results from the 
literature~\cite[Section 5]{DavisDrPa20} on these (approximate) stationary points of $\Phi$ 
to pinpoint the location of empirical stationary points.

To set up the stage, we recall the location of the stationary point and approximate stationary point of the population objective $\Phi$
formally derived in~\cite{DavisDrPa20}. 

\vspace{.2cm}
\begin{proposition}[{\cite[Theorem 5.2]{DavisDrPa20}}]
\label{theorem:davis-dr-pa20-exact}
The stationary points of the population objective $\Phi$ are 
\begin{equation*}
	\{0\} \cup \{\pm \bar{x}\} \cup \{x: \langle x, \bar{x}\rangle = 0, \norm{x} = c\cdot \norm{\bar{x}}\}
\end{equation*} 
where $c > 0$ is the unique solution of the equation $\frac{\pi}{4} = \frac{c}{1+c^2} + \arctan(c)$. 
\end{proposition} 

The following Theorem~\ref{theorem:davis-dr-pa20} can be easily deduced from 
\cite[Theorem 5.2+Theorem 5.3]{DavisDrPa20}. Recall that $\mathcal{Z} = \{x: 0 \in \partial \Phi(x)\}$ denotes 
the set of stationary points of $\Phi$.

\vspace{.2cm}
\begin{proposition}[{\cite[Theorem 5.3]{DavisDrPa20}}]
\label{theorem:davis-dr-pa20}
There exist numerical constants $\gamma, C > 0$ such that the following holds. For any point $x \in \R^d$ with 
\begin{equation*}
	\dist(0; \partial \Phi(x)) \le \gamma \norm{x},
\end{equation*}
\vspace{.2cm}
it must be the case that $\norm{x} \le C\norm{\bar{x}}$ and $x$ satisfies $\dist(x, \mathcal{Z}) \le C \dist(0, \partial \Phi(x))$. 
\end{proposition}

The following result is immediate given~\cite[Corollary 6.3]{DavisDrPa20}.

\vspace{.2cm}
\begin{proposition}[{\cite[Corollary 6.3]{DavisDrPa20}}]
\label{proposition:stationary-points-empirical-phase-retrieval}
There exist a numerical constant $C > 0$ so that for $m \ge C d$, with an inner probability at least 
$1-2\exp(-d)$\footnote{The original expression in the paper~\cite{DavisDrPa20}, $1-2\exp(-\min\{c_2m, d^2\})$, should be corrected to 
$1-2\exp(-d)$. }, %, based on a personal communication with Damek Davis and Dmitriy Drusvyatskiy.},
the set of stationary points $\mathcal{Z}_S$ of $\Phi_S$ satisfies: 
\begin{equation*}
	\mathbb{D}(\mathcal{Z}_S, \mathcal{Z}) \le C \sqrt[4]{\frac{m}{d}} \norm{\bar{x}}.
\end{equation*} 
\end{proposition}

We shall build on top of these existing results to establish Proposition~\ref{theorem:tight-nonsmooth-phase-retrieval-landscape}.
First, Proposition~\ref{theorem:davis-dr-pa20-exact} and 
Proposition~\ref{proposition:stationary-points-empirical-phase-retrieval} imply that there exists a numerical 
constant $C_1  > 0$ such that: 
\begin{equation}
\label{eqn:high-probability-bound-on-phase-retrieval-all-points}
	\P_*\left(\norm{x} \le C_1\norm{\bar{x}}~~\forall x \in \mathcal{Z}_S\right) \ge 1-2e^{-d}.  
\end{equation}
Corollary~\ref{corollary:phase-retrieval-peeling} refines 
Corollary~\ref{corollary:phase-retrieval} using a standard peeling argument to substitute 
the norm $\norm{x}$ for the constant radius $r$ in the original probability expression in equation~\eqref{eqn:high-probability-bound-phase-retrieval}.
This adjustment allows varying distances $x$ from the origin, ranging from $\underline{r}$ to $r$, 
thus giving bounds on Hausdorff distances between $\partial \Phi$ and $\partial \Phi_S$ over different scales. 
For completeness, we give its proof in Section~\ref{sec:proof-corollary-phase-retrieval-peeling}.

\vspace{.5cm}

\begin{corollary}
\label{corollary:phase-retrieval-peeling}
Assume the measurement vector $a$ is $\sigma$-subgaussian, and $\E[|b|^2] < \infty$. 

Then there exists a universal constant $C > 0$ such that for every $\delta \in (0, 1)$ and $r \ge \underline{r} > 0$: 
\begin{equation}
\label{eqn:high-probability-bound-phase-retrieval-peeling}
	\P_* \left( \H(\partial \Phi(x), \partial \Phi_S(x)) \le C\sigma \norm{x}\cdot 
		\max\{\tilde{\Delta}_\Phi, \tilde{\Delta}_\Phi^2\}~~\forall x~\text{satisfying}~\underline{r} \le \norm{x} \le r\right) \ge 1-\delta.
\end{equation}
In the above,  
\begin{equation}
	\tilde{\Delta}_\Phi = \sqrt{\frac{1}{m} \cdot \left(d \log d\log m+ \log(\frac{\log(er/\underline{r})}{\delta})\right)}.
\end{equation}
\end{corollary} 

\vspace{0.2cm} 
We are now ready to prove Proposition~\ref{theorem:tight-nonsmooth-phase-retrieval-landscape}. We pick 
\begin{equation*}
	\underline{r} = C_1\sqrt{\frac{d}{m}} \norm{\bar{x}},~~~r = C_1\norm{\bar{x}},~~~\delta = e^{-d},~~~\sigma =4
\end{equation*}
in Corollary~\ref{corollary:phase-retrieval-peeling}. We then derive for some absolute constant $C_2$: 
\begin{equation}
\label{eqn:peeling-event}
\begin{split} 
	&\P_*\bigg(
		\H(\partial \Phi(x), \partial \Phi_S(x)) \le C_2 \norm{x}\cdot 
		\max\{\Delta_0, \Delta_0^2\}, \\
		&~~~~~~~~~~~~~~~~~~~~~~~~~~~~~~~~~~~~\forall x~\text{satisfying}~C_1\sqrt{d/m}\norm{\bar{x}} \le \norm{x} \le C_1\norm{\bar{x}} \bigg) \ge 1-e^{-d}
\end{split} 
\end{equation} 
where 
\begin{equation*}
	\Delta_0 = \sqrt{\frac{d}{m}  \log d\log m}.
\end{equation*} 
We can without loss of generality assume that 
$C_2 \max\{\Delta_0, \Delta_0^2\} \le \gamma$ where 
$\gamma$ is the numerical constant specified in Theorem~\ref{theorem:davis-dr-pa20}, since otherwise 
Proposition~\ref{theorem:tight-nonsmooth-phase-retrieval-landscape} trivially follows 
equation~\eqref{eqn:high-probability-bound-on-phase-retrieval-all-points}.

With equations~\eqref{eqn:high-probability-bound-on-phase-retrieval-all-points} and~\eqref{eqn:peeling-event}, 
we know that with an inner probability at least $1-3e^{-d}$, the following simultaneously happen: 
\begin{enumerate}
\item $\norm{x} \le C_1\norm{\bar{x}}$ for every $x \in \mathcal{Z}_S$.
\item $\H(\partial \Phi(x), \partial \Phi_S(x)) \le C_2 \norm{x}\cdot 
		\max\{\Delta_0, \Delta_0^2\}$ for every $x$ with $C_1\sqrt{d/m}\norm{\bar{x}} \le \norm{x} \le C_1\norm{\bar{x}}$.
\end{enumerate}
Consequentially, on this event, any stationary point $x^*$ of $\Phi_S$ must either obey: 
\begin{equation*}
	\norm{x^*} \le C_1\sqrt{d/m}\norm{\bar{x}}
\end{equation*}
or satisfies 
\begin{equation*}
	\H(\partial \Phi(x^*), \partial \Phi_S(x^*)) \le C_2 \norm{x^*}\cdot 
		\max\{\Delta_0, \Delta_0^2\}.
\end{equation*}
In the first case, since $0 \in \mathcal{Z}$ by Theorem~\ref{theorem:davis-dr-pa20-exact}, we immediately obtain: 
\begin{equation*}
	\dist(x^*, \mathcal{Z})  \le \norm{x^*} \le  C_1\sqrt{d/m}\norm{\bar{x}}.
\end{equation*}
In the second case, since $0 \in \partial \Phi_S(x^*)$, then $\dist(0, \partial \Phi(x^*)) \le \H(\partial \Phi(x^*), \partial \Phi_S(x^*))$
by definition of Hausdorff distance, and thus: 
\begin{equation*}
	\dist(0, \partial \Phi(x^*)) \le C_2 \norm{x^*}\cdot 
		\max\{\Delta_0, \Delta_0^2\}.
\end{equation*}
Since we have assumed that $C_2 \max\{\Delta_0, \Delta_0^2\} \le \gamma$, by Theorem~\ref{theorem:davis-dr-pa20} we obtain that 
$x^*$ must satisfy $\dist(x^*, \mathcal{Z}) \le C_3 \dist(0, \partial \Psi(x^*))$ for some numerical constant $C_3 > 0$, and thereby, 
\begin{equation*}
\begin{split} 
	\dist(x^*, \mathcal{Z}) &\le C_3 \dist(0, \partial \Psi(x^*)) 
		\le C_2 C_3 \norm{x^*}\cdot \max\{\Delta_0, \Delta_0^2\} \le C_1 C_2 C_3\norm{\bar{x}}\cdot \max\{\Delta_0, \Delta_0^2\}
\end{split} 
\end{equation*}
where the last inequality uses the fact that  $\norm{x} \le C_1\norm{\bar{x}}$ for every $x \in \mathcal{Z}_S$ on this event.

Summarizing, we have established with an inner probability at least $1-3e^{-d}$, any stationary point $x^*$ of $\Phi_S$ must obey: 
\begin{equation*}
	\dist(x^*, \mathcal{Z}) \le C \norm{\bar{x}} \cdot \max\{\Delta_0, \Delta_0^2\}
\end{equation*}
where $C > 0$ is an absolute constant. This completes the proof of Proposition~\ref{theorem:tight-nonsmooth-phase-retrieval-landscape}.

\subsection{Proof of Corollary~\ref{corollary:phase-retrieval-peeling}}
\label{sec:proof-corollary-phase-retrieval-peeling}
We can define a sequence $\{r^{(i)}\}_{i=1}^N$ so that $r^{(0)} = \underline{r}$, $r^{(N)} = r$, $r^{(i+1)}/r^{(i)} \le e$ and $N \le \log(er/\underline{r})$.
Under the condition of Corollary~\ref{corollary:phase-retrieval-peeling}, we can apply Theorem~\ref{corollary:phase-retrieval} and 
obtain for every $1\le i \le N$
\begin{equation*}
	\P^* \left( \sup_{x: \norm{x} \le r^{(i)}} \H(\partial \Phi(x), \partial \Phi_S(x)) \ge C\sigma r^{(i)}\cdot 
		\max\{\tilde{\Delta}_\Phi, \tilde{\Delta}_\Phi^2\}\right) \le \delta/N
\end{equation*}
In the above, $C$ is a universal constant. 
A union bound shows that: 
\begin{equation*}
 \P_* \left(\sup_{x: \norm{x} \le r^{(i)}} \H(\partial \Phi(x), \partial \Phi_S(x)) \le C\sigma r^{(i)}\cdot 
		\max\{\tilde{\Delta}_\Phi, \tilde{\Delta}_\Phi^2\}~~\text{for every $1\le i\le N$}\right) \ge 1-\delta.
\end{equation*}
On this event, for every $x$ with $\underline{r} \le \norm{x} \le r$, we find $1\le i(x) \le N$ such that 
$r^{(i(x)-1)} \le \norm{x} \le r^{(i(x))}$, and then we obtain that 
\begin{equation*}
	 \H(\partial \Phi(x), \partial \Phi_S(x)) \le C\sigma r^{(i(x))}\cdot 
		\max\{\tilde{\Delta}_\Phi, \tilde{\Delta}_\Phi^2\} \le Ce\sigma \norm{x} \max\{\tilde{\Delta}_\Phi, \tilde{\Delta}_\Phi^2\}
\end{equation*}
where the last inequality is due to $r^{(i(x))}/\norm{x} \le r^{(i(x))}/r^{(i(x)-1)} \le e$ by construction. Hence, we have proven the existence of a
numerical constant $C' > 0$ such that: 
\begin{equation*}
	 \P_* \left( \H(\partial \Phi(x), \partial \Phi_S(x)) \le C'\sigma \norm{x} \max\{\tilde{\Delta}_\Phi, \tilde{\Delta}_\Phi^2\}~~\text{for all $x$ such that $\underline{r} \le \norm{x} \le r$}\right) \ge 1-\delta.
\end{equation*} 
This completes the proof of Corollary~\ref{corollary:phase-retrieval-peeling}.

\section{Proof of Corollary~\ref{corollary:blind-deconvolution}}
We apply Theorem~\ref{theorem:final-result}. We first check the assumptions. 
For the nonsmooth $\loss(z) = |z|$, it is decomposed 
as $\loss = \loss \ns + \loss\sm$ where $\loss\sm(z) = -z$ and $\loss\ns(z) = 2(z)_+$.
\begin{itemize}
\item (Assumption~\ref{assumption:sub-exponential-assumption}). By definition,
	\begin{equation*}
		\grad c(x, \xi) = 
		\begin{bmatrix}
			\partial_{y} c(x, \xi)\\
			\partial_{w} c(x, \xi)
		\end{bmatrix}= \begin{bmatrix}
			\langle v, w\rangle u \\
			\langle u, y\rangle v
		\end{bmatrix},~~~(\loss\sm)'(c(x; \xi)) \grad c(x; \xi) = - \grad c(x, \xi).
	\end{equation*}
	At the center $x = (0, 0) \in \R^{d_1} \times \R^{d_2}$, the gradients vanish $\grad c(0, \xi) = (\loss\sm)'(c(0; \xi)) \grad c(0; \xi) = 0$. 
%	This verifies Assumption~\ref{assumption:sub-exponential-assumption}.
%	For every $x$ with $\norm{x} \le r$,  the scalar random variable 
%	$\langle a, x \rangle$ is $r$-subgaussian, and for every $v$ with $\norm{v} \le 1$, $\langle a, v\rangle$ is $1$-subgaussian. 
%	Given the property that the product of two subgaussian random variables is 
%	subexponential~\cite[Lemma 2.7.7]{Vershynin18}, we derive that for some universal constant $C > 0$,
%	$\langle a, x \rangle \langle a, v\rangle$ must be $Cr$ subexponential for every unit vector $v$ with $\norm{v}=1$. 
%	Thus, both $\grad c(x, \xi)$ and  $(\loss\sm)'(c(x; \xi))$ are $Cr$ subexponential by definition.
\item (Assumption~\ref{assumption:Lipschitz-condition}). It is easy to verify that, for $x_1 = (y_1, w_1)$ and 
	$x_2 = (y_2, w_2)$: 
	\begin{equation*}
	\begin{split} 
		\grad c(x_1, \xi) - \grad c(x_2, \xi) &= \begin{bmatrix}
			\langle v, w_1 - w_2\rangle u \\
			\langle u, y_1 - y_2 \rangle v
		\end{bmatrix} \\
		e(x_1, \xi) - e(x_2, \xi) &= \begin{bmatrix}
			\langle v, w_2 - w_1\rangle u \\
			\langle u, y_2 - y_1 \rangle v
		\end{bmatrix}
	\end{split},
	\end{equation*}
	where $e(x, \xi) = (\loss\sm)'(c(x; \xi)) \grad c(x; \xi)$. 
	%Notably, for every vector $v$, $\langle a, v\rangle$ is $\sigma \norm{v}$-subgaussian. 
	Given the property that the product of two subgaussian random variables is 
	subexponential~\cite[Lemma 2.7.7]{Vershynin18}, and that the sum of two 
	subexponential random vectors remain subexponential~\cite[Lemma 2.7.11]{Vershynin18}, 
	one can easily show that $\grad c(x_1, \xi) - \grad c(x_2, \xi)$ and $e(x_1, \xi) - e(x_2, \xi)$ 
	are $C \sigma^2 \norm{x_1 - x_2}$ subexponential vectors where $C$ is absolute. 

\item (Assumption~\ref{assumption:real-Lipschitz-condition}). It is easy to verify that 
	\begin{equation*}
	\begin{split} 
		\norm{\grad c(x_1, \xi) - \grad c(x_2, \xi)} &\le 2 \norm{u}\norm{v} \norm{x_1 - x_2} \\
		\norm{e(x_1, \xi) - e(x_2, \xi)} &\le 2 \norm{u}\norm{v} \norm{x_1 - x_2}
	\end{split},
	\end{equation*}
	where $e(x, \xi) = (h\sm)'(c(x; \xi)) \grad c(x; \xi)$. 	Since $u, v$ are independent $\sigma$-subgaussian
	random vectors, 
	$\E[L(\xi)]  = 2\E[\norm{u}\norm{v}] < \infty$.
	Thus, Assumption~\ref{assumption:real-Lipschitz-condition} is satisfied 
	with $L(\xi) = 2\norm{u}\norm{v}$.

\item (Integrability). Notably $\E[h(c(x; \xi))]  = \E[|b-\langle u, y \rangle \langle v, w \rangle|] < \infty$ for every $x = (y, w) \in \R^{d_1} \times \R^{d_2}$. 
\end{itemize}
We then compute the bound in Theorem~\ref{theorem:final-result}. Namely, we need to compute $\zeta$ and $\VC(\F)$.

For the nonsmooth function $\loss(z) = |z|$, the corresponding value of $\zeta$ is given by $\zeta = 3$ 
as $0$ is the only nondifferentiable point of $\loss$. The corresponding $\F$ is given by 
\begin{equation*}
	\F = \left\{ \{u \in \R^{d_1}, v \in \R^{d_2}, b \in \R: \langle u, y\rangle \langle v, w\rangle - b \ge t \}\mid y \in \R^{d_1}, w \in \R^{d_2}, t \in \R \right\}.
\end{equation*}
We now bound its VC dimension using Theorem~\ref{theorem: VC-bounds-on-polynomials}. 
Note $(y, w) \mapsto  \langle u, y\rangle \langle v, w\rangle - b$ 
is a degree $2$ polynomial in $(y, w)$ for every $(u, v, b)$. 
Theorem~\ref{theorem: VC-bounds-on-polynomials} implies for some universal constant $C > 0$: 
\begin{equation*}
	\VC(\F) \le C (d_1+d_2) \log (d_1+ d_2). 
\end{equation*}
Corollary~\ref{corollary:blind-deconvolution} then follows by applying Theorem~\ref{theorem:final-result}.
%To summarize, if we plug in all the estimates above into the formula of $\Delta$ in equation~\eqref{eqn:error-definition}, we obtain 
%that if we choose 
%\begin{equation*}
%\Delta_\Phi = \sqrt{\frac{1}{m} \cdot \left(d \log d\log m+ \log(\frac{1}{\delta})\right)}.
%\end{equation*} 
%then the high probability bound~\eqref{eqn:high-probability-bound-phase-retrieval} holds when we choose a large absolute constant $C$. 
%The collection of half planes $\mathcal{G} = \{\{ v \in \R^d, w\in \R: \langle v, x\rangle + w \ge 0\} \mid x \in \R^d\}$

\section{Proof of Corollary~\ref{corollary:matrix-sensing}}
\label{sec:proof-corollary-matrix-sensing}
We apply Theorem~\ref{theorem:final-result}. We first check the assumptions. 
For the nonsmooth $\loss(z) = |z|$, it is decomposed 
as $\loss = \loss \ns + \loss\sm$ where $\loss\sm(z) = -z$ and $\loss\ns(z) = 2(z)_+$.
\begin{itemize}
\item (Assumption~\ref{assumption:sub-exponential-assumption}). 
	For the matrix sensing problem, 
	\begin{equation*}
		\grad c(X, \xi) = (A+A^T)X,~~~(\loss\sm)'(c(X; \xi)) \grad c(X; \xi) = -(A+A^T)X.
	\end{equation*}
	Thus, at $X = 0 \in \R^{d \times r}$, both gradients vanish: $\grad c(0, \xi) = (\loss\sm)'(c(0; \xi)) \grad c(0; \xi) = 0$. 

%	For every $x$ with $\norm{x} \le r$,  the scalar random variable 
%	$\langle a, x \rangle$ is $r$-subgaussian, and for every $v$ with $\norm{v} \le 1$, $\langle a, v\rangle$ is $1$-subgaussian. 
%	Given the property that the product of two subgaussian random variables is 
%	subexponential~\cite[Lemma 2.7.7]{Vershynin18}, we derive that for some universal constant $C > 0$,
%	$\langle a, x \rangle \langle a, v\rangle$ must be $Cr$ subexponential for every unit vector $v$ with $\norm{v}=1$. 
%	Thus, both $\grad c(x, \xi)$ and  $(\loss\sm)'(c(x; \xi))$ are $Cr$ subexponential by definition.
\item (Assumption~\ref{assumption:Lipschitz-condition}). It is easy to verify that 
	\begin{equation*}
	\begin{split} 
		\grad c(X_1, \xi) - \grad c(X_2, \xi) &= (A+A^T)(X_1-X_2) \\
		e(X_1, \xi) - e(X_2, \xi) &= (A+A^T)(X_2-X_1)
	\end{split},
	\end{equation*}
	where $e(X, \xi) = (\loss\sm)'(c(X; \xi)) \grad c(X; \xi)$. 
	For every matrix $V \in \R^{D \times D}$, $\langle A, V\rangle$ is $\sigma \norm{V}_F$-subgaussian. 
	Given the property that the sum of two subgaussian random variables is still subgaussian~\cite[Exercise 2.5.7]{Vershynin18}, 
	we deduce that $\langle A+A^T, V\rangle$ is $C\norm{V}_F$-subgaussian for every matrix $V$, where $C > 0$ is a universal constant. As a consequence, 
	we derive that, each of 
	\begin{equation*}
	\begin{split} 
		\langle V, \grad c(X_1, \xi) - \grad c(X_2, \xi)\rangle &= \langle (A+A^T), (X_1-X_2)^T V\rangle \\
		\langle V, e(X_1, \xi) - e(X_2, \xi) \rangle &= \langle (A+A^T), (X_2-X_1)^T V\rangle
	\end{split},
	\end{equation*}
	must be $C \sigma \norm{X_1 - X_2}_F$ subgaussian, and thus $C' \sigma \norm{X_1 - X_2}_F$ subexponential
	 for every matrix $V \in \R^{D \times D}$ with $\norm{V}_F = 1$ for some universal constant $C'  > 0$. 
	 Here we use the fact that every sub-gaussian 
	distribution is subexponential~\cite[Section 2.7]{Vershynin18}.

\item (Assumption~\ref{assumption:real-Lipschitz-condition}). It is easy to verify that 
	\begin{equation*}
	\begin{split} 
		\norm{\grad c(X_1, \xi) - \grad c(X_2, \xi)}_F &= \norm{(A+A^T)(X_1-X_2)}_F \le 2 \norm{A}_F \norm{X_1-X_2}_F \\
		\norm{e(X_1, \xi) - e(X_2, \xi)}_F &= \norm{(A+A^T)(X_2-X_1)}_F \le 2 \norm{A}_F \norm{X_1-X_2}_F
	\end{split},
	\end{equation*}
	where $e(X, \xi) = (h\sm)'(c(X; \xi)) \grad c(X; \xi)$. Since $A$ is $\sigma$-subgaussian, $\E[\norm{A}_F] < \infty$.
	As a result, Assumption~\ref{assumption:real-Lipschitz-condition} is satisfied with $L(\xi) = 2\norm{A}_F$.
\item (Integrability). Notably $\E[h(c(X; \xi))] = \E[|b-\langle A, X \rangle^2|] < \infty$ for every $x \in \R^d$. 
\end{itemize}
We then compute the bound in Theorem~\ref{theorem:final-result}. Namely, we need to compute $\zeta$ and $\VC(\F)$.

For the nonsmooth function $\loss(z) = |z|$, the corresponding value of $\zeta$ is given by $\zeta = 3$
as $0$ is the only nondifferentiable point of $\loss$. The corresponding $\F$ is given by 
\begin{equation*}
	\F = \left\{ \{A \in \R^{D \times D}, b \in \R: \langle A, XX^T\rangle - b \ge t \}\mid X \in \R^{D \times r_0}, t \in \R \right\}.
\end{equation*}
Note that $X \mapsto \langle A, XX^T\rangle - b$ is a degree $2$ polynomial in $X$ for every $A, b$. 
We then bound its VC dimension using Theorem~\ref{theorem: VC-bounds-on-polynomials}, which yields that for some 
universal constant $C > 0$: 
\begin{equation*}
	\VC(\F) \le C Dr_0 \log (Dr_0). 
\end{equation*}
Corollary~\ref{corollary:matrix-sensing} then follows by applying Theorem~\ref{theorem:final-result}.

\end{appendices}


%\bibliographystyle{amsalpha}
%\bibliographystyle{amsalpha}
\bibliography{sn-bibliography}% common bib file

\end{document}